\documentclass[12pt]{amsart}
\usepackage{graphicx}
\usepackage{amsmath}
\usepackage{amsfonts}
\usepackage{amssymb}
\usepackage{setspace}
\usepackage{datetime}
\usepackage{color,enumitem,graphicx}
\usepackage[colorlinks=true,urlcolor=blue,
citecolor=red,linkcolor=blue,linktocpage,pdfpagelabels,
bookmarksnumbered,bookmarksopen]{hyperref}
\usepackage{geometry}
\allowdisplaybreaks[4]

\newtheorem{thm}{Theorem}[section]
\newtheorem{cor}[thm]{Corollary}

\newtheorem{lem}[thm]{Lemma}
\newtheorem{prop}[thm]{Proposition}
\theoremstyle{definition}
\newtheorem{defn}[thm]{Definition}
\theoremstyle{remark}
\newtheorem{rem}[thm]{Remark}
\theoremstyle{conclusion}

\theoremstyle{question}
\numberwithin{equation}{section}

\newcommand{\lr}{\left(}
\newcommand{\rr}{\right)}
\newcommand{\R}{\mathbb{R}}
\newcommand{\mms}{\mathbb{S}}

\newcommand{\N}{\mathbb{N}}
\newcommand{\md}{\mathrm{d}}
\newcommand{\Sm}{\mathbb{S}}
\newcommand{\lb}{\left\{}
\newcommand{\rb}{\right\}}

\newcommand{\be}{\begin{equation}}
	\newcommand{\ee}{\end{equation}}

\newcommand{\al}{\alpha}

\newcommand{\la}{\lambda}

\newcommand{\var}{\varepsilon}

\begin{document}

\title[Non-radial solutions for the critical quasi-linear H\'{e}non equation]{Non-radial solutions for the critical quasi-linear H\'{e}non equation involving $p$-Laplacian in $\R^N$}

\author{Wei Dai, Lixiu Duan, Changfeng Gui, Yuan Li}

\address{School of Mathematical Sciences, Beihang University (BUAA), Beijing 100191, P. R. China, and Key Laboratory of Mathematics, Informatics and Behavioral Semantics, Ministry of Education, Beijing 100191, P. R. China}
\email{weidai@buaa.edu.cn}

\address{School of Mathematical Sciences, Beihang University (BUAA), Beijing 100191, P. R. China}
\email{lixiuduan@buaa.edu.cn}

\address{Department of Mathematics, University of Macau, Macau SAR, P. R. China}
\email{changfenggui@um.edu.mo}

\address{School of Mathematical Sciences, Key Laboratory of MEA (Ministry of Education) and Shanghai Key Laboratory of PMMP, East China Normal University, Shanghai, 200241, P. R. China}
\email{liyuan5397@163.com}

\thanks{Wei Dai is supported by the NNSF of China (No. 12222102), the National Science and Technology Major Project (2022ZD0116401) and the Fundamental Research Funds for the Central Universities. Lixiu Duan is supported by the National Science and Technology Major Project (2022ZD0116401), the Fundamental Research Funds for the Central Universities and the Academic Excellence Foundation of BUAA
for PhD Students. Changfeng Gui is supported by University of Macau research grants CPG2024-00016-FST, CPG2025-00032-FST, SRG2023-00011-FST, MYRGGRG2023-00139-FST-UMDF, UMDF Professorial Fellowship of Mathematics, Macao SAR FDCT 0003/2023/RIA1 and Macao SAR FDCT 0024/2023/RIB1. Yuan Li is supported by the NSFC (No. 12401132),   China Postdoctoral Science Foundation (No. 2022M721164) and Science and Technology Commission of Shanghai Municipality (No. 22DZ2229014).}

\maketitle

\begin{abstract}
In this paper, we investigate the following $D^{1,p}$-critical quasi-linear H\'enon equation involving $p$-Laplacian
\begin{equation*}\label{00}
\left\{
\begin{aligned}
&-\Delta_p u=|x|^{\alpha}u^{p_\al^*-1}, & x\in \R^N, \\
&u>0, & x\in \R^N,
\end{aligned}
\right.
\end{equation*}
where $N\geq2$, $1<p<N$, $p_\al^*:=\frac{p(N+\al)}{N-p}$ and $\alpha>0$. By carefully studying the linearized problem and applying the approximation method and bifurcation theory, we prove that, when the parameter $\al$ takes the critical values $\al(k):=\frac{p\sqrt{(N+p-2)^2+4(k-1)(p-1)(k+N-1)}-p(N+p-2)}{2(p-1)}$ for $k\geq2$, the above quasi-linear H\'enon equation admits non-radial solutions $u$ such that $u\sim |x|^{-\frac{N-p}{p-1}}$ and $|\nabla u|\sim |x|^{-\frac{N-1}{p-1}}$ at $\infty$. One should note that, $\alpha(k)=2(k-1)$ for $k\geq2$ when $p=2$. Our results successfully extend the classical work of F. Gladiali, M. Grossi, and S. L. N. Neves in \cite{GGN} concerning the Laplace operator (i.e., the case $p=2$) to the more general setting of the nonlinear $p$-Laplace operator ($1<p<N$). We overcome a series of crucial difficulties, including the nonlinear feature of the $p$-Laplacian $\Delta_p$, the absence of Kelvin type transforms and the lack of the Green integral representation formula.
\end{abstract}


\section{Introduction}
In this paper, we investigate positive weak solution $u \in D^{1,p}(\R^N)$ to the following $D^{1,p}$-critical quasi-linear H\'enon equation involving $p$-Laplacian
\begin{equation}\label{11}
\left\{
\begin{aligned}
&-\Delta_p u=|x|^{\alpha}u^{p_\al^*-1}, & x\in \R^N, \\
&u>0, & x\in \R^N,
\end{aligned}
\right.
\end{equation}
where $\Delta_p(\cdot):=\text{div}(|\nabla(\cdot)|^{p-2}\nabla(\cdot))$ denotes the $p$-Laplace operator, $N\geq2, \ 1<p<N, \ p_\al^*=\frac{p(N+\al)}{N-p}$ and $\alpha>0$. In particular, if $\alpha=0$, \eqref{11} becomes
\begin{align}\label{criticeqution}
	-\Delta_p u = u^{p^*-1},  \quad\, u>0 \quad\,\,\,\,&\mbox{in}\,\, \R^N,
\end{align}
where $u \in D^{1,p}(\R^N)$, $1<p<N$ and $p^*:=p^{*}_0=\frac{Np}{N-p}$ is the critical Sobolev embedding exponent. We say equation \eqref{11} is $D^{1,p}(\R^{N})$-critical in the sense that both \eqref{11} and the $D^{1,p}$-norm $\| \nabla u \|_{L^p(\R^N)}$ are invariant under the scaling $u\mapsto u_{\lambda}(\cdot):=\lambda^{\frac{N-p}{p}}u(\lambda\cdot)$.

\medskip

Serrin \cite{JSLB}, along with Serrin and Zou \cite{SJZH02}, among other contributions, established the regularity results for quasi-linear equations under some assumptions and derived the sharp lower bound in the asymptotic estimate of super-$p$-harmonic functions. From Theorem 1.3 in \cite{DLL}, we know that any positive $D^{1,p}(\mathbb{R}^{N})$-weak solution $u$ to \eqref{11} satisfies $u\in C^{1,\eta}(\mathbb{R}^{N})\cap L^{\infty}(\mathbb{R}^{N})$ for some $0<\eta<\min\{1,\frac{1}{p-1}\}$ and the following sharp asymptotic estimates (see also Theorems 1.2-1.4 in \cite{CDL}):
\begin{equation}\label{eq0806}
  \frac{c_0}{1+|x|^\frac{N-p}{p-1}} \leq u(x) \leq \frac{C_0}{1+|x|^\frac{N-p}{p-1}} \qquad \mbox{in}\,\,\,\R^N,
\end{equation}
\begin{align}\label{eq0806+}
\frac{c_0}{|x|^\frac{N-1}{p-1}} \leq |\nabla u(x)| \leq \frac{C_0}{|x|^\frac{N-1}{p-1}} \qquad \mbox{in}\,\,\,\R^N \setminus B_{R_0}(0)
\end{align}
for some constants $c_0, C_0, R_0 > 0$. Guedda and Veron \cite{GV} derived the uniqueness of the radially symmetric positive solution $U(x) = U(r)$ with $r = |x|$ to equation \eqref{criticeqution}, which satisfies $U(0)=b>0$ and $U'(0)=0$. Consequently, they reduced the classification of solutions to showing the radial symmetry of solutions. For literature on the radial symmetry of positive $D^{1,p}(\mathbb{R}^{N})$-weak solutions to the $D^{1,p}(\mathbb{R}^{N})$-critical quasi-linear equation \eqref{criticeqution}, refer to \cite{CFR,LD,LDSMLMSB,DP,DLPFRM,LDMR,LDBS04,Ou,BS16,VJ16} and the references therein.

\medskip

In \cite{LPLDHT}, the authors classified positive solutions to the \(D^{1,p}(\mathbb{R}^N)\)-critical quasi-linear Hardy equation of type \eqref{11} with $\alpha<0$ via the method of moving planes. In \cite{LM}, Lin and Ma obtained the classification of positive solutions for the weighted $p$-Laplace equation of the type \eqref{11} with a singular weight on partial variables, and hence derived the best constant and extremal functions for a class of Hardy-Sobolev-Maz'ya inequalities. For more literature on Liouville results, classification results and existence of (nonradial) solutions for ($D^{1,p}$-critical weighted) semi-linear or quasi-linear elliptic equations and $N$-Laplacian Liouville equations, please c.f. \cite{CDL,CDQ0,CPY,CMR,CDQ,CL,CL2,DGL,DLL,DPQ,DQ,DQ0,DFSV,DT,DSPV,E,GY00,GHM,GM,LPLDHT,LM,OSV,Ou,PS,PT,SJZH02,SSW,Z} and the references therein.

\medskip

In the special case $p=2$, the equation \eqref{11} reduces to the classical second order H\'{e}non equation with regular Laplace operator, i.e.,
\begin{equation}\label{12}
\left\{
\begin{aligned}
&-\Delta u=C(\al)|x|^\al u^{p_\al}, & x\in \R^N, \\
&u>0, & x\in \R^N,
\end{aligned}
\right.
\end{equation}
where $N\geq3$, $p_\al=\frac{N+2+2\al}{N-2}$ and $\alpha>0$. In \cite{GGN}, when $\alpha=\alpha(k)=2(k-1)$ with $k\geq2$ is an even integer, F. Gladiali, M. Grossi and S. L. N. Neves proved the existence of nonradial solutions to equation \eqref{12} using bifurcation theory. Specially, when $\al=2$ and $N\geq4$ is even, they obtained that, for any $a\in\R$, the functions
$$u(x)=u(|x'|,|x''|)=\frac{1}{\lr1+|x|^4-2a(|x'|^2-|x''|^2)+a^2\rr^{\frac{N-2}{4}}}$$
form a branch of nonradial solutions to \eqref{12} bifurcating from the radial solution  $U_2:=\frac{1}{(1+|x|^4)^{\frac{N-2}{4}}}$, where $x\in\R^{\frac{N}{2}}\times\R^{\frac{N}{2}}$, $x=(x',x'')$ with $x'\in\R^{\frac{N}{2}}$ and $x''\in\R^{\frac N2}$.

\medskip

For general $1<p<N$, problem \eqref{11} admits an explicit solution given by
\begin{align}\label{14}
U_{\la,\al}(x)=\frac{C_{N,p,\alpha}\la^{\frac{N-p}{p}}}
{(1+\la^{\frac{p+\al}{p-1}}|x|^{\frac{p+\alpha}{p-1}
		})^{\frac{N-p}{p+\al}}},\quad\,\, x\in\R^N,\end{align}
where \begin{align}\label{15}C_{N,p,\alpha}=\lr(N+\al)\lr\frac{N-p}{p-1}
\rr^{p-1}\rr^{\frac{N-p}{p(p+\alpha)}}.\end{align}
In particular, setting $\la=1$ yields
\begin{equation}\label{e1}
  U_\al:=U_{1,\al}=\frac{C_{N,p,\alpha}}
{(1+|x|^{\frac{p+\alpha}{p-1}
		})^{\frac{N-p}{p+\al}}}.
\end{equation}
For $\al>0$, we have the following result on uniqueness of radially symmetric solution to \eqref{11}.
\begin{thm}\label{thm11}
Suppose $N\geq 2$ and $\al>0$. Then problem \eqref{11} in $D^{1,p}(\R^N)$ has a unique (up to scalings)
radial nontrivial solution of the form $ U_{\lambda,\alpha}(x)$ for some $\lambda>0$, where $U_{\lambda,\alpha}(x)$ defined in \eqref{14}.
\end{thm}

Being essentially different from the cases $\al\leq0$ for equation \eqref{11}, we can not prove the radial symmetry of positive solution to equation \eqref{11} with $\al>0$ via the method of moving planes or other methods. Nonradial solutions to \eqref{11} may exist. In this paper, by carefully studying the linearized problem and applying the approximation method and bifurcation theory, we aim to show the existence of nonradial solutions to \eqref{11} and extend the results in F. Gladiali, M. Grossi, and S. L. N. Neves in \cite{GGN} from $p=2$ to the more general cases $1<p<N$. We will overcome a series of crucial difficulties, including the nonlinear feature of the $p$-Laplacian $\Delta_p$, the absence of Kelvin type transforms and the lack of Green representation formula etc.

\medskip

To this end, we first study the linearized problem of equation \eqref{11} at the radial solution $U_\al$, i.e.,
\begin{equation}\label{16}
	\left\{
	\begin{aligned}
		&-\textrm{div}(|\nabla U_\al|^{p-2}\nabla v)-(p-2)
		\textrm{div}
		\lr|\nabla U_\al|^{p-4}(\nabla U_\al\cdot\nabla v)\nabla U_\al\rr=(p_\al^*-1)|x|^{\alpha}U_\al^{p_\al^*-2}v, \\
		&v\in D^{1,p}(\R^N).
	\end{aligned}
	\right.
\end{equation}
Specifically, when $\alpha = 0$, Pistoia and Vaira \cite{PV} proved the non-degeneracy of the solution $U_{0}:=U_{1,0}$ of equation \eqref{11}. In particular, all solutions $v\in\mathcal{D}^{1,2}_{0,*}(\mathbb{R}^N)$ of the equation
\begin{align}\label{Ppwhlp}
	& -{\rm div}(|\nabla U_{0}|^{p - 2}\nabla v)-(p - 2){\rm div}(|\nabla U_{0}|^{p - 4}(\nabla U_{0}\cdot\nabla v)\nabla U_{0})=\left(p^*-1\right)U_{0}^{p^*-2}v
\end{align}
in $\mathbb{R}^N$ are linear combinations of the functions
\begin{equation*}
	Z_0(x)=\frac{N - p}{p}U_{0}+x\cdot\nabla U_{0},\qquad Z_i(x)=\frac{\partial U_{0}(x)}{\partial x_i},\quad i = 1,\ldots,N.
\end{equation*}
Here, $\mathcal{D}^{1,2}_{0,*}(\mathbb{R}^N)$ is a weighted Sobolev space defined by the completion of $C^\infty_c(\mathbb{R}^N)$ with respect to the norm
\begin{align}\label{defd120*}
	\|v\|_{\mathcal{D}^{1,2}_{0,*}(\mathbb{R}^N)}:=\left(\int_{\mathbb{R}^N}|\nabla U_{0}|^{p - 2}|\nabla v|^2\mathrm{d}x\right)^{\frac{1}{2}}.
\end{align}
Pistoia and Vaira \cite{PV} derived the above conclusion by proving the continuous embedding $\mathcal{D}^{1,2}_{0,*}(\mathbb{R}^N)\hookrightarrow L^{2}_{0,*}(\mathbb{R}^N)$, where $L^{2}_{0,*}(\mathbb{R}^N)$ consists of measurable functions $v:\mathbb{R}^N\to\mathbb{R}$ with the norm
\[
\|v\|_{L^{2}_{0,*}(\mathbb{R}^N)}:=\left(\int_{\mathbb{R}^N}U_{0}^{p^*-2}v^2\mathrm{d}x\right)^{\frac{1}{2}}<+\infty.
\]
Subsequently, Figalli and Neumayer \cite{FN} established that $\mathcal{D}^{1,2}_{0,*}(\mathbb{R}^N)\hookrightarrow\hookrightarrow L^{2}_{0,*}(\mathbb{R}^N)$ compactly when $2\leq p<N$. Based on this, they showed that the solutions of \eqref{Ppwhlp} in $L^{2}_{0,*}(\mathbb{R}^N)$ are linear combinations of the functions $Z_0$ and $Z_i$ ($i = 1,\ldots,N$). For $1 < p<N$, Figalli and Zhang \cite{FZ} further proved that $\mathcal{D}^{1,2}_{0,*}(\mathbb{R}^N)\hookrightarrow\hookrightarrow L^{2}_{0,*}(\mathbb{R}^N)$ compactly and the non-degeneracy conclusion in \cite{FN} still holds.

\medskip

For $\al>0$, we prove that the linearized problem \eqref{16} is non-degenerate.
\begin{thm}\label{th11}
Let $\alpha\geq0$. If $\alpha>0$ and $\alpha\ne\al(k)$, then the space of solutions of \eqref{16} has dimension $1$ and is spanned by
\begin{align}\label{17}
Z(x)=\frac{(p-1)-|x|^{\frac{p+\al}{p-1}}}
{(1+|x|^{\frac{p+\al}{p-1}})^{\frac{N+\al}{p+\al}}}
.\end{align}
If $\alpha=\al(k)$ for some $k\in \mathbb{N}$, then the spaces of solutions of \eqref{16} has dimension $1+\frac{(N+2k-2)(N+k-3)!}{(N-2)!k!}$ and is spanned by
\begin{align}\label{18}
	Z(x)=\frac{(p-1)-|x|^{\frac{p+\al}{p-1}}}
	{(1+|x|^{\frac{p+\al}{p-1}})^{\frac{N+\al}{p+\al}}},\quad\quad
	Z_{k,i}(x)=\frac{|x|^{\frac{p+\alpha}{p(p-1)}}\Phi_{k,i}(x)}
	{(1+|x|^{\frac{p+\al}{p-1}})^{\frac{N+\al}{p+\al}}},
\end{align}
where $$\al(k):=\frac{p\sqrt{(N+p-2)^2+4(k-1)(p-1)(k+N-1)}-p(N+p-2)}{2(p-1)}$$
for $k\in \mathbb{N}$ and $\{\Phi_{k,i}\}$ ($i = 1,\cdots$, $\frac{(N+2k-2)(N+k-3)!}{(N-2)!k!}$) form a basis of the space $\boldsymbol{\Phi}_k(\R^N)$ consisting of all homogeneous harmonic polynomials of degree $k$ in $\R^N$.
\end{thm}

\begin{rem}
In Theorem \ref{th11}, if $\al=0$, one has $k=1$, then our non-degeneracy result is consistent with the non-degeneracy result in Pistoia and Vaira \cite{PV}, Figalli and Neumayer \cite{FN}, and Figalli and Zhang \cite{FZ} for $\al=0$. For all $\al>0$, the problem \eqref{11} is invariant under the scaling $u\mapsto u_{\lambda}(\cdot):=\lambda^{\frac{N-p}{p}}u(\lambda\cdot)$, however, it does not possess translation invariance. Theorem \ref{th11} indicates that, if $\al=\al(k)>0$, then there exist new solutions to \eqref{16}, which supersede the solutions that would arise from translation invariance.
\end{rem}

As a corollary of Theorem \ref{th11}, we can compute the Morse index at the solution $U_\al$.
\begin{cor}\label{c12}
Let $U_\al$ be the radial solution of \eqref{11} given by \eqref{e1}, then its Morse index $m(\alpha)$ is equal to
$$m(\alpha)=\sum_{0\leq k<\zeta_k,\ k\in \N}\frac{(N+2k-2)(N+k-3)!}{(N-2)!k!},$$
where $\zeta_k=\frac{(2p-Np)+\sqrt{4(p-1)\al^2+4p(N+p-2)\al+N^2p^2}}{2p}$. In particular, we have that the Morse index of $U_\al$ varies as $\alpha$ cross $\al(k)$ ($\forall \, k\geq2$), and $m(\alpha)\to+\infty$ as $\alpha\to+\infty$.
\end{cor}

Next, by applying Theorem \ref{th11}, we will show the existence of nonradial solutions to equation \eqref{11}. To this end, we introduce some essential notations and definitions.

\medskip

Define the weighted norm, for every $g\in L^\infty(\R^N)$,
\begin{align}\label{d18}
	\|g\|_\gamma:=\sup_{x\in\R^N}(1+|x|)^\gamma|g(x)|,
\end{align}
where $\gamma\in\left(\frac{N(N-p)}{Np-(N-p)},\frac{N-p}{p-1}\right)$. Denote
  $L^\infty_\gamma(\R^N):=\{g\in L^\infty(\R^N)\
\text{such\ that} \ \exists C>0\ \text{and}\ \|g\|_\gamma<C\} $.

\medskip

Set
\begin{align}\label{19}
	X=D^{1,p}(\R^N)\bigcap L^\infty_\gamma(\R^N),
\end{align}
$X$ is a Banach space with the norm
\begin{align}\label{110}
	\|g\|_X:=\max\{\|g\|_{1,p},\|g\|_\gamma\}
,\end{align}
where $\|\cdot\|_{1,p}$ denotes the usual norm in $D^{1,p}(\R^N)$, i.e., $\|g\|_{1,p}=\left(\int_{\R^N}|\nabla g|^p\md x\right)^{\frac{1}{p}}$ for $g\in D^{1,p}(\R^N)$.

\begin{defn}
Let $U_\al$ be the radial solution of \eqref{11} defined in \eqref{14}. Denote a nonradial bifurcation occurs at $(\bar\al,U_{\bar\al})$, if every neighborhood of $(\bar\al,U_{\bar\al})$ in $(0,+\infty)\times X$, there exists a point $(\al,v_\al)$ where the $v_\al$ is the nonradial solution of \eqref{11}.
\end{defn}

Let $\mathcal{O}(k)$ be the orthogonal group in $\R^k$. Our main result is the following theorem.
\begin{thm}\label{th14}
Let $\al=\al(k)$ with $k\in\N, k\geq2$. Then \\
(i) there exists at least a continuum of nonradial solutions to \eqref{11}, invariant with respect to $\mathcal{O}(N-1)$, bifurcating from the pair $(\al,U_\al)$; \\
(ii) if $k$ is even, there exist at least $[\frac{N}{2}]$ continua of nonradial solutions to \eqref{11} bifurcating from $(\al,U_\al)$. The $l$-th branch is invariant with respect to $\mathcal{O}(N-l)\times \mathcal{O}(l)$ for $l=1,\cdots,[\frac{N}{2}]$.\\
Furthermore, all these nonradial solutions $h$ derived in (i) and (ii) satisfy $h\sim |x|^{-\frac{N-p}{p-1}}$ and $|\nabla h|\sim |x|^{-\frac{N-1}{p-1}}$, as $|x|\rightarrow+\infty$.
\end{thm}

\begin{rem}
Theorem \ref{th14} indicates that, when $\al>0$, the structure of solutions to equation \eqref{11} becomes significantly more complex compared to the case $\al=0$, particularly highlighting the emergence of nonradial solutions when $\al=\al(k)$. Notably, when $p=2$, $\al(k)=2(k-1)$($k\in\N$) are even integers, our results in Theorems \ref{th11} and \ref{th14}, and Corollary \ref{c12} are consistent with those results in Theorems 1.3, 1.6 and Corollary 1.4 of \cite{GGN}, respectively. Therefore, our results successfully extend the classical work of F. Gladiali, M. Grossi, and S. L. N. Neves in \cite{GGN} concerning the Laplace operator (i.e., the special case $p=2$) to the more general setting of the nonlinear $p$-Laplace operator ($1<p<N$).
\end{rem}

\begin{rem}
In the special case $p=2$, F. Gladiali, M. Grossi, and S. L. N. Neves proved in Theorem 1.6 in \cite{GGN} that, when $\al=\al(k)=2(k-1)$, all these nonradial solutions $h$ satisfy $\limsup\limits_{|x|\rightarrow+\infty}|x|^{N-2}h(x)<+\infty$. In our Theorem \ref{th14}, for general $1<p<N$, we can obtain the precise asymptotic estimate of all these nonradial solutions $h$ derived in (i) and (ii), i.e., $h\sim |x|^{-\frac{N-p}{p-1}}$ and $|\nabla h|\sim |x|^{-\frac{N-1}{p-1}}$, as $|x|\rightarrow+\infty$, which follows from Theorem 1.3 in \cite{DLL} (see \eqref{eq0806} and \eqref{eq0806+}).
\end{rem}

From Theorem \ref{th11}, we find that the linearized equation \eqref{16} has a radial solution $Z(x)$ for any $\al\in(0,+\infty)$, which is given by \eqref{17}. In particular, if $\alpha=\al(k)$ ($k\in \N$), the kernel of the linearized operator is generated by the radial solutions $Z(x)$ and $Z_{k,i}(x)$ given by \eqref{18}. This implies that, in the whole space $\mathbb{R}^{N}$, the radial solutions $U_\al$ are degenerate in the space of radial functions. Therefore, we cannot directly apply the classical bifurcation theory to obtain the existence of nonradial bifurcation points. Instead, we will investigate an approximate problem \eqref{31} in balls $B_{\frac{1}{\var}}(0)$ with radius $\frac{1}{\var}$. One should note that, the radial approximate solution $u_{\var,\alpha}$ to \eqref{31} is nondegenerate in the space of radial functions (see Lemma \ref{lem31}), and the operator $I - \mathcal{L}_{h}^n(\al,u_{n,\al}):(0,+\infty)\times\mathcal{Z}_n\to \mathcal{Z}_n$ is invertible for $\al \neq\al_k^n$, $k=1,2,\cdots$ (see Theorem \ref{th38}). Therefore, by applying the classical bifurcation theory in the balls $B_{\frac{1}{\var}}(0)$, we can obtain the existence of nonradial solutions that bifurcate from some radial approximate solutions close to $U_{\al}$. Finally, we will prove that these nonradial solutions will converge, in a certain sense, to the nonradial solutions of problem \eqref{11} in $\mathbb{R}^{N}$.

\medskip

We would like to mention some of the key difficulties in our proof, comparing with the special semi-linear case $p=2$ studied by F. Gladiali, M. Grossi and S. L. N. Neves \cite{GGN}. The main difficulties lie on the nonlinear virtue of the $p$-Laplacian $\Delta_p$, the unavailability of Kelvin type transforms and absence of the Green function representation formula, see e.g. Section 4.

\smallskip

In order to overcome the difficulties caused by the absence of the Green function representation formula, by comparing with the radial function $\Phi_n$ and maximum principle, we first proved the uniformly fast decay estimate for approximate solutions $v_n$ on large ball $B_{\frac{1}{\var_n}}$ in Proposition \ref{pp41} without using the Green function representation formula, and derived the uniformly fast decay estimate for $|\nabla v_n|$ via rescaling arguments in Proposition \ref{pro43}. Then, based on \eqref{aa438}, we can establish the uniformly fast decay estimate \eqref{412} for $w_n$ in Proposition \ref{ppq43} via a De Giorgi-Moser-Nash iteration argument, without using the Green function representation formula, where $w_n:=\frac{u_n-v_n}{\|u_n-v_n\|_\infty}$, $u_n$ and $v_n$ are approximate solutions on large ball $B_{\frac{1}{\var_n}}$. Moreover, in Proposition \ref{pp42}, by using the Pohozaev identity and maximum principle, and comparing with the radial function $\Psi_n$ and $|\nabla \Psi_n|$, we were also able to prove $\la=1$ without resorting the Green function representation formula. By using the Pohozaev identity on uniform bounded domains $\Omega_n$ such that $\Omega_n\rightrightarrows B_r(0)$ as $n\rightarrow+\infty$ with $r=(p-1)^{\frac{p-1}{p+\al}}$, we finally proved the crucial uniform lower bound $\|u_n-v_n\|_\infty\geq C_0$ in Proposition \ref{pp44} without using the Green function representation formula.

\smallskip

In order to overcome the difficulties caused by the unavailability of Kelvin type transforms, based on \eqref{aa438}, through a De Giorgi-Moser-Nash iteration argument, we can finally establish the uniformly fast decay estimate \eqref{412} for $w_n$ in Proposition \ref{ppq43} without using Kelvin type transforms.

\smallskip

In order to overcome the difficulties caused by the nonlinear virtue of the $p$-Laplacian $\Delta_p$, by proving a weighted Sobolev inequality \eqref{aa436} in Proposition \ref{pp43}, we overcame the nonlinear virtue of the $p$-Laplacian $\Delta_p$ and first proved a uniformly decay estimate in integral form \eqref{aa438} for $w_n$ in Proposition \ref{ppp43}, where $w_n:=\frac{u_n-v_n}{\|u_n-v_n\|_\infty}$. Finally, by considering the problem satisfied by $w_n$ directly, using the Pohozaev identity on uniform bounded domains $\Omega_n$ and the nontrivial fast decay estimate \eqref{412}, we successfully overcame the nonlinear virtue of the $p$-Laplacian $\Delta_p$ and proved the uniform lower bound $\|u_n-v_n\|_\infty\geq C_0$ in Proposition \ref{pp44}, which indicates that the limit of approximate solutions is non-radial solution and plays a quite crucial role in our proof of global bifurcation result in Section 5.

\medskip

The rest of our paper is organized as follows. Theorems \ref{thm11}, \ref{th11} and Corollary \ref{c12} are proved in Section 2. In Section 3, we investigate the approximate problem in the ball and show the bifurcation results of approximate solutions. In Section 4, we establish some crucial estimates for approximate solutions (see e.g. Proposition \ref{pp44}), which indicates that the limit of approximate solutions is non-radial solution. Finally, in Section 5, we prove our main result, i.e., Theorem \ref{th14}.

\vspace{0.3cm}

\noindent{\bfseries Notations.}
Throughout this paper, $B_R:=B_R(0)$ denotes the ball with radius $R$ centered at the origin. Moreover, $c$, $C$, $C'$ and $C_i$ are used to denote various absolutely positive constants whose values may differ from line to line. The notation $a\sim b$ means that $C'b\leq a\leq Cb$.

\bigskip

\section{Non-degeneracy of the linearized operator}
\begin{proof}[Proof of Theorem \ref{thm11}.]
Considering the radial solution $u$ to PDE \eqref{11}, we obtained the following ODE
	\begin{equation}\label{21q}
		\left\{
		\begin{aligned}
			&-(p-1)\lr r^{N-1}|u'(r)'|^{p-2}u'\rr'=(p_\al^*-1)r^{N+\alpha-1}u^{p_\al^*-1}   &\text{in}\ \lr0,+\infty\rr, \\
			&u'(0)=0,\quad u(0)=1, \quad u>0.
		\end{aligned}
		\right.
	\end{equation}
We will use contradiction argument. Assume that \(v\) is also a radial solution of \eqref{11} satisfying ODE \eqref{21q} and \(v\neq u\). Integrating both sides of the equation \eqref{21q} from $0$ to $r$, we can obtain
\begin{align}\label{22q}
	-u'(r)=\lr\frac{p_\al^*-1}{(p-1)r^{N-1}}\int_0^rt^{N+\al-1}u^{p_\al^*-1}\md t\rr^{\frac{1}{p-1}},
\end{align}
and
\begin{align}\label{23q}
	-v'(r)=\lr\frac{p_\al^*-1}{(p-1)r^{N-1}}\int_0^rt^{N+\al-1}v^{p_\al^*-1}\md t\rr^{\frac{1}{p-1}}.
\end{align}
Define
\begin{align}\label{24q}
r_0:=\sup\{r\geq0|u(s)=v(s),\ \ \forall s\leq r\}.
\end{align}
From $u(0)=v(0)$ and $u\ne v$, we know $0\leq r_0<+\infty$. Without loss of generality, we may assume that, there exists $\var>0$ small enough, such that \(u(r) > v(r)\) and $u'(r)>v'(r)$ for any $r\in(r_0,r_0+\var)$. By subtracting equation \eqref{22q} from equation \eqref{23q}, we can get, for any $r\in(r_0,r_0+\var)$,
\begin{align}\label{25q}
	&\quad 0>-(u'(r)-v'(r))=\\
&\lr\frac{1}{r^{N-1}}\int_0^{r}t^{N+\al-1}u^{p_\al^*-1}\md t\rr^{\frac{1}{p-1}}-\lr\frac{1}{r^{N-1}}\int_0^{r}t^{N+\al-1}v^{p_\al^*-1}\md t\rr^{\frac{1}{p-1}}>0,\nonumber
\end{align}
which is absurd.

Note that $U_{\alpha}(x)$ is a radial solution of the equation \eqref{11}. Therefore, up to scalings, the radial solution to equation \eqref{11} is unique and has the form of $U_{\la,\al}$.  This finishes the proof of Theorem \ref{thm11}.
\end{proof}

In the rest of this section, we investigate the non-degeneracy of the linearized problem \eqref{16} and carry out the proof of Theorem \ref{th11} and Corollary \ref{c12}. To this end, let us rewrite the linearized equation \eqref{16} as
\begin{align}\label{21pp}
	&\quad |x|^2 \Delta v+(p - 2) \sum_{i,j=1}^{N} \frac{\partial^2 v}{\partial x_i \partial x_j} x_i x_j +\frac{ (p - 2)(N +\al)}{p-1} \frac{1}
	{1+|x|^{\frac{p+\al}{p-1}}}(x \cdot\nabla v) \\
	&+\frac{(N+\al)(Np+p\al-N+p)}{p-1}\frac{|x|^{\frac{p+\al}{p-1}}}
	{(1+|x|^{\frac{p+\al}{p-1}})^2} v =0\nonumber.
\end{align}

\begin{proof}[Proof of Theorem \ref{th11}.]
We perform a standard spherical harmonic decomposition on $v$, i.e.,
$$v=v(r,\theta)=\sum_{k=0}^{\infty}\varphi_k(r)\Phi_k(\theta),$$
where $r=|x|, \theta=\frac{x}{|x|}\in \mms^{N-1}$ and
$$\varphi_k(r)=\int_{\mms^{N-1}}v(r,\theta)\Phi_k(\theta)\md\theta.$$
Here $\Phi_k(\theta)$ denotes the $k$-th spherical harmonic, i.e., it satisfies
\begin{align}\label{21}
	-\Delta_{\mms^{N-1}}\Phi_k=\lambda_k\Phi_k,
\end{align}
where $\Delta_{\mms^{N-1}}$ is the Laplace-Beltrami operator on $\mms^{N-1}$ with respect to the standard metric $g_{\mms^{N-1}}$ and $\lambda_k$ is the $k$-th eigenvalue of $\Delta_{\mms^{N-1}}$. It is well known that
$$\lambda_k=k(N-2+k),\quad \ k=0,1,2,\cdots,$$
whose multiplicity is
$$\frac{(N+2k-2)(N+k-3)!}{(N-2)!k!}$$
and
$$\textrm{Ker}(\Delta_{\mms^{N-1}}+\la_k)
=\boldsymbol{\Phi}_k(\R^N)|_{\mms^{N-1}},$$
where $\boldsymbol{\Phi}_k(\R^N)$ is the space of all homogeneous harmonic polynomials of degree $k$ in $\R^N$. The first eigenvalue $\la_0=0$ and the corresponding eigen-function of \eqref{21} is well known as constant function and the eigen-functions corresponding to the second eigenvalue $\la_1=N-1$ are $\frac{x_i}{x},\ i=1,\cdots,N$.

Through direct calculations, we know that
\begin{equation}\label{22p}
	\Delta(\varphi_k(r)\Phi_k(\theta)) = \Phi_k(\theta) \lr \varphi_k'' + \frac{N-1}{r}\varphi_k'\rr + \frac{1}{r^2}\varphi_k(r)\Delta_{\mms^{N-1}} \Phi_k(\theta).
\end{equation}
Since
\begin{align*}
	\frac{\partial(\varphi_k(r)\Phi_k(\theta))}{\partial {x_i}}= \varphi_k'(r)\frac{x_i}{r}\Phi_k(\theta) + \varphi_k(r)\sum_{h=1}^{N-1}\frac{\partial \Phi_k}{\partial\theta_h}\frac{\partial\theta_h}{\partial x_i}, \end{align*}
we have
	\begin{align}\label{23p}
		(\nabla(\varphi_k(r)\Phi_k(\theta))\cdot x) &= \sum_{i=1}^{N} x_i\frac{\partial(\varphi_k(r)\Phi_k(\theta))}{\partial x_i}\\
		& = \varphi_k'(r)r\Phi_k(\theta) + \varphi_k(r)\sum_{i=1}^{N}\sum_{h=1}^{N-1}\frac{\partial \Phi_k}{\partial\theta_h}\frac{\partial\theta_h}{\partial x_i}x_i \nonumber\\&= \varphi_k'(r)r\Phi_k(\theta).\nonumber
	\end{align}
From
\begin{align*}
	\frac{\partial^2(\varphi_k(r)\Phi_k(\theta))}{\partial x_i\partial x_j} &= \varphi_k''(r)\frac{x_ix_j}{r^2}\Phi_k(\theta) + \varphi_k'(r)\lr \frac{\delta_{ij}}{r} - \frac{x_ix_j}{r^3} \rr\Phi_k(\theta) + \varphi_k'(r)\frac{x_i}{r}\sum_{h=1}^{N-1}\frac{\partial \Phi_k}{\partial\theta_h}\frac{\partial\theta_h}{\partial x_j} \\
	&\quad + \varphi_k'(r)\frac{x_j}{r}\frac{\partial \Phi_k}{\partial\theta_h}\frac{\partial\theta_h}{\partial x_i} + \varphi_k(r)\frac{\partial^2 \Phi_k}{\partial\theta_h\partial\theta_{\ell}}\frac{\partial\theta_{\ell}}{\partial x_j}\frac{\partial\theta_h}{\partial x_i} + \varphi_k(r)\sum_{h=1}^{N-1}\frac{\partial \Phi_k}{\partial\theta_h}\frac{\partial^2\theta_h}{\partial x_i\partial x_j},
\end{align*}
it follows that
\begin{align}
	\label{24p}&\quad\sum_{i,j=1}^{N} \frac{\partial^2(\varphi_k(r)\Phi_k(\theta))}{\partial x_i\partial x_j} x_ix_j \\ &= \varphi_k''(r)r^2\Phi_k(\theta) + 2\varphi_k'(r)r\sum_{i=1}^{N}\sum_{h=1}^{N-1}\frac{\partial \Phi_k}{\partial\theta_h}\frac{\partial\theta_h}{\partial x_i}x_i \nonumber \\
	&\quad + \varphi_k(r)\sum_{i,j=1}^{N}\sum_{h=1}^{N-1}\frac{\partial^2 \Phi_k}{\partial\theta_h\partial\theta_{\ell}}\frac{\partial\theta_{\ell}}{\partial x_j}\frac{\partial\theta_h}{\partial x_i}x_ix_j + \varphi_k(r)\sum_{i,j=1}^{N}\sum_{h=1}^{N-1}\frac{\partial \Phi_k}{\partial\theta_h}\frac{\partial^2\theta_h}{\partial x_i\partial x_j}x_ix_j \nonumber\\
	&= \varphi_k''(r)r^2\Phi_k(\theta)\nonumber,
\end{align}
where we have used the fact that
\begin{align*}
	\sum_{i=1}^{N}\frac{\partial\theta_h}{\partial x_i}x_i = 0, \quad\quad
	\sum_{i,j=1}^{N}\frac{\partial^2\theta_h}{\partial x_i\partial x_j}x_ix_j = 0, \quad h = 1,\cdots,N-1.
\end{align*}
Putting \eqref{21}, \eqref{22p}, \eqref{23p} and \eqref{24p} into \eqref{21pp}, we deduce that the function $v$ is a solution of \eqref{21pp} if and only if $\varphi_k(r)$ is a
classical solution of the ODE
\begin{equation}\label{22}
	\left\{
	\begin{aligned}
		&\quad-(p-1)\varphi_k''(r)-\frac{\varphi_k'(r)}{r}\lr(N-1)+
		\frac{(p-2)(N+\al)}{1+r^{\frac{p+\al}{p-1}}}\rr+
		\la_k\frac{\varphi_k(r)}{r^2}\\
		&=\frac{(N+\al)(Np+p\al-N+p)}{p-1}\frac{r^{\frac{p+\al}{p-1}-2}}
{(1+r^{\frac{p+\al}{p-1}})^2}\varphi_k(r), \quad  r\in(0,\infty),\,\, \varphi_k\in\mathcal{B}, \\
		&\varphi_k'(0)=0\ \ \text{if}\ \ k=0\ \ \text{and}
		\ \ \varphi_k(0)=0\ \ \text{if} \ \ k\geq1,
	\end{aligned}
	\right.
\end{equation}
where the function space \begin{equation}\label{23b}\mathcal{B}:=\lb\varphi\in C([0,\infty))\Big|\int_0^\infty
|\varphi(r)|^2U_\al^{p_\al^*-2}r^{N+\al-1}\md r<+\infty\rb.\end{equation}
Because $v(r,\theta)=\sum_{k=0}^{\infty}\varphi_k(r)\Phi_k(\theta)$ is a smooth solution of \eqref{21pp}, we must have that
$\varphi_0'(0)=0$ and $\varphi_k(0)=0$ for $k\geq1$.

By changing of variable $s=r^{\frac{p}{p+\alpha}}$ and letting
\begin{equation}\label{e23}
	\eta_k(s)=\varphi_k(r^{\frac{p}{p+\alpha}}),\end{equation}
we rewrite \eqref{22} as
\begin{equation}\label{23}
	\left\{
	\begin{aligned}
		&\quad-\eta_k''(s)-\frac{\eta_k'(s)}{s}\lr\frac{M-1}{p-1}+
		\frac{(p-2)M}{(p-1)(1+s^{\frac{p}{p-1}})}\rr+\frac{q^2\la_k}{(p-1)}\frac{\eta_k(s)}{s^2}\\
		&=
		\frac{M(Mp-M+p)}{(p-1)^2}\frac{s^{\frac{p}{p-1}-2}}{(1+s^{\frac{p}{p-1}})^2}\eta_k(s), \quad  s\in(0,\infty),\,\, \eta_k\in\mathcal{\widetilde{B}},\\
		&\eta_k'(0)=0\ \ \text{if}\ \ k=0\ \ \text{and}
		\ \ \eta_k(0)=0\ \ \text{if} \ \ k\geq1,
	\end{aligned}
	\right.
\end{equation}
where
$$\mathcal{\widetilde{B}}:=\lb\eta\in C([0,\infty))\Big|\int_0^\infty
s^{M-1}|\eta(s)|^2|W|^{\frac{Mp}{M-p}-2}\md s<+\infty\rb,\quad W(s)=U(r),$$
\begin{align}\label{d14}
	U(r)={C_{M,p,\alpha}}
	{(1+r^{\frac{p}{p-1}
		})^{-\frac{M-p}{p}}},\quad x\in\R^M,\end{align}
and
$$M=\frac{p(N+\al)}{p+\al}>p\quad \text{and}\quad q=\frac{p}{p+\al}.$$
Fix $M$, let us consider the following eigenvalue problem
\begin{align}\label{qq210}
	&\quad-\psi''(s)-\frac{\psi'(s)}{s}\lr\frac{M-1}{p-1}+
	\frac{(p-2)M}{(p-1)(1+s^{\frac{p}{p-1}})}\rr+\frac{\gamma}{(p-1)}\frac{\psi(s)}{s^2}\\
	&=
	\frac{M(Mp-M+p)}{(p-1)^2}\frac{s^{\frac{p}{p-1}-2}}{(1+s^{\frac{p}{p-1}})^2}\psi(s).\nonumber
\end{align}
Equation \eqref{qq210} can be written in weak form as
	\begin{align}\label{d212}
	\mathcal{L}(\psi)=0,
	\end{align}
where the operator $\mathcal{L}$ is defined by
\begin{align*}
\mathcal{L}(\psi):=(s^{N-1}|U'(s)|^{p-2}\psi')'
+\frac{1}{p-1}\lr\frac{Mp}{M-p}-1\rr s^{N-1}(U(s))^{p^*-2}\psi-\frac{\gamma}{p-1}s^{N-3}|U'(s)|^{p-2}\psi.
\end{align*}
We will show existence of solutions $\psi$ to \eqref{qq210} (or equivalently, \eqref{d212}) in $\mathcal{D}$, where $\mathcal{D}$ is the completion of $C_c^1([0,+\infty))$ with respect to the norm
\begin{align*}
\|\psi\|:=\lr\int_0^{+\infty}s^{N-1}|U'(s)|^{p-2}|\psi'(s)|^2\md s
	+\frac{\gamma}{p-1}\int_0^{+\infty}s^{N-3}|U'(s)|^{p-2}|\psi(s)|^2\md s\rr^{\frac12}.
\end{align*}
The proof will be carried out by discussing the following two different cases.

\medskip

\textbf{Case 1.} $\gamma=\gamma_1:=M-1$. In such case, we have
\begin{align}\label{d214}
\psi_1(s):=\frac{s^{\frac{1}{p-1}}}
	{(1+s^{\frac{p}{p-1}})^{\frac{N+\alpha}{p+\alpha}}}
\end{align}
is a solution to equation \eqref{d212}. We claim that the solutions to the equation $\mathcal{L}(\psi)=0$ in  $\mathcal{D}$ can be characterized by  $\psi=c\psi_1$, where $c\in\mathbb{R}$. A straightforward computation verifies that $\psi_1\in\mathcal{D}$ and fulfills the equation $\mathcal{L}(\psi_1)=0$. Assume $h$ is a second solution to equation \eqref{d212} that is linearly independent with respect to $\psi_1$ and possesses the form
$$h(s)=c(s)\psi_1(s).$$
Then we obtain
$$c''(s)\psi_1(s)+c'(s)\lr2\psi_1'(s)+\frac{\psi_1(s)}{s}\lr\frac{N-1}{p-1}+\frac{M(p-2)}{p-1}\frac{1}{1+s^{\frac{p}{p-1}}}\rr\rr=0,$$
hence
$$\frac{c''(s)}{c'(s)}=-2\frac{\psi_1'(s)}{\psi_1(s)}-\frac{1}{s}\lr\frac{N-1}{p-1}+\frac{M(p-2)}{p-1}\frac{1}{1+s^{\frac{p}{p-1}}}\rr.$$
Direct computation implies that
$$
c'(s)=C\frac{(1+s^{\frac{p}{p-1}})^{\frac{(N+\alpha)(p-2)}{p+\alpha}}}
{(\psi_1(s))^2s^{\frac{N-1+M(p-2)}{p-1}}}\qquad \ \text{for\ some}\,\, C\in\R\setminus\{0\}.
$$
Thus
$$c(s)\sim As^{\frac{p(N-1)+\al(2p-N-1)}{(p-1)(p+\al)}+1},$$
and hence
$$h(s)=c(s)\psi_1(s)\sim As^{\frac{p(p-1)+\al(2p-N+1)}{(p-1)(p+\al)}}, \qquad \text{as}\,\, s\to+\infty\ \text{with}\ A\ne0.$$
By the weighted Hardy-Sobolev inequality (see, e.g., Lemma 2.3 in \cite{LDMR}), we get $h\notin\mathcal{D}$.

\medskip

\textbf{Case 2.} $\gamma=\gamma_2:=0$. In such case, we have
\begin{align}\label{d213}
\psi_2(s):=\frac{(p-1)-s^{\frac{p}{p-1}}}
	{(1+s^{\frac{p}{p-1}})^{\frac{N+\alpha}{p+\alpha}}}
\end{align}
is a solution to equation \eqref{d212}. We assert that all solutions to the equation $\mathcal{L}(\psi)=0$ in $\mathcal{D}$ can be expressed in the form $\psi=c\psi_2$, where $c\in\R$. A simple calculation reveals that $\psi_2\in\mathcal{D}$ and satisfies $\mathcal{L}(\psi_2)=0$. Now, assume $h$ is a second solution to equation \eqref{d212} that is linearly independent with respect to $\psi_2$ and has the form
$$h(s)=c(s)\psi_2(s).$$
Then we have
$$c''(s)\psi_2(s)+c'(s)\lr2\psi_2'(s)+\frac{\psi_2(s)}{s}\lr\frac{N-1}{p-1}+\frac{M(p-2)}{p-1}\frac{1}{1+s^{\frac{p}{p-1}}}\rr\rr=0,$$
and thus
$$\frac{c''(s)}{c'(s)}=-2\frac{\psi_2'(s)}{\psi_2(s)}-\frac{1}{s}\lr\frac{N-1}{p-1}+\frac{M(p-2)}{p-1}\frac{1}{1+s^{\frac{p}{p-1}}}\rr.$$
Through straightforward calculations, it can be shown that
$$
c'(s)=C\frac{(1+s^{\frac{p}{p-1}})^{\frac{(N+\alpha)(p-2)}{p+\alpha}}}
	{(\psi_2(s))^2s^{\frac{N-1+M(p-2)}{p-1}}}\qquad \ \text{for\ some}\,\, C\in\R\setminus\{0\}.
$$
Thus
$$c(s)\sim As^{\frac{(p-\al)(N-p)}{(p-1)(p+\al)}},$$
and hence
$$h(s)=c(s)\psi_2(s)\sim As^{-\frac{\al(N-p)}{(p-1)(p+\al)}}, \,\quad\, \text{as}\,\, s\to+\infty\ \text{with}\ A\ne0.$$
However, $h\notin\mathcal{D}$. In fact, from the weighted Hardy-Sobolev inequality (see, e.g., Lemma 2.3 in \cite{LDMR}), we get, if $q\geq1, s>q-N$ and $R\geq0$, then
\begin{equation}\label{dd214}
\int_{\mathbb{R}^N \setminus B_R(0)} |x|^{s-q}|\phi|^q \md x
\leq c(N, q, s) \int_{\mathbb{R}^N \setminus B_R(0)} |x|^s|\nabla \phi|^q \md x \,\,\quad\text{for\ any} \,\, \phi\in C_c^1(\R^N).
\end{equation}

If $p\in(1,2]$, applying \eqref{dd214} with $s=\max\left\{\frac{2\al(N-p)}{(p-1)(p+\al)}+2-N,\frac{(2-p)p(N-1)}{(p-1)(p+\al)}+\frac{(2-p)\al}{p+\al}\right\}$ and $q=2$, and combining with \eqref{d14}, we have
\begin{align}\label{dd215}
	&\quad\int_{\mathbb{R}^N \setminus B_R(0)} |x|^{s-2}|h|^2 \md x \\
	&\leq c \int_{\mathbb{R}^N \setminus B_R(0)} |x|^{s}|\nabla h|^2 \md x \nonumber\\
	&\leq c(R) \int_{\mathbb{R}^N} |\nabla U|^{p-2}|\nabla h|^2 \md x.\nonumber
\end{align}

If $p\in(2,N)$, applying \eqref{dd214} with $s=\min\left\{\frac{2\al(N-p)}{(p-1)(p+\al)}+2-N,-\frac{(p-2)p(N-1)}{(p-1)(p+\al)}-\frac{(p-2)\al}{p+\al}\right\}$ and $q=2$, and combining with \eqref{d14}, we can also obtain \eqref{dd215}. But $h(x)=h(|x|)$ does not satisfy inequality \eqref{dd215}, which is a contradiction.

\smallskip

Thus we obtain that \eqref{qq210} has solutions
\begin{align}\label{26}
	\gamma_1=M-1,\ \ \psi_1(s)=\frac{s^{\frac{1}{p-1}}}
	{(1+s^{\frac{p}{p-1}})^\frac{M}{p}} \ \ \text{and}\ \
\gamma_2=0,\ \ \psi_2(s)=\frac{(p-1)-s^{\frac{p}{p-1}}}
{(1+s^{\frac{p}{p-1}})^\frac{M}{p}}.
\end{align}	
Moreover, from Sturm-Liouville theorem, \eqref{qq210} has a sequence of simple eigenvalues $\gamma_1>\gamma_2>\cdots$. By \eqref{26}, we get $\gamma_j\leq0,\ j\geq3$. Since $q^2\la_k\geq0$, we conclude that \eqref{23} has a nontrivial solution if and only if
$$q^2\la_k\in\{0,M-1\}.$$
If $q^2\la_k=0$, then $k=0$. If $q^2\la_k=M-1$, that is,
\begin{align}\label{27}
	(p-1)\alpha^2+p(N+p-2)\alpha +p^2(1-k)(k+N-1)=0,\ \ \text{for\ some}\ k\in\N.
\end{align}	
The equation \eqref{27} has a solution for $\alpha\geq0$, which is
\begin{align}\label{k224}\alpha=\al(k):=\frac{p\sqrt{(N+p-2)^2+4(k-1)(p-1)(k+N-1)}-p(N+p-2)}{2(p-1)}.\end{align}	
Turning back to \eqref{22}, we get the solutions
\begin{align}\label{28}
	& \varphi_0(r)=\frac{(p-1)-r^{\frac{p+\alpha}{p-1}}}
	{(1+r^{\frac{p+\alpha}{p-1}})^{\frac{N+\alpha}{p+\alpha}}},\quad \text{if}\  \alpha\ne\al(k),\ k\in\N,\nonumber\\
	&\varphi_0(r)=\frac{(p-1)-r^{\frac{p+\alpha}{p-1}}}
	{(1+r^{\frac{p+\alpha}{p-1}})^{\frac{N+\alpha}{p+\alpha}}},\quad 	\varphi_k(r)=\frac{r^{\frac{p+\alpha}{p(p-1)}}}
	{(1+r^{\frac{p+\alpha}{p-1}})^{\frac{N+\alpha}{p+\alpha}}}\quad \text{if}\  \alpha=\al(k),\ k\in\N.
\end{align}
This completes our proof of Theorem \ref{th11}.
\end{proof}

\begin{proof}[Proof of Corollary \ref{c12}.]
The equation \eqref{11} corresponds to the following eigenvalue problem
\begin{equation}\label{29}
	\left\{
	\begin{aligned}
		&-\textrm{div}(|\nabla U_\al|^{p-2}\nabla v)-(p-2)
		\textrm{div}
		\lr|\nabla U_\al|^{p-4}(\nabla U_\al\cdot\nabla v)\nabla U_\al\rr=\mu (p_\al^*-1)|x|^{\alpha}U_\al^{p_\al^*-2}v, \\
		&v\in D^{1,p}(\R^N).
	\end{aligned}
	\right.
\end{equation}
The Morse index of $U_\al$ is the sum of dimensions of the eigenspaces of  \eqref{29} associated with $\mu<1$. In the previous proof of Theorem \ref{th11}, we obtain the equation
\begin{equation}\label{210}
\left\{
\begin{aligned}
	&\quad-(p-1)\varphi_k''(r)-\frac{\varphi_k'(r)}{r}\lr(N-1)+\
	\frac{(p-2)(N+\al)}{1+r^{\frac{p+\al}{p-1}}}\rr+
	\la_k\frac{\varphi_k(r)}{r^2}\\
	&=\mu\frac{(N+\al)(Np+p\al-N+p)}{p-1}
	\frac{r^{\frac{p+\al}{p-1}-2}}{(1+r^{\frac{p+\al}{p-1}})^2}\varphi_k(r), \quad  r\in(0,\infty),\,\, \varphi_k\in\mathcal{B}, \\
	&\varphi_k'(0)=0\ \ \text{if}\ \ k=0\ \ \text{and}
	\ \ \varphi_k(0)=0\ \ \text{if} \ \ k\geq1.
\end{aligned}
\right.
\end{equation}
For every $k\geq0$, the equation \eqref{210} has a monotonically increasing sequence of eigenvalues $\mu_{i,k},\ i=1,2,\cdots$, and an associated eigenfunction that has exactly $i-1$ zeros in $(0, \infty)$.

~We claim that $\mu_{i,k}\geq1$ for all $i\geq2$ and for any $k\geq0$. In fact, by  transformation \eqref{e23}, similar to our proof of Theorem \ref{th11}, we have
\begin{equation}\label{212}
\left\{
\begin{aligned}
	&\quad-\eta_k''(s)-\frac{\eta_k'(s)}{s}\lr\frac{M-1}{p-1}+
	\frac{(p-2)M}{(p-1)(1+s^{\frac{p}{p-1}})}\rr+\frac{q^2\la_k}{(p-1)}\frac{\eta_k(s)}{s^2}\\
	&=\mu
	\frac{M(Mp-M+p)}{(p-1)^2}\frac{s^{\frac{p}{p-1}-2}}{(1+s^{\frac{p}{p-1}})^2}\eta_k(s), \quad s\in(0,\infty),\ \eta_k\in\mathcal{\widetilde{B}},\\
	&\eta_k'(0)=0\ \ \text{if}\ \ k=0\ \ \text{and}
	\ \ \eta_k(0)=0\ \ \text{if} \ \ k\geq1.
\end{aligned}
\right.
\end{equation}
For any fixed $\mu<1$, we consider the more general equation
\begin{align}\label{213}
	&\quad-\eta_k''(s)-\frac{\eta_k'(s)}{s}\lr\frac{M-1}{p-1}+
	\frac{(p-2)M}{(p-1)(1+s^{\frac{p}{p-1}})}\rr+\frac{\gamma}{(p-1)}\frac{\eta_k(s)}{s^2}\\
	&=\mu
	\frac{M(Mp-M+p)}{(p-1)^2}\frac{s^{\frac{p}{p-1}-2}}{(1+s^{\frac{p}{p-1}})^2}\eta_k(s)\nonumber,
\end{align}
which has a decreasing sequence of eigenvalues $\gamma_1(\mu)>\gamma_2(\mu)>\cdots$. Based on the min-max characterization of the eigenvalues and \eqref{26}, we get that
\begin{align}\label{214}
	M-1=\gamma_1(1)>\gamma_1(\mu)\ \ \text{and}\ \
	0=\gamma_2(1)>\gamma_2(\mu).
\end{align}	
Let $\varphi_{i,k}$ be an eigenfunction of \eqref{210} corresponding to the eigenvalue $\mu_{i,k}$ for some $i\geq 2, \ k\geq0$, and assume that $\mu_{i,k}<1$. Then $\varphi_{i,k}(r^{\frac{p}{p+\al}})$ is an eigenfunction of \eqref{212} associated with $\gamma(\mu_{i,k})=q^2\la_k\geq0$, which changes sign exactly $i-1$ times ($i\geq2$). However, according to \eqref{214}, we get $\gamma(\mu_{i,k})\leq\gamma_2(\mu_{i,k})<\gamma_2(1)=0$ for $i\geq 2$ and $k\geq0$. Thus a contradiction arises and our claim was proved.

~Therefore, it follows that only the first eigenvalue of \eqref{210} may be relevant for the Morse index when $i= 1$. By Theorem \ref{th11}, when $\alpha=\al(k)$, due to \eqref{28}, we get
\begin{align}\label{215}
\varphi_k(r)=\frac{r^{\frac{p+\alpha}{p(p-1)}}}
	{(1+r^{\frac{p+\alpha}{p-1}})^{\frac{N+\alpha}{p+\alpha}}}
\end{align}
is the first eigenfunction of \eqref{210} associated with the eigenvalue $\mu=1$. Moreover, for general $\alpha>0$ and $k\geq0$, we can verify that the first eigenfunction of \eqref{210} is
\begin{align}\label{q216}
\varphi_k(r)=\frac{r^{\rho_k}}
{(1+r^{\frac{p+\alpha}{p-1}})^{\nu_k}},
\end{align}
where $$\rho_k=\frac{\sqrt{((p-2)(N+\al)+N-p)^2+4(p-1)\la_k}-((p-2)(N+\al)+N-p)}{2(p-1)}$$ and
$$\nu_k=\frac{\sqrt{(N-p)^2+4(p-1)\la_k}+2\rho_k(p-1)+N-p}{2(p+\al)},$$ and the corresponding first eigenvalue
\begin{align}\label{216}
	\mu_{1,k}=\frac{(p+\al)^2\nu_k(\nu_k+1)+\nu_k(p+\al)(p-2)(N+\al)}{(N+\al)(Np+p\al-N+p)}.
\end{align}
Now we regard $\mu_{1,k}$ as a function of nonnegative real variable $k\geq0$. Since $\mu_{1,k}=1$ is equivalent to equation \eqref{27}, i.e., $k=\frac{(2p-Np)+\sqrt{4(p-1)\al^2+4p(N+p-2)\al+N^2p^2}}{2p}$, and $\mu_{1,k}$ given by \eqref{216} is monotonically increasing with respect to \(k\), thus the eigenvalue $\mu_{1,k}<1$ is equivalent to
$$k<\frac{(2p-Np)+\sqrt{4(p-1)\al^2+4p(N+p-2)\al+N^2p^2}}{2p},$$
and the eigenfunction of \eqref{210} associated with $\mu_{1,k}$ are linear combination of functions of the form
\begin{align}\label{e216}
	v_{k,i}=\frac{|x|^{\rho_k}}
	{(1+|x|^{\frac{p+\alpha}{p-1}})^{\nu_k}}\Phi_{k,i}\lr\frac{x}{|x|}\rr,\ \quad
	\text{where}\ \Phi_{k,i}\in\boldsymbol{\Phi}_k(\R^N).
\end{align}
Because $\text{dim}(\boldsymbol{\Phi}_k(\R^N))=\frac{(N+2k-2)(N+ k-3)!}{(N-2)!k!}$, the proof of Corollary \ref{c12} is completed.
\end{proof}

\bigskip

\section{The approximate problem in $B_{\frac1\var}(0)$}

In this section, we consider the approximate problem
\begin{align}\label{31}
	\begin{cases}
		-\Delta_p u=|x|^{\alpha}\lr u+U_\al\lr\frac1\var\rr\rr^{p_\al^*-1}, & x \in B_{\frac1\var}(0), \\
		u\geq0, & x\in B_{\frac1\var}(0),\\
		u=0, & x\in \partial B_{\frac1\var}(0),
	\end{cases}
\end{align}
where $\alpha$ is fixed, $U_\al(x)$ is defined in \eqref{14} and $B_{\frac1\var}(0)$ denotes the ball of radius $\frac1\var$ centered at the origin. Define
\begin{equation}\label{32}
	u_{\var,\alpha}(x)=\left\{
	\begin{aligned}
		&U_\al(x)-U_\al\lr\frac1\var\rr=\frac{C_{N,p,\alpha}}
		{(1+|x|^{\frac{p+\alpha}{p-1}
			})^{\frac{N-p}{p+\al}}}-
			\frac{C_{N,p,\alpha}\var^{\frac{N-p}{p-1}}}
			{(1+\var^{\frac{p+\alpha}{p-1}
				})^{\frac{N-p}{p+\al}}},&  \text{if}\ |x|\leq\frac1\var, \\
		&0, & \text{if}\ |x|>\frac1\var.
	\end{aligned}
	\right.
\end{equation}
For the sake of convenience, we define $\beta_\al(n):=U_\al\lr\frac1\var\rr=
\frac{C_{N,p,\alpha}\var^{\frac{N-p}{p-1}}}
{(1+\var^{\frac{p+\alpha}{p-1}
	})^{\frac{N-p}{p+\al}}}$.
\begin{lem}\label{lem31}
	For any $\al\geq0$ and any sufficiently small $0<\var<(p-1)^{-\frac{p-1}{p+\al}}$, the function $u_{\var,\alpha}$ is a radial solution of \eqref{31}, which is nondegenerate in the space of radial functions.
\end{lem}
\begin{proof}
For any sufficiently small $0<\var<(p-1)^{-\frac{p-1}{p+\al}}$ and $x\in B_{\frac1\var}(0)$, it is easy to verify that $u_{\var,\alpha}$ is a radial solution of \eqref{31}. If the linearized problem
\begin{equation}\label{33}
\left\{
\begin{aligned}
&\quad-\textrm{div}(|\nabla U_\al|^{p-2}\nabla \phi)-(p-2)
\textrm{div}
\lr|\nabla U_\al|^{p-4}(\nabla U_\al\cdot\nabla \phi)\nabla U_\al\rr\\
&= (p_\al^*-1)|x|^{\alpha}\lr u_{\var,\al}+U_\al\lr\frac1\var\rr\rr^{p_\al^*-2}\phi \ \quad \text{in}\ B_{\frac1\var}(0), \\
&\phi=0
	\quad\quad 	\quad\quad 	\quad\quad 	\quad\quad 	\quad\quad 	\quad\quad 	\quad\quad	\quad\quad \text{on} \,\, \partial B_{\frac1\var}(0)
	\end{aligned}
	\right.
\end{equation}
has no nontrivial radial solution, then we say $u_{\var,\alpha}$ is radially nondegenerate. Write \eqref{33} in the form of radial coordinates, i.e.,
\begin{equation}\label{34}
	\left\{
	\begin{aligned}
		&-(p-1)\lr r^{N-1}|U_\al'|^{p-2}\phi'\rr'=(p_\al^*-1)r^{N+\alpha-1}U_\al^{p_\al^*-2}\phi   &\text{in}\ \lr0,\frac1\var\rr, \\
		&\phi'(0)=0,\quad\phi\lr\frac1\var\rr=0.
	\end{aligned}
	\right.
\end{equation}
Note that the function
$$z(r)=\frac{(p-1)-r^{\frac{p+\alpha}{p-1}}}
{(1+r^{\frac{p+\alpha}{p-1}})^{\frac{N+\alpha}{p+\alpha}}}$$
satisfies the linearized equation
\begin{align}\label{35}
	-(p-1)\lr r^{N-1}|U_\al'|^{p-2}z'\rr'=(p_\al^*-1)r^{N+\alpha-1}U_\al^{p_\al^*-2}z \quad \text{in}\ \R^+,\quad z'(0)=0.
\end{align}
However, we find that $z(\frac1\var)\ne0$ for $\var<(p-1)^{-\frac{p-1}{p+\al}}$.

~Multiplying \eqref{34} by $z$, \eqref{35} by $\phi$ and integrating in $\left(0,\frac1\var\right)$, we obtain
$$\phi'\lr\frac1\var\rr=0.$$
Moreover, $\phi\lr\frac1\var\rr=0$, from the classical existence and uniqueness theorem for initial value problems, we must have $\phi\equiv0$. This shows that \eqref{34} does not have any nontrivial radial solutions, so $u_{\var,\al}$ is radially nondegenerate.
\end{proof}

\begin{prop}\label{pp32}
	Let $\al_n$ and $\var_n$ be sequences such that $\al_n\to\al>0$ and $\var_n\to0$ as $n\to+\infty$.
	Defining $u_n:=u_{\var_n,\al_n}$ with $u_{\var,\al}$ be given by \eqref{32}, we have that
\begin{align}\label{36}
	\|u_n-U_\al\|_X\to0,\quad\quad as \ n\to+\infty.
\end{align}	
\end{prop}
\begin{proof}
	According to the definition of norm \eqref{110}, we need to prove $\|u_n-U_\al\|_{1,p}\to0$ and $\|u_n-U_\al\|_{\gamma}\to0$ as $n\to+\infty$.

We obtain
\begin{align*}
	\int_{\R^N}\left|\nabla u_n-\nabla U_\al\right|^p\md x =\int_{B_{\frac{1}{\var_n}}}\left|\nabla U_{\al_n}-\nabla U_\al\right|^p\md x +\int_{\R^N\backslash B_{\frac{1}{\var_n}}}\left|\nabla U_\al\right|^p\md x .
\end{align*}	
Because $U_\al\in D^{1,p}(\R^N)$, we have
\begin{align*}
	\int_{\R^N\backslash B_{\frac{1}{\var_n}}}\left|\nabla U_\al\right|^p\md x\to0, \quad\text{as}\ n\to+\infty .
\end{align*}
From Lebesgue's dominated convergence theorem, we get
\begin{align*}
\int_{B_{\frac{1}{\var_n}}}\left|\nabla U_{\al_n}-\nabla U_\al\right|^p\md x\to0, \quad\text{as}\ n\to+\infty .
\end{align*}
Using  the mean value theorem, it follows that
	\begin{align*}
	\|u_n-U_\al\|_{\gamma}&=\sup_{x\in\R^N}(1+|x|)^\gamma|u_n-U_\al|\\
&\leq \sup_{x\in B_{\frac{1}{\var_n}}}(1+|x|)^\gamma|U_{\al_n}(x)-U_\al(x)|
+\sup_{x\in B_{\frac{1}{\var_n}}}(1+|x|)^\gamma U_{\al_n}\lr\frac1{\var_n}\rr
\\
&\quad+C_{N,\al,p}\sup_{x\in \R^N\backslash B_{\frac{1}{\var_n}}}\frac{(1+|x|)^\gamma }{(1+|x|^{\frac{p+\al}{p-1}})^{\frac{N-p}{p+\al}}}\\
&\leq O(|\al-\al_n|)+O\lr\var_n^{\frac{N-p}{p-1}-\gamma}\rr,
\end{align*}	
where we used the fact $\gamma<\frac{N-p}{p-1}$. This finishes our proof of Proposition \ref{pp32}.
\end{proof}

\smallskip

\subsection{Convergence of the spectrum}

Consider the linearized problem associated with \eqref{31}, i.e.,
\begin{equation}\label{37}
	\left\{
	\begin{aligned}
		&\quad-\textrm{div}(|\nabla U_\al|^{p-2}\nabla v)-(p-2)
		\textrm{div}
		\lr|\nabla U_\al|^{p-4}(\nabla U_\al\cdot\nabla v)\nabla U_\al\rr\\
		&= (p_\al^*-1)|x|^{\alpha}\lr u_{\var,\al}+U_\al\lr\frac1\var\rr\rr^{p_\al^*-2}v &\text{in}\ B_{\frac1\var}(0), \\
		&v=0, &\text{on}\ \partial B_{\frac1\var}(0).
	\end{aligned}
	\right.
\end{equation}
Since $u_{\var,\al}=U_\al(x)-U_\al(\frac1\var)$, the equation \eqref{37} can be reduced to
\begin{equation}\label{38}
	\left\{
	\begin{aligned}
		&\quad-\textrm{div}(|\nabla U_\al|^{p-2}\nabla v)-(p-2)
		\textrm{div}
		\lr|\nabla U_\al|^{p-4}(\nabla U_\al\cdot\nabla v)\nabla U_\al\rr\\
		&= (p_\al^*-1)|x|^{\alpha}U_\al^{p_\al^*-2}v &\text{in}\ B_{\frac1\var}(0), \\
		&v=0, &\text{on}\ \partial B_{\frac1\var}(0).
	\end{aligned}
	\right.
\end{equation}
The problem \eqref{38} can be decomposed into the radial part and the angular part by using the spherical harmonic functions in similar way as in Section 2. $v$ is a solution of \eqref{38} if and only if $v_k(r)=\int_{\mms^{N-1}}v(r,\theta)\Phi_k(\theta)\md\theta$ is a solution of
\begin{equation}\label{39}
	\left\{
	\begin{aligned}
		&\quad-(p-1)v_k''(r)-\frac{v_k'(r)}{r}\lr(N-1)+\
		\frac{(p-2)(N+\al)}{1+r^{\frac{p+\al}{p-1}}}\rr+
		\la_k\frac{v_k(r)}{r^2}\\
		&=\frac{(N+\al)(Np+p\al-N+p)}{p-1}
		\frac{r^{\frac{p+\al}{p-1}-2}}{(1+r^{\frac{p+\al}{p-1}})^2}v_k(r), \quad  r\in\left(0,\frac1\var\right),\ \varphi_k\in\mathcal{B}, \\
		&v_k'(0)=0\ \ \text{if}\ \ k=0\ \ \text{and}
		\ \ v_k(0)=0\ \ \text{if} \ \ k\geq1
	\end{aligned}
	\right.
\end{equation}
for $\la_k=k(N+k-2)$, where $\Phi_k(\theta)$ denotes the $k$-th spherical harmonic function. Thus, if $v_k$ is the solution to \eqref{39}, we deduce that, all eigenfunctions of \eqref{37} are given by $v_k(r)\Phi_k(\theta)$. According to Lemma \ref{lem31}, we get that, only when $k\ne0$, \eqref{39} possesses a solution. Let's consider the following eigenvalue problem
\begin{equation}\label{310}
	\left\{
	\begin{aligned}
		&\quad-(p-1)w''(r)-\frac{w'(r)}{r}\lr(N-1)+\
		\frac{(p-2)(N+\al)}{1+r^{\frac{p+\al}{p-1}}}\rr\\&
		-\frac{(N+\al)(Np+p\al-N+p)}{p-1}
		\frac{r^{\frac{p+\al}{p-1}-2}}{(1+r^{\frac{p+\al}{p-1}})^2}w(r)
		=\mu\frac{w(r)}{r^2}
	, \quad  r\in\lr0,\frac1\var\rr,\\
		&w(0)=0=w\lr\frac1\var\rr.
	\end{aligned}
	\right.
\end{equation}
There is a sequence of eigenvalues $\mu_i^\var(\al)$ for \eqref{310}. So the equation \eqref{39} is equivalent to finding integers $i, k\geq1$ and $\alpha>0$ such that
\begin{align}\label{311}
	-\la_k=\mu_i^\var(\al).
\end{align}

Based on the above analysis, we will consider the eigenvalues of equation \eqref{310}, and prove the following Lemma.
\begin{lem}\label{lem33}
	For any $\var>0$ sufficiently small, we have that
	\begin{align}\label{312}
		\mu_1^\var(\al)<0
	\end{align}	
and
	\begin{align}\label{313}
	\mu_2^\var(\al)>0.
\end{align}	
\end{lem}
\begin{proof}
The standard calculation shows that $\mu_1^\var(\al)<0$ for $\var>0$ small enough.  In fact, if $w$ is an eigenfunction of the corresponding equation with the eigenvalue $\mu_1^\var(\al)$, i.e.,
\begin{equation}\label{314e}
	\left\{
	\begin{aligned}
		&\quad-(p-1)w''(r)-\frac{w'(r)}{r}\lr(N-1)+\
		\frac{(p-2)(N+\al)}{1+r^{\frac{p+\al}{p-1}}}\rr\\&
		-\frac{(N+\al)(Np+p\al-N+p)}{p-1}
		\frac{r^{\frac{p+\al}{p-1}-2}}{(1+r^{\frac{p+\al}{p-1}})^2}w(r)
		=\mu_1^\var(\al)\frac{w(r)}{r^2}
	, \quad  r\in\lr0,\frac1\var\rr,\\
		&w(0)=0=w\lr\frac1\var\rr.
	\end{aligned}
	\right.
\end{equation}
According to the definition of the eigenvalue, $\mu_1^\var(\al)$ is the minimum value point corresponding to the following functional
\begin{equation}\label{315}
	I(w)=\frac{\int_0^{\frac1\var}r^{N-1}|U_\al'|^{p-2}|w'|^2-(p_\al^*-1)\int_0^{\frac1\var}r^{N-1}|U_\al|^{p_\al^*-2}w^2}
{\int_0^{\frac1\var}r^{N-3}|U_\al'|^{p-2}w^2},
\end{equation}
i.e.,
$$\mu_1^\var(\al):=\min_{w\in\mathcal{B}} I(w).$$
We can use the test function
	$$w(r)=\frac{r^{\frac{p+\alpha}{p(p-1)}}}
	{(1+r^{\frac{p+\alpha}{p-1}})^{\frac{N+\alpha}{p+\alpha}}}\eta_\var(r)=:g(r)\eta_\var(r)$$
to prove \eqref{312}, where $\eta_\var(r)\in C_0^\infty(0,\frac1\var)$ is a suitable cut-off function.
From \eqref{315}, it's equivalent to proving that the numerator of \eqref{315} is strictly less than $0$, i.e.,
\begin{align}\label{316}
	&\quad\int_0^{\frac1\var}r^{N-1}|U_\al'|^{p-2}|g'|^2\eta_\var^2+\int_0^{\frac1\var}r^{N-1}|U_\al'|^{p-2}|g|^2|\eta_\var'|^2+2\int_0^{\frac1\var}r^{N-1}|U_\al'|^{p-2}gg'\eta_\var\eta'_\var\\
&<(p_\al^*-1)\int_0^{\frac1\var}r^{N-1}|U_\al|^{p_\al^*-2}g^2\eta_\var^2,\nonumber
\end{align}
and $g(r)$ satisfies the equation
\begin{align}\label{317}
-(r^{N-1}|U_\al'|^{p-2}g')'=(p_\al^*-1)r^{N-1}|U_\al|^{p_\al^*-2}g-\lambda_1 r^{N-3}|U_\al'|^{p-2}g.
\end{align}
Multiplying \eqref{317} with $g\eta_\var^2$ and integrating, and putting it into \eqref{316}, we derive that \eqref{316} is equivalent to
\begin{align}\label{318}
\int_0^{\frac1\var}r^{N-1}|U_\al'|^{p-2}|g|^2|\eta_\var'|^2<\la_1\int_0^{\frac1\var}r^{N-3}|U_\al'|^{p-2}|g|^2\eta_\var^2.
\end{align}
Taking
\begin{equation*}
	\eta_\var=0\ \text{in} \ (0,\var)\ \text{and}\ \lr\frac{1}{2\var},\frac{1}{\var}\rr,\quad  \eta_\var=1\ \text{in} \ \left(2\var,\frac{1}{4\var}\right), \quad \eta_\var\in(0,1)\ \text{in} \ (\var,2\var)\ \text{and}\ \lr\frac{1}{4\var},\frac{1}{2\var}\rr,
\end{equation*}
we can derive \eqref{318} and hence \eqref{312}.

As to \eqref{313}, it follows from the monotonicity of the eigenvalues with respect to the domain. Let's use $\mu=\mu_i(\al)$ to represent the eigenvalue of the problem
\begin{equation}\label{314}
	\left\{
	\begin{aligned}
		&-(p-1)w''(r)-\frac{w'(r)}{r}\lr(N-1)+\
		\frac{(p-2)(N+\al)}{1+r^{\frac{p+\al}{p-1}}}\rr\\&\quad
		-\frac{(N+\al)(Np+p\al-N+p)}{p-1}
		\frac{r^{\frac{p+\al}{p-1}-2}}{(1+r^{\frac{p+\al}{p-1}})^2}w(r)
		=\mu\frac{w(r)}{r^2}
	, \quad  r\in\lr0,+\infty\rr,\\
		&w\in\mathcal{B},
	\end{aligned}
	\right.
\end{equation}
where $\mathcal{B}$ is defined by \eqref{23b}. We obtain that $\mu_2^\var(\al)>\mu_2(\al)$. By \eqref{e23}, \eqref{213}, \eqref{214}, we derive $\mu_2(\al)=0$, which completes our proof.
\end{proof}

\begin{lem}\label{lem34}
	Let $\al>0$ and $\var_n$ be a sequence such that $\var_n\to0$ as $n\to+\infty$.
	Let $\mu_1^{\var_n}(\al)$ be the first eigenvalue for \eqref{310} in $\lr0,\frac1{\var_n}\rr$ related to the exponent $\al$. Then
\begin{align}\label{320}
\mu_1^{\var_n}(\al)\to\mu_1(\al)=-\frac{(p+\al)(Np+p\al-p-\al)}{p^2},\quad\quad as \,\, n\to+\infty.
\end{align}	
Furthermore, the convergence of \eqref{320} is uniform in $\al$ on compact sets of $(0,+\infty)$.

Finally, let $w_n(r)$ be the first positive eigenfunction of \eqref{310} with respect to $\mu_1^{\var_n}(\al)$ such that $\|w_n\|_{L^\infty}=1$, we get
\begin{align}
&|w_n'(r)|\leq Cr^{-\frac{N-1}{p-1}},\quad\quad |w_n(r)|\leq Cr^{-\frac{N-p}{p-1}},\label{321}\\
&w_n(r)\to w(r)=\frac{r^{\frac{p+\al}{p(p-1)}}}
{(1+r^{\frac{p+\al}{p-1})^{\frac{N+\al}{p+\al}}}}\label{322}
\end{align}
uniformly in $r$ on compact sets of $[0,+\infty)$.
\end{lem}
\begin{proof}
Convergence in \eqref{320} arises from the dependence on the domain of eigenvalues or Sturm-Liouville theory.

Furthermore, from \eqref{22}-\eqref{23}, it follows that
$\mu_1(\al)=-\frac{(p+\al)(Np+p\al-p-\al)}{p^2}$.

   Let $w_n(r)$ be the first positive eigenfunction of \eqref{310} with respect to $\mu_1^{\var_n}(\al)$ such that $\|w_n\|_{L^\infty}=1$. Thus $w_n$ satisfies
\begin{equation}\label{323}
	\left\{
	\begin{aligned}
		&\quad-(p-1)w_n''(r)-\frac{w_n'(r)}{r}\lr(N-1)+\
		\frac{(p-2)(N+\al)}{1+r^{\frac{p+\al}{p-1}}}\rr\\&
		-\frac{(N+\al)(Np+p\al-N+p)}{p-1}
		\frac{r^{\frac{p+\al}{p-1}-2}}{(1+r^{\frac{p+\al}{p-1}})^2}w_n(r)
		=\mu_1^{\var_n}(\al)\frac{w_n(r)}{r^2}
	, \quad  r\in\lr0,\frac{1}{\var_n}\rr,\\
		&w_n\lr\frac{1}{\var_n}\rr=0,\quad\|w_n\|_{L^\infty}=1.
	\end{aligned}
	\right.
\end{equation}
Noting that, for $r$ large enough and $\al>0$, we have
\begin{align}\label{324}
\frac{(N+\al)(Np+p\al-N+p)}{p-1}\frac{r^{\frac{p+\al}{p-1}+N-3}}{(1+r^{\frac{p+\al}{p-1}})^2}+\mu_1^{\var_n}(\al)r^{N-3}<0,
\end{align}
where we have used $\mu_1^{\var_n}(\al)<0$ for $n$ large. Integrating \eqref{323} on $\lr r,\frac1{\var_n}\rr$, we obtain
	\begin{align}\label{325}
 & \quad r^{N-1}|U'_\al(r)|^{p-2}w'_n(r) \\
 &=\left(\frac{1}{\var_n}\right)^{N-1}
\Big|U'_\al\lr\frac{1}{\var_n}\rr\Big|^{p-2}w'_n\lr\frac{1}{\var_n}\rr+\frac{\mu_1^{\var_n}(\al)}{p-1}\int_r^{\frac{1}{\var_n}}t^{N-3}|U'_\al(t)|^{p-2}w_n(t)\md t \nonumber \\
&\quad+\frac{(N+\al)(Np+p\al-N+p)}{(p-1)^2}\int_r^{\frac{1}{\var_n}}
		\frac{t^{\frac{p+\al}{p-1}+N-3}|U'_\al(t)|^{p-2}}{(1+t^{\frac{p+\al}{p-1}})^2}w_n(t)\md t.\nonumber
	\end{align}
Noting that $0\leq w_n(r)\leq1$ and $w'_n\lr\frac{1}{\var_n}\rr\leq0$, from \eqref{324}, we obtain
\begin{align}\label{326}
 w'_n(r)<0\quad\quad\text{for}\ r\ \text{large\ enough}.
	\end{align}
Integrating \eqref{323} on $\lr 0,r\rr$, we obtain
	\begin{align}\label{327}
 -r^{N-1}|U'_\al(r)|^{p-2}w'_n(r)&=\frac{\mu_1^{\var_n}(\al)}{p-1}\int_0^{r}t^{N-3}|U'_\al(t)|^{p-2}w_n(t)\md t\\
&\quad+\frac{(N+\al)(Np+p\al-N+p)}{(p-1)^2}\int_0^{r}
		\frac{t^{\frac{p+\al}{p-1}+N-3}|U'_\al(t)|^{p-2}}{(1+t^{\frac{p+\al}{p-1}})^2}w_n(t)\md t.\nonumber
	\end{align}
Since $0\leq w_n(r)\leq1$ and $\mu_1^{\var_n}(\al)<0$ for $n$ large, from \eqref{326}, we get
\begin{equation}\label{328}
	|w'_n(r)|\leq\frac{C}{r^{N-1}|U'_\al(r)|^{p-2}}\times\left\{
	\begin{aligned}
		&C,&\text{if}\ N<2p+\al,\\
&C+\ln r,&\text{if}\ N=2p+\al,\\
&C+r^{\frac{N-2p-\al}{p-1}},&\text{if}\ N>2p+\al.
	\end{aligned}
	\right.
\end{equation}
If $N<2p+\al$, we get the optimal decay in \eqref{321}. If else, when $N\geq2p + \alpha$, from \eqref{328}, we have
\begin{equation}\label{329qq}
	|w_n(r)|\leq Cr^{-\frac{N-p}{p-1}}r^{\frac{N-2p-\al}{p-1}}.
\end{equation}
Then, inserting \eqref{328} into \eqref{327}, after $k$ steps of iteration, we obtain the following iterative sequence
\begin{equation}\label{330qq}
	|w_n(r)|\leq Cr^{f_k(N)}.
\end{equation}
After iterating the formula \eqref{330qq} once more, we can obtain
\begin{equation}\label{331qq}
	|w_n(r)|\leq Cr^{-\frac{N-p}{p-1}}r^{\frac{N-2p-\al}{p-1}}r^{f_k(N)}=:Cr^{f_{k+1}(N)},
\end{equation}
where $f_{k+1}(N):=-\frac{N-p}{p-1}+\frac{N-2p-\al}{p-1}+f_k(N)$ with $f_0(N)=\frac{N}{p-1}-\frac{2p+\al}{p-1}$. Thus we have $f_{k}(N)=\frac{N}{p-1}-\frac{2p+\al}{p-1}-k\frac{p+\al}{p-1}$. Let $f_k(N)=0$, we have $N=N_k:=(p+\al)k+2p+\al$ and $N_k\to+\infty$, as $k\to+\infty$. By using \eqref{330qq} instead of $|w_n|\leq1$, we can deduce in similar way as \eqref{328} that $|w'_n(r)|\leq\frac{C}{r^{N-1}|U'_\al(r)|^{p-2}}$ for any $N<N_k$. Consequently, we get \eqref{321}.

By \eqref{310} and \eqref{321}, we have $|w''_n(r)|\leq r^{-\frac{N+p-2}{p-1}}$. From the Arzel$\acute{a}$-Ascoli's Theorem, $w_n\to w$ in $\mathcal{B}$ weakly and uniformly w.r.t. $r$ on compact sets of $[0,+\infty)$. Using \eqref{321} again, we can pass the limit to \eqref{310} and get that $w$ is the solution of \eqref{314} corresponding to the eigenvalue $\mu_1(\al)$. Moreover, $w\not\equiv0$, indeed, from \eqref{321}, the maximum point of $w_n(r)$ converges to a point $r_0\in[0,+\infty)$ and $|w(r_0)|=1$ from the uniform convergence.

Finally, we prove the uniform convergence of $\mu_1^{\var_n}(\al)\to\mu_1(\al)$ in $\al$ on compact sets.  By multiplying \eqref{323} by $wr^{N-1}|U'_\al|^{p-2}$ and integrating on $\lr0,\frac1{\var_n}\rr$, multiplying \eqref{314} by $w_nr^{N-1}|U'_\al|^{p-2}$ and integrating on $\lr0,\frac1{\var_n}\rr$, and then subtracting, we get
\begin{align*}
 -\lr\frac{1}{\var_n}\rr^{N-1}
\Big|U'_\al\lr\frac{1}{\var_n}\rr\Big|^{p-2}w'_n\lr\frac{1}{\var_n}\rr w\lr\frac{1}{\var_n}\rr&=\frac{\lr\mu_1^{\var_n}(\al)-\mu_1(\al)\rr}{p-1}\int_0^{\frac{1}{\var_n}}t^{N-3}|U'_\al(t)|^{p-2}w_n(t)w(t)\md t.
	\end{align*}
From \eqref{321}, we have
\begin{align*}
 -\lr\frac{1}{\var_n}\rr^{N-1}
\Big|U'_\al\lr\frac{1}{\var_n}\rr\Big|^{p-2}w'_n\lr\frac{1}{\var_n}\rr w\lr\frac{1}{\var_n}\rr=O\left(\var_n^{\frac{N-1}{p-1}+\frac{\al}{p}}\right),
\end{align*}
as $n\to+\infty$, uniformly in $\al$ on compact sets of $(0,+\infty)$, while
\begin{align*}
\int_0^{\frac{1}{\var_n}}t^{N-3}|U'_\al(t)|^{p-2}w_n(t)w(t)\md t\to\int_0^{+\infty}t^{N-3}|U'_\al(t)|^{p-2}w^2(t)\md t,
	\end{align*}
as $n\to+\infty$, uniformly in $\al$ on compact sets of $(0,+\infty)$, which implies that
\begin{align*}
\sup_{\al\in K}|\mu_1^{\var_n}(\al)-\mu_1(\al)|=o(1),
	\end{align*}
as $n\to+\infty$, for any compact set $K\subset(0,+\infty)$. This completes our proof.
\end{proof}

\begin{prop}\label{pp35}
	Let $\al>0$ and $\var_n$ be a sequence such that $\var_n\to0$, as $n\to+\infty$. Let $\mu_1^{\var_n}(\al)$ be the first
eigenvalue of \eqref{323} in $\lr0,\frac1{\var_n}\rr$. Then we have that
\begin{align}\label{329}
	\frac{\partial \mu_1^{\var_n}(\al)}{\partial \al }\to-\frac{p(p+2\al+N-2)-2\al}{p^2}=\frac{\partial \mu_1(\al)}{\partial \al }<0
\end{align}	
uniformly on the compact set $K\subset(0,+\infty)$. Furthermore, for any integer $k\geq1$, the equation
\begin{align}\label{330}
	-k(N+k-2)=-\la_k=\mu_1^{\var_n}(\al)
\end{align}	
has only one solution $\al_k^n$ and
\begin{align}\label{331}
	\al_k^n\to \al(k), \quad as\ n\to+\infty,
\end{align}
where $\al(k)=\frac{p\sqrt{(N+p-2)^2+4(k-1)(p-1)(k+N-1)}-p(N+p-2)}{2(p-1)}$.	
\end{prop}
\begin{proof}
From known results (c.f. e.g. \cite{GGN}), we get that, if $w_n$ is the first eigenfunction of \eqref{323} associated with $\mu_1^{\var_n}(\al)$, then $\frac{\partial w_n}{\partial\al}$ and $\frac{\partial \mu_1^{\var_n}(\al)}{\partial\al}$ are continuous functions. By taking the derivative in \eqref{323} with respect to $\al$, we get that
	\begin{equation}\label{332}
	\left\{
	\begin{aligned}
		&\quad-(p-1)\lr\frac{\partial w_n}{\partial\al}\rr''-\frac{1}{r}\lr\frac{\partial w_n}{\partial\al}\rr'\lr(N-1)+
	\frac{(p-2)(N+\al)}{1+r^{\frac{p+\al}{p-1}}}\rr\\&
-\frac{w'_n}{r}\frac{\partial}{\partial \al}\lr
	\frac{(p-2)(N+\al)}{1+r^{\frac{p+\al}{p-1}}}\rr
-\frac{(N+\al)(Np+p\al-N+p)}{p-1}\frac{r^{\frac{p+\al}{p-1}-2}}{(1+r^{\frac{p+\al}{p-1}})^2}\frac{\partial w_n}{\partial\al}\\
		&-\frac{\partial}{\partial \al}\lr\frac{(N+\al)(Np+p\al-N+p)}{p-1}
		\frac{r^{\frac{p+\al}{p-1}-2}}{(1+r^{\frac{p+\al}{p-1}})^2}\rr w_n\\
		&=\frac{\partial \mu_1^{\var_n}(\al)}{\partial \al}\frac{w_n}{r^2}+\mu_1^{\var_n}(\al)\frac{\frac{\partial w_n}{\partial \al}}{r^2},
	 \qquad  r\in\lr0,\frac1{\var_n}\rr,\\
		&\lr\frac{\partial w_n}{\partial \al}\rr\lr\frac1{\var_n}\rr=0.
	\end{aligned}
	\right.
\end{equation}
	Multiplying \eqref{323} by $\frac{\partial w_n}{\partial \al}$ and \eqref{332} by $w_n$, integrating and subtracting, we obtain
	\begin{align}\label{333e}
&\quad	-\int_0^{\frac1{\var_n}}\frac{\partial}{\partial \al}\lr\frac{(N+\al)(Np+p\al-N+p)}{p-1}
		\frac{r^{\frac{p+\al}{p-1}-2}}{(1+r^{\frac{p+\al}{p-1}})^2}\rr w_n^2r^{N-1}|U'_\al|^{p-2}\md r\\
&-\int_0^{\frac1{\var_n}}\frac{\partial}{\partial \al}\lr
	\frac{(p-2)(N+\al)}{1+r^{\frac{p+\al}{p-1}}}\rr w_n w'_nr^{N-2}|U'_\al|^{p-2}\md r=\frac{\partial \mu_1^{\var_n}(\al)}{\partial \al}\int_0^{\frac1{\var_n}}w_n^2r^{N-3}|U'_\al|^{p-2}
\md r.
\nonumber
	\end{align}
By Lemma \ref{lem34}, we can pass to the limit in \eqref{333e} and derive that, as $n\rightarrow+\infty$,
\begin{align}\label{334e}
	\frac{\partial \mu_1^{\var_n}(\al)}{\partial \al}\to&-\frac{\int_0^{+\infty}\frac{\partial}{\partial \al}\lr\frac{(N+\al)(Np+p\al-N+p)}{p-1}
		\frac{r^{\frac{p+\al}{p-1}-2}}{(1+r^{\frac{p+\al}{p-1}})^2}\rr w^2r^{N-1}|U'_\al|^{p-2}\md r}{\int_0^{+\infty}w^2r^{N-3}|U'_\al|^{p-2}
\md r}\\
&\quad-\frac{\int_0^{+\infty}\frac{\partial}{\partial \al}\lr
	\frac{(p-2)(N+\al)}{1+r^{\frac{p+\al}{p-1}}}\rr w w'r^{N-2}|U'_\al|^{p-2}\md r}{\int_0^{+\infty}w^2r^{N-3}|U'_\al|^{p-2}
\md r}\nonumber\\
&=\frac{\partial \mu_1(\al)}{\partial \al}=-\frac{p(p+2\al+N-2)-2\al}{p^2}\nonumber
	\end{align}
uniformly on compact sets of $[0,+\infty)$, which proves \eqref{329}.

Finally, due to $\frac{\partial \mu_1^{\var_n}(\al)}{\partial \al}<0$ and $\mu_1^{\var_n}(\al)\to\mu_1(\al)$, we get that, the equation \eqref{330} has exactly one root $\al=\al_k^n$ for any $k\geq1$. Since $\mu_1(\al)=-\frac{(p+\al)(Np+p\al-p-\al)}{p^2}$ and $\al=\al(k)$ solves the equation
\begin{align}\label{335e}
	\frac{(p+\al)(Np+p\al-p-\al)}{p^2}=\la_k=k(N+k-2),
\end{align}	
thus we have $\al(k)=\lim\limits_{n\to+\infty}\al_k^n$. This finishes our proof.
\end{proof}

\begin{rem}\label{morse-index}
From identities \eqref{39}, \eqref{311} and \eqref{330}, the solution $u_{n,\alpha}$ to problem \eqref{31} with parameter $\varepsilon_n$ is degenerate if and only if $\alpha = \alpha_k^n$ for $k = 1,2,\cdots$.  Furthermore, at each points $\alpha_k^n$, the Morse index of $u_{n,\alpha_k^n}$ changes.
	
Specifically, across the interval $(\alpha_k^n - \delta, \alpha_k^n + \delta)$ with $\delta > 0$ sufficiently small, the Morse index of $u_{n,\alpha}$ increases by precisely $\operatorname{Ker}(\Delta_{\mathbb{S}^{N-1}} + \lambda_k)$.
	
\end{rem}

\subsection{The bifurcation result in the ball}

In this section, using some ideas in \cite{GGPS}, we prove bifurcation result of \eqref{31} in the ball. To this end, we need some notations. As mentioned before, we use $u_{n,\al}$ to represent the radial solution of \eqref{31} corresponding to the exponent $\al$, for $\var=\var_n$, $B_{\frac{1}{\var_n}}$ to represent a ball centered at the origin with radius $\frac1{\var_n}$. Define
\begin{align}\label{336}
\mathcal{K}(n):=\Big\{&(\al,u_{n,\al})\in(0,+\infty)\times C_0^{1,\eta}(\overline B_{\frac{1}{\var_n}}) \ \text{such\  that}\ u_{n,\al} \\ &\text{is\ the\ radial\ solution\ of}\ \eqref{31} \ \text{defined\ in}\ \eqref{32}\Big\}\nonumber.
\end{align}
For the given curve $\mathcal{K}(n)$, we call a point $(\al_j,u_{n,\al_j})\in\mathcal{K}(n)$ is a nonradial bifurcation point, if every neighborhood of $(\al_j,u_{n,\al_j})$ in $(0,+\infty)\times C_0^{1,\eta}(\overline B_{\frac{1}{\var_n}})$, there exists a point $(\al,v_{n,\al})$ which is the nonradial solution of \eqref{31}.

Define the subspace $\mathcal{Z}_n$ of $C_0^{1,\eta}(\overline B_{\frac{1}{\var_n}})$ as follows
\begin{align}\label{337}
\mathcal{Z}_n=\Big\{&h\in C_0^{1,\eta}(\overline B_{\frac{1}{\var_n}}) \ \text{such\  that}\ h(x_1,\cdots,x_N)=h\lr g(x_1,\cdots,x_{N-1}),x_N\rr \\ &\text{for\ any}\ g\in\mathcal{O}(N-1)\Big\}\nonumber,
\end{align}
where $\mathcal{O}(N-1)$ is the orthogonal group in $\R^{N-1}$.

Now we need the following Theorem.

\begin{thm}[\textbf{\cite{AM}}]
\label{the36}
Let \(T \in C^1(D, X)\) be compact and such that \(1\) is not a characteristic value of \(T'(x_0)\) (i.e., $S'(x_0)$ is invertible) for some \(x_0 \in D\). Then, setting \(S(x) = x - T(x)\) and \(S(x_0) = p\), one has that \(x_0\) is an isolated solution of \(S(x) = p\) and there holds
	\[
	i(S, x_0) = (-1)^{\beta},
	\]
	where $i(S, x_0)$ is the Morse index and \(\beta\) is the sum of the algebraic multiplicities of all the characteristic values of \(T'(x_0)\) contained in \((0, 1)\).
\end{thm}

\begin{thm}\label{thm38}
For any $\al\in(0,+\infty)$, let the operator $\mathcal{L}^n(\al,h):(0,+\infty)\times\mathcal{Z}_n\to \mathcal{Z}_n$ be defined by $\mathcal{L}^n(\al,h):=(-\Delta_p)^{-1}(|x|^{\al}| h+\beta_{\al}(n)|^{p_{\al}^*-2}(h+\beta_{\al}(n)))$, then the operator $I - \mathcal{L}_{h}^n(\al,u_{n,\al}):(0,+\infty)\times\mathcal{Z}_n\to \mathcal{Z}_n$  is invertible for $\al \neq\al_k^n$, $k=1,2,\cdots$, where $\al_k^n$ are given by \eqref{331} and
\begin{align*}
	\mathcal{L}_{h}^n(\al,u_{n,\al}):=\lim_{t\to0}\frac{(-\Delta_p)^{-1}(|x|^{\al}| u_{n,\al}+th+\beta_{\al}(n)|^{p_{\al}^*-2}(u_{n,\al}+th+\beta_{\al}(n)))-u_{n,\al}}{t}.
\end{align*}	
\end{thm}
\begin{proof}
Let $v\in \text{Ker}\lr I - \mathcal{L}_{h}^n(\al,u_{n,\al})\rr$, i.e.,
\begin{align*}
	\mathcal{L}_{v}^n(\al,u_{n,\al}):=\lim_{t\to0}\frac{(-\Delta_p)^{-1}(|x|^{\al}| u_{n,\al}+tv+\beta_{\al}(n)|^{p_{\al}^*-2}(u_{n,\al}+tv+\beta_{\al}(n)))-u_{n,\al}}{t}=v.
\end{align*}	
Thus, for any test function $\varphi\in C_c^\infty(B_{\frac1{\var_n}})$, we have
\begin{align}\label{339q}
	&\quad\int_{B_{\frac1{\var_n}}}\lr|\nabla u_{n,\al}+t\nabla v+o(t)\nabla w|^{p-2}(\nabla u_{n,\al}+t\nabla v+o(t)\nabla w)-|\nabla u_{n,\al}|^{p-2}\nabla u_{n,\al}\rr\nabla \varphi\md x\\
	&=\int_{B_{\frac1{\var_n}}}|x|^{\al}\lr(u_{n,\al}+tv+\beta_{\al}(n))^{p_{\al}^*-1}-
	(u_{n,\al}+\beta_{\al}(n))^{p_{\al}^*-1}\rr\varphi\md x\nonumber,
\end{align}
where $w\in X$. On the one hand, from the LHS of the equation \eqref{339q}, we get
\begin{align}\label{340q}
	&\quad\int_{B_{\frac1{\var_n}}}\lr|\nabla u_{n,\al}+t\nabla v+o(t)\nabla w|^{p-2}(\nabla u_{n,\al}+o(t)\nabla w)-|\nabla u_n|^{p-2}\nabla u_{n,\al}\rr\nabla \varphi\md x\\
&=\int_{B_{\frac1{\var_n}}}\int_0^1(p-2)|\nabla u_{n,\al}+st\nabla v+o(t)s\nabla w|^{p-4}\lr\lr\nabla u_{n,\al}+st\nabla v+o(t)s\nabla w\rr\cdot\nabla \varphi\rr\nonumber\\
&\quad\times\lr\lr\nabla u_{n,\al}+st\nabla v+o(t)s\nabla w\rr\cdot\lr t\nabla v+o(t)w\rr\rr\nonumber\\
&\quad+|\nabla u_{n,\al}+st\nabla v+o(t)s\nabla w|^{p-2}\lr t\nabla v+o(t)w\rr\cdot\nabla \varphi\md s\md x.\nonumber
\end{align}
On the other hand, from the RHS of the equation \eqref{339q}, we have
\begin{align}\label{341q}
	&\quad\int_{B_{\frac1{\var_n}}}|x|^{\al}\lr(u_{n,\al}+tv+\beta_\al(n))^{p_{\al}^*-1}-
	(u_{n,\al}+\beta_\al(n))^{p_{\al}^*-1}\rr\varphi\md x\\
	&=t(p_{\al}^*-1)\int_{B_{\frac1{\var_n}}}|x|^{\al}\lr u_{n,\al}+\theta tv+\beta_\al(n)\rr^{p_{\al_n}^*-2}v\varphi\md x,\nonumber
\end{align}
where $\theta\in(0,1)$. Combining \eqref{340q}, \eqref{341q}  with \eqref{339q}, and letting $t\to0$, we get
\begin{align}\label{342qq}
&\quad\int_{B_{\frac1{\var_n}}}\lr|\nabla u_{n,\al}|^{p-2}\nabla v\cdot\nabla\varphi+(p-2)
|\nabla u_{n,\al}|^{p-4}(\nabla u_{n,\al}\cdot\nabla v)\nabla u_{n,\al}\cdot\nabla\varphi\rr\md x\\
&= (p_\al^*-1)\int_{B_{\frac1{\var_n}}}|x|^{\alpha}(u_{n,\al}+\beta_\al(n))^{p_\al^*-2}v\varphi\md x,\qquad\forall\ \varphi\in C_c^\infty(B_{\frac1{\var_n}}),\nonumber
\end{align}
that is, $v$ weakly solves \eqref{38}, which is the linearized equation of \eqref{31} around $u_{n,\al}$. By Lemma \ref{lem33} and Remark \ref{morse-index}, we know that, if $\alpha \ne\alpha_k^n$ for $k = 1,2,\dotsc$, then $u_{n,\al}$ is non-degenerate and hence $v=0$. Therefore, $\text{Ker}\lr I - \mathcal{L}_{h}^n(\al,u_{n,\al})\rr=\{0\}$, i.e., $I - \mathcal{L}_{h}^n(\al,u_{n,\al})$ is invertible when $\al \neq\al_k^n$, for $k = 1,2,\cdots$. This completes our proof.
\end{proof}

For the proof of the invertibility of the operator $I - \mathcal{L}_{h}^n(\al,u_{n,\al})$, we can also c.f. \cite{DJF}.

\begin{thm}\label{th36}
Fix $n\in \N$ and let $\al_k^n$ be given by \eqref{330}. Then the points $(\al_k^n,u_{n,\al_k^n})$ are nonradial bifurcation points of the curve $\mathcal{K}(n)$. Furthermore, these nonradial solutions bifurcating from $(\al_k^n,u_{n,\al_k^n})$ lie in the space $\mathcal{Z}_n$ and hence are $\mathcal{O}(N - 1)$-invariant.
\end{thm}
\begin{proof}
From Smoller and Wasserman \cite{SW1}, we know that, for any $k\geq1$, the eigenspace $V_k$ of the Laplace-Beltrami operator on $\mms^{N-1}$, which is spanned by $\mathcal{O}(N-1)$ invariant eigenfunctions corresponding to the eigenvalue $\la_k$, is one-dimensional. This implies that
\begin{align}\label{338t}
m(\al_k^n+\delta)-m(\al_k^n-\delta)=1,
\end{align}
if $\delta>0$ is sufficiently small, where $m(\al)$ is the Morse index of the radial solution $u_{n,\al}$ in the space $\mathcal{Z}_n$.

Consider the operator $\mathcal{L}^n(\al,h):(0,+\infty)\times\mathcal{Z}_n\to \mathcal{Z}_n$, where $\mathcal{L}^n(\al,h):=(-\Delta_p)^{-1}(|x|^{\al}| h+\beta_{\al}(n)|^{p_{\al}^*-2}(h+\beta_{\al}(n)))$ and $\displaystyle \beta_\al(n):=U_\al\lr\frac{1}{\var_n}\rr=\frac{C_{N,p,\alpha}\var_n^{\frac{N-p}{p-1}}}
			{(1+\var_n^{\frac{p+\alpha}{p-1}
				})^{\frac{N-p}{p+\al}}}$. Then $\mathcal{L}^n$ is a compact operator for every fixed $\al$ and is continuous with respect to $\al$.

Suppose on the contrary that $(\al_k^n,u_{n,\al_k^n})$ is not a nonradial bifurcation point. Let $T^n(\al,h):=h-\mathcal{L}^n(\al,h)$. From Lemma \ref{lem31}, we know that $u_{n,\al}$ is radially nondegenerate and hence is the unique radial solution to \eqref{31} for any $\al$. Thus there exists $\delta_0>0$ such that, for $\delta\in(0,\delta_0)$ and every $\rho\in(0,\delta_0)$, we have
\begin{align}\label{339t}
&T^n(\al,h)\ne0,\quad\quad \forall \,\, h\in\mathcal{Z}_n\quad\text{such\ that}\quad\|h-u_{n,\al}\|_{\mathcal{Z}_n}\leq \rho \\
&\quad
\text{and}\quad h\ne u_{n,\al},\quad \forall\al\in(\al_k^n-\delta,\al_k^n+\delta).\nonumber
\end{align}
At the same time, we can choose $\delta_0$ such that the interval $[\al_k^n-\delta,\al_k^n+\delta]$ does not contain any other points $\al_j^n$ with $j\ne k$.

Let the set $\mathcal{S}:=\{(\al,h)\in[\al_k^n-\delta,\al_k^n+\delta]\times\mathcal{Z}_n:\|h-u_{n,\al}\|_{\mathcal{Z}_n}\leq\rho\}$. Noting that $T^n(\al,\cdot)$ is a compact perturbation of the identity, we will consider the Leray-Schauder topological degree $\deg(T^n(\al,\cdot),\mathcal{S}_\al,0)$ of $T^n(\al,\cdot)$ on the set $\mathcal{S}_\al:=\{h\in\mathcal{Z}_n\ \text{such\ that}\ (\al,h)\in\mathcal{S}\}$. By \eqref{339t}, there does not exist solutions of $T^n(\al,h)=0$ on $[\al_k^n-\delta,\al_k^n+\delta]\times\partial\mathcal{S}_\al$. From the homotopic invariance of the degree, we have
\begin{align}\label{340t}
\deg(T^n(\al,\cdot),\mathcal{S}_\al,0) \quad\text{is\ constant\ on}\ [\al_k^n-\delta,\al_k^n+\delta].
\end{align}

Since from Theorem \ref{thm38}, operator $I - \mathcal{L}_{h}^n(\al,u_{n,\al}):(0,+\infty)\times\mathcal{Z}_n\to \mathcal{Z}_n$  is invertible for $\al_k^n-\delta$ and $\al_k^n+\delta$, and by Theorem \ref{the36}, we get
\begin{align*}
\deg(T^n(\al_k^n-\delta,\cdot),\mathcal{S}_{\al_k^n-\delta},0)=(-1)^{m(\al_k^n-\delta)}
\end{align*}
and
\begin{align*}
\deg(T^n(\al_k^n+\delta,\cdot),\mathcal{S}_{\al_k^n+\delta},0)=(-1)^{m(\al_k^n+\delta)}.
\end{align*}
From the choice of $\al_k^n$ and of the space $\mathcal{Z}_n$ and \eqref{338t}, we obtain
\begin{align*}
\deg(T^n(\al_k^n+\delta,\cdot),\mathcal{S}_{\al_k^n+\delta},0)=-\deg(T^n(\al_k^n-\delta,\cdot),\mathcal{S}_{\al_k^n-\delta},0),
\end{align*}
which contradicts \eqref{340t}. From Lemma \ref{lem31}, $u_{n,\al}$ is radially nondegenerate and hence is the unique radial solution to \eqref{31} for any $\al$, thus $(\al_k^n,u_{n,\al_k^n})$ is a nonradial bifurcation point and the bifurcating solutions are nonradial. This finishes our proof.
\end{proof}

\begin{thm}\label{thm39}
Let $k$ be an even integer and $\alpha_{k}^n$ be given by equation \eqref{330}. Then, there exist $\left[\frac{N}{2}\right]$ distinct nonradial solutions of \eqref{31} that bifurcate from the pair $(\alpha_{k}^n, u_{n, \alpha_{n}^k})$, where $[x]$ denotes the largest integer that is less than or equal to $x$.
	\end{thm}
	
	\begin{proof}
Let us take the subgroups $\mathcal{G}_l$ of $\mathcal{O}(N)$ into account, which are defined as follows:
\begin{align}\label{347qq}
\mathcal{G}_l = \mathcal{O}(l) \times \mathcal{O}(N - l) \quad\,\, \text{for} \,\, 1 \leq l \leq \left[\frac{N}{2}\right].
\end{align}
From \cite{SW2} (also see \cite{Wi}), we also know that, when $k$ is an even integer, the eigenspace $V_{k,l}$ corresponding to the eigenvalue $\la_k$ of the Laplace-Beltrami operator on $\mms^{N - 1}$, which remains invariant under the action of $\mathcal{G}_l$, is one-dimensional.

First, let $\mathcal{Z}_n^l$ be the subspace of $C_0^{1,\eta}(\overline{B}_{\frac{1}{\var_n}})$ consisting of functions that are invariant under the action of $\mathcal{G}_l$. Then, similar to the expression in \eqref{338t}, we can conclude that
\[
m^l\left(\alpha_{k}^n + \delta\right) - m^l\left(\alpha_{k}^n - \delta\right) = 1
\]
holds when $\delta > 0$ is sufficiently small. Here, $m^l(\alpha)$ represents the Morse index of the radial solution $u_{n,\alpha}$ within the space $\mathcal{Z}_n^l$.

By applying the same reasoning as in the proof of previous Theorem \ref{th36}, we obtain that $(\alpha_{k}^n, u_{n,\alpha_{n}^k})$ is a bifurcation point. Moreover, the solution that bifurcates from this point remains invariant under the action of $\mathcal{G}_l$.

Furthermore, assume that we obtain a solution $h$ that remains invariant under the actions of two groups $\mathcal{G}_{l_1}$ and $\mathcal{G}_{l_2}$, where $l_1 \neq l_2$. According to the results in \cite{SW2}, $h$ have to be radial. However, this is impossible because the radial solutions $u_{n,\alpha}$ are isolated.

Consequently, we can conclude the existence of $\left[\frac{N}{2}\right]$ distinct non-radial solutions of \eqref{31} that bifurcate from the point $(\alpha_{k}^n, u_{n,\alpha_{n}^k})$.
\end{proof}

Let $\Pi_n$ be the closure in $(0,+\infty)\times \mathcal{Z}_n$ of the set of solutions of $T^n(\al,h)=0$ that are different from $u_{n,\al}$, i.e.,
\begin{align}\label{338}
\Pi_n:=\overline{\Big\{(\al,h)\in(0,+\infty)\times \mathcal{Z}_n\ \text{such\ that}\ T^n(\al,h)=0\ \text{with}\ h\ne u_{n,\al}\Big\}},
\end{align}
where $\mathcal{Z}_n$ is defined in \eqref{337} and $T^n(\al,h)$ is defined in Theorems \ref{th36} and \ref{thm39}. If $(\al^n_k,u_{n,\al^n_k})\in\mathcal{K}(n)$ is a
nonradial bifurcation point, then $(\al^n_k,u_{n,\al^n_k})\in \Pi_n$.

We call $\mathcal{C}(\al_k^n)\subset\Pi_n$ the closed connected component of $\Pi_n$ which
contains $(\al^n_k,u_{n,\al^n_k})$ and it is maximal with respect to the inclusion.

\begin{prop}\label{th310}
Let $\al_k^n$ be defined by equation \eqref{330}. If $(\al,h_\al)\in\mathcal{C}(\al_k^n)$, then $h_\al$ is a solution of \eqref{31} corresponding to $\var_n$, in particular, $h_\al>0$ in $B_{\frac{1}{\var_n}}$.
\end{prop}
\begin{proof}
Obviously, $h_\al$ is a solution of \eqref{31}. Let $\mathcal{C}\subset\mathcal{C}(\al_k^n)$ be a subset consisting of the points $(\al,h_\al)$ which are non-negative solutions of $T^n(\al,h)=0$. Note that $(\al^n_k,u_{n,\al^n_k})\in\mathcal{C}$. We aim to prove $\mathcal{C}$ is open and closed in $\mathcal{C}(\al_k^n)$, then $\mathcal{C}=\mathcal{C}(\al_k^n)$ because $\mathcal{C}(\al_k^n)$ is connected. Furthermore, by the maximum principle, we deduce that, if $(\al,h_\al)\in\mathcal{C}$, then either $h_\al>0$ or $h_\al\equiv0$, however, zero is not a solution of \eqref{31}, thus $h_\al>0$.

First, we prove $\mathcal{C}$ is open in $\mathcal{C}(\al_k^n)$. Let $(\al,h_\al)\in\mathcal{C}$ and $(\hat\al, h_{\hat\al})\in\mathcal{C}(\al_k^n)$ such that
$\|h_\al- h_{\hat\al}\|_{\mathcal{Z}_n}<\frac12\beta_{\al}(n)$ and $\frac12\beta_{\al}(n)<\beta_{\hat\al}(n)<2\beta_{\al}(n)$, then
\begin{align}\label{339}
-\Delta_p h_{\hat\al}&=|x|^{\hat\al}| h_{\hat\al}+\beta_{\hat\al}(n)|^{p_{\hat\al}^*-2}(h_{\hat\al}+\beta_{\hat\al}(n))\\
&=|x|^{\hat\al}| h_{\hat\al}+\beta_{\hat\al}(n)|^{p_{\hat\al}^*-2}(h_\al+h_{\hat\al}-h_\al+\beta_{\hat\al}(n))\geq0\quad\text{in}\ B_{\frac{1}{\var_n}}.\nonumber
\end{align}
Then $(\hat\al, h_{\hat\al})\in\mathcal{C}$. Thus $\mathcal{C}$ is open in $\mathcal{C}(\al_k^n)$. On the other hand, if $(\al,h_\al)$ is a point in the closure of $\mathcal{C}$, then there is a sequence of points $(\al_i,h_i)$ in $\mathcal{C}$ that converges to $(\al, h_{\al})\in(0,+\infty)\times C_0^{1,\eta}(\overline B_{\frac{1}{\var_n}})$, as $i\to+\infty$. We have $h_\al$ is a solution of $T^n(\al,h_\al)=0$ and $h_\al\geq0$ in $B_{\frac{1}{\var_n}}$. Thus $\mathcal{C}$ is closed in $\mathcal{C}(\al_k^n)$, and hence $\mathcal{C}=\mathcal{C}(\al_k^n)$. This concludes our proof of Proposition \ref{th310}.
\end{proof}

\begin{thm}\label{th38}
Let $\al_k^n$ be defined by equation \eqref{330} and let $\mathcal{C}(\al_k^n)$ be the closed connected component of $\Pi_n$ which contains $(\al^n_k,u_{n,\al^n_k})$ and it is maximal with respect to the inclusion. Then either\\
(i) $\mathcal{C}(\al_k^n)$ is unbounded in $[0,+\infty)\times\mathcal{Z}_n$, or \\
(ii) there exists $\al_i^n$ with $i\ne k$ such that $(\al^n_i,u_{n,\al^n_i})\in\mathcal{C}(\al_k^n)$, or\\
(iii) $\mathcal{C}(\al_k^n)$ meets $\{0\}\times\mathcal{Z}_n$.
\end{thm}
\begin{proof}
The proof is similar to the global bifurcation result presented in \cite{AM}. For more detailed information, please refer to \cite{G1,R1}.
\end{proof}

\begin{rem}\label{rem312}The results of Proposition \ref{th310} and Theorem \ref{th38} are applicable to every bifurcation point arising from an odd change in the Morse index of the radial solution $u_{n,\alpha}$. After that, by making use of Theorem \ref{thm39}, when $k$ is an even integer, we are able to identify $\left[\frac{N}{2}\right]$ distinct continua of (positive) nonradial solutions that branch off from the point $(\alpha_{k}^n, u_{n,\alpha^{n}_k})$. Furthermore, these continua are global, as they adhere to Theorem \ref{th38}.
\end{rem}

\bigskip

\section{Crucial estimates of approximate solutions}
In this section, we will prove some crucial estimates for the approximate solution of \eqref{31}, where the solution does not necessarily need to be nonradial.

First, by comparing with the radial function $\Phi_n$ and maximum principle, we will overcome the absence of the Green function representation formula and prove the uniformly fast decay estimate for approximate solutions $v_n$ on large ball $B_{\frac{1}{\var_n}}$ in the following proposition.
\begin{prop}\label{pp41}
	Let $\al_n$ and $\var_n$ be sequences such that $\al_n\to\al>0$ and $\var_n\to0$ as $n\to+\infty$.
	Let $v_n$ be a sequence of solutions to
\begin{align}\label{41}
	\begin{cases}
		-\Delta_p v_n=|x|^{\alpha_n}\lr v_n+\beta_{\al_n}(n)\rr^{p_{\al_n}^*-1}, & x\in B_{\frac{1}{\var_n}}, \\
		v_n\geq0, & x\in  B_{\frac{1}{\var_n}},\\
		v_n=0, & x\in \partial B_{\frac{1}{\var_n}},
	\end{cases}
\end{align}
where $\displaystyle \beta_{\al_{n}}(n):=U_{\al_{n}}\lr\frac{1}{\var_n}\rr=\frac{C_{N,p,{\al_{n}}}\var_n^{\frac{N-p}{p-1}}}
			{(1+\var_n^{\frac{p+\alpha_{n}}{p-1}
				})^{\frac{N-p}{p+\al_{n}}}}$. Suppose that $\|v_n\|_X\leq A$ for some uniform positive constant $A$ for every $n$. Then, there exists $C>0$ independent of $n$ such that
\begin{align}\label{42}
		v_n(x)\leq\frac{C}{(1+|x|)^{\frac{N-p}{p-1}}}
\quad\,\, \, for\ every\ x\in\R^N\ and\ for\ every\ n\in\N.
	\end{align}
\end{prop}
\begin{proof}
	Since there exists a uniform constant $C>0$ such that
\begin{align}\label{43}
		\beta_{\al_n}(n)\leq\frac{C}{(1+|x|)^{\frac{N-p}{p-1}}},\qquad\forall \,\, x\in B_{\frac{1}{\var_n}}.
	\end{align}
Note that $\|v_n\|_\gamma\leq\|v_n\|_X\leq A$ and from \eqref{43}, we have
	\begin{align}\label{44}
		|v_n+\beta_{\al_n}(n)|\leq\frac{C_0}{(1+|x|)^{\gamma}},\qquad\forall \,\, x\in B_{\frac{1}{\var_n}}\ \text{with}\ \gamma<\frac{N-p}{p-1}
	\end{align}
for a uniform constant $C_0>0$. Let $\Phi_n(|x|)$ be a radial function satisfying
\begin{align}\label{45}
	\begin{cases}
		-\Delta_p \Phi_n=\frac{C_0^{p_\al^*-1}|x|^{\al_n}}{(1+|x|)^{\gamma(p_{\al_n}^*-1)}}, & x\in B_{\frac{1}{\var_n}}, \\
		\Phi_n=0, & x\in \partial B_{\frac{1}{\var_n}}.
	\end{cases}
\end{align}
	Writing the equation \eqref{45} in radial variable, we can get
\begin{align*}
		-\lr r^{N-1}|\Phi_n'(r)|^{p-2}\Phi_n'(r)\rr'=\frac{C_0^{p_{\al_n}^*-1}r^{\al_n+N-1}}{(1+r)^{\gamma(p_{\al_n}^*-1)}}.
	\end{align*}
Integrating from $0$ to $r$, we get
\begin{align*}
		-r^{N-1}|\Phi_n'(r)|^{p-2}\Phi_n'(r)=C_0^{p_{\al_n}^*-1}\int_0^r\frac{ t^{\al_n+N-1}}{(1+t)^{\gamma(p_{\al_n}^*-1)}}\md t>0,
	\end{align*}
thus $\Phi_n'(r)<0$. Furthermore, we have
\begin{align}\label{46}
		-\Phi_n'(r)=\lr \frac{C_0^{p_{\al_n}^*-1}}{r^{N-1}}\int_0^r\frac{ t^{\al_n+N-1}}{(1+t)^{\gamma(p_{\al_n}^*-1)}}\md t\rr^{\frac1{p-1}}.
\end{align}

In order to ensure
$\frac{ r^{\al_n+N-1}}{(1+r)^{\gamma(p_{\al_n}^*-1)}}\in L^1(\R^+)$, we demand $\gamma>\frac{N(N-p)}{Np-(N-p)}>\frac{N+\al_n}{p_{\al_n}^*-1}$. Thus, from \eqref{46}, we get $-\Phi_n'(r)\leq Cr^{-\frac{N-1}{p-1}}$. Consequently,
$$\Phi_n(r)\leq \frac{C}{r^{\frac{N-p}{p-1}}}.$$
Combining \eqref{41} with \eqref{45}, we have
	\begin{align}\label{47}
	\begin{cases}
		-\Delta_p \Phi_n\geq-\Delta_p v_n & \text{in}\ B_{\frac{1}{\var_n}}, \\
		\Phi_n=v_n &  \text{on}\ \partial B_{\frac{1}{\var_n}}.
	\end{cases}
\end{align}
By the maximum principle (see e.g. \cite{CDL,DLL,LD,LDSMLMSB,LDBS,BS16}), we have $v_n\leq\Phi_n$, thus \eqref{42} holds true.
\end{proof}

In order to show the existence of bifurcation points, by using the Pohozaev identity and maximum principle, and comparing with the radial function $\Psi_n$ and $|\nabla \Psi_n|$, we will overcome the absence of the Green function representation formula and prove $\la=1$ in the following proposition.
\begin{prop}\label{pp42}
	Let $\al_n$ and $\var_n$ be sequences such that $\al_n\to\al>0$ and $\var_n\to0$ as $n\to+\infty$.
	Let $v_n$ be a sequence of solutions of \eqref{41} in $B_{\frac{1}{\var_n}}$, corresponding to the exponent $\al_n$. If $v_n\to U_{\la,\alpha}$ in $X$, then we have $\la=1$.
\end{prop}

\begin{proof}
Multiplying the equation \eqref{41} by $v_n+\beta_{\al_n}(n)$, and integrating on $B_{\frac{1}{\var_n}}$, we have
\begin{align}\label{48p}
		&\int_{B_{\frac{1}{\var_n}}}|\nabla v_n|^p\md x+\int_{\partial B_{\frac{1}{\var_n}}}|\nabla v_n|^{p-1}\beta_{\al_n}(n)\md \mms=\int_{B_{\frac{1}{\var_n}}}|x|^{\alpha_n}\lr v_n+\beta_{\al_n}(n)\rr^{p_{\al_n}^*}\md x.
	\end{align}
On the one hand, multiplying the equation \eqref{41} by $x\cdot\nabla(v_n+\beta_{\al_n}(n))$ and integrating on $B_{\frac{1}{\var_n}}$, we obtain
\begin{align}\label{49p}
		&\quad\frac{N-p}{p}\int_{B_{\frac{1}{\var_n}}}|\nabla v_n|^p\md x+\frac{p-1}{p}\frac{1}{\var_n}\int_{\partial B_{\frac{1}{\var_n}}}|\nabla v_n|^{p}\md \mms\\
&=\frac{N+\al_n}{p^*_{\al_n}}\int_{B_{\frac{1}{\var_n}}}|x|^{\alpha_n}\lr v_n+\beta_{\al_n}(n)\rr^{p_{\al_n}^*}\md x-\frac{1}{p^*_{\al_n}}\lr\frac1{\var_n}\rr^{N+\al_n}\lr
\beta_{\al_n}(n)\rr^{p^*_{\al_n}}\omega_{N-1}\nonumber.
	\end{align}
where $\omega_{N-1}$ represents the $N-1$ dimensional unit sphere area.

By subtracting \eqref{48p}$\times\frac{N-p}{p}$ from \eqref{49p}, from \eqref{41}, we get
\begin{align}\label{410p}
		&\quad\frac{p-1}{p}\frac{1}{\var_n}\int_{\partial B_{\frac{1}{\var_n}}}|\nabla v_n|^{p}\md \mms\\
&=\frac{N-p}{p}\beta_{\al_n}(n)\int_{B_{\frac{1}{\var_n}}}|x|^{\alpha_n}\lr v_n+\beta_{\al_n}(n)\rr^{p_{\al_n}^*-1}\md x-\frac{1}{p^*_{\al_n}}\lr\frac1{\var_n}\rr^{N+\al_n}\lr
\beta_{\al_n}(n)\rr^{p^*_{\al_n}}\omega_{N-1}\nonumber.
	\end{align}
From H\"older inequality, we obtain
\begin{align}\label{411p}
\int_{\partial B_{\frac{1}{\var_n}}}|\nabla v_n|^{p}\md \mms
&\geq \lr\int_{\partial B_{\frac{1}{\var_n}}}|\nabla v_n|^{p-1}\md \mms\rr^{\frac{p}{p-1}}\lr\int_{\partial B_{\frac{1}{\var_n}}}\md \mms\rr^{-\frac{1}{p-1}}\\
&=\lr\int_{B_{\frac{1}{\var_n}}}|x|^{\alpha_n}\lr v_n+\beta_{\al_n}(n)\rr^{p_{\al_n}^*-1}\md x\rr^{\frac{p}{p-1}}\lr\int_{\partial B_{\frac{1}{\var_n}}}\md \mms\rr^{-\frac{1}{p-1}}\nonumber.
\end{align}
By \eqref{410p} and \eqref{411p}, we obtain
\begin{align}\label{412p}
&\quad \frac{N-p}{p}\beta_{\al_n}(n)\int_{B_{\frac{1}{\var_n}}}|x|^{\alpha_n}\lr v_n+\beta_{\al_n}(n)\rr^{p_{\al_n}^*-1}\md x-\frac{1}{p^*_{\al_n}}\lr\frac1{\var_n}\rr^{N+\al_n}\lr\beta_{\al_n}(n)\rr^{p^*_{\al_n}}\omega_{N-1}\\
&\geq \frac{p-1}{p}\frac{1}{\var_n}
\lr\int_{B_{\frac{1}{\var_n}}}|x|^{\alpha_n}\lr v_n+\beta_{\al_n}(n)\rr^{p_{\al_n}^*-1}\md x\rr^{\frac{p}{p-1}}\lr\int_{\partial B_{\frac{1}{\var_n}}}\md \mms\rr^{-\frac{1}{p-1}}\nonumber.
\end{align}
Let $n\to+\infty$, we have
\begin{equation}\label{417ppp}
\frac{N-p}{p}\int_{\R^N}|x|^{\alpha} U_{\la,\al}^{p_{\al}^*-1}\md x\geq\frac{p-1}{p}
\lr\int_{\R^N}|x|^{\alpha} U_{\la,\al}^{p_{\al}^*-1}\md x\rr^{\frac{p}{p-1}}\omega_{N-1}^{-\frac1{p-1}}\nonumber.
\end{equation}
Thus, we can get
\begin{align}\label{412pp}\la\geq1.\end{align}
On the other hand, since $v_n\to U_{\la,\alpha}$ in $X$ as $n\to\infty$, then $\forall\delta>0$, $\exists\ N_0$ such that, $\forall\ n>N_0,$
$$\|v_n-U_{\la,\alpha}\|_{X}<\delta.$$
Moreover,
$$v_n\leq U_{\la,\alpha}+\frac{\delta}{(1+|x|)^\gamma},$$
where $\gamma\in\left(\frac{N(N-p)}{Np-(N-p)},\frac{N-p}{p-1}\right)$. Let us consider a radial function $\Psi_n(|x|)$ satisfying
	\begin{align}\label{413p}
	\begin{cases}
		-\Delta_p \Psi_n(|x|)=|x|^{\alpha_n}\lr U_{\la,\alpha}+\frac{\delta}{(1+|x|)^\gamma}+\beta_{\al_n}(n)\rr^{p_{\al_n}^*-1} &  \text{in}\ B_{\frac{1}{\var_n}}, \\
		\Psi_n(|x|)=0 &  \text{on}\ \partial B_{\frac{1}{\var_n}}.
	\end{cases}
\end{align}

By \eqref{413p}, we get
\begin{align}\label{414p}
	\begin{cases}
		-\Delta_p \Psi_n\geq-\Delta_p v_n  & \text{in}\ B_{\frac{1}{\var_n}}, \\
		\Psi_n=v_n &  \text{on}\ \partial B_{\frac{1}{\var_n}}.
	\end{cases}
\end{align}
Moreover, we have $\nabla(\Psi_n-v_n)\cdot\nabla v_n\geq0$ and $\nabla(\Psi_n-v_n)\cdot\nabla \Psi_n\geq0$ on $\partial B_{\frac{1}{\var_n}}$, and hence
$$|\nabla\Psi_n|\geq|\nabla v_n| \qquad \text{on}\quad \partial B_{\frac{1}{\var_n}}.$$
According to the Pohozaev identity, we have
\begin{align}\label{416p}
&\quad \frac{N-p}{p}\beta_{\al_n}(n)\int_{B_{\frac{1}{\var_n}}}|x|^{\alpha_n}\lr v_n+\beta(n)\rr^{p_{\al_n}^*-1}\md x-\frac{1}{p^*_{\al_n}}\lr\frac1{\var_n}\rr^{N+\al_n}\lr\beta_{\al_n}(n)\rr^{p^*_{\al_n}}\omega_{N-1}\\
&=\frac{p-1}{p}\frac{1}{\var_n}\int_{\partial B_{\frac{1}{\var_n}}}|\nabla v_n|^{p}\md \mms\nonumber\\
&\leq\frac{p-1}{p}\frac{1}{\var_n}\int_{\partial B_{\frac{1}{\var_n}}}|\nabla\Psi_n|^{p}\md \mms\nonumber\\
&= \frac{p-1}{p}\frac{1}{\var_n}
\lr\int_{B_{\frac{1}{\var_n}}}|x|^{\alpha_n}\lr U_{\la,\alpha}+\frac{\delta}{(1+|x|)^\gamma}+\beta_{\al_n}(n)\rr^{p_{\al_n}^*-1}\md x\rr^{\frac{p}{p-1}}\lr\int_{\partial B_{\frac{1}{\var_n}}}\md \mms\rr^{-\frac{1}{p-1}}\nonumber.
\end{align}
Letting $\delta\to0, n\to+\infty$, from Lebesgue's dominated convergence theorem, we get
\begin{equation}\label{417p}
\frac{N-p}{p}\int_{\R^N}|x|^{\alpha}U_{\la,\al}^{p_{\al}^*-1}\md x\leq\frac{p-1}{p}
\lr\int_{\R^N}|x|^{\alpha}U_{\la,\al}^{p_{\al}^*-1}\md x\rr^{\frac{p}{p-1}}\omega_{N-1}^{-\frac1{p-1}}\nonumber.
\end{equation}
Thus, we infer that
 \begin{align}\label{418pp}\la\leq1.\end{align}
From \eqref{412pp} and \eqref{418pp}, we obtain $\la=1.$ This concludes our proof of Proposition \ref{pp42}.
\end{proof}

Furthermore, based on Proposition \ref{pp41}, we can obtain an upper bound estimate of the $\nabla v_n(x)$.
\begin{prop}\label{pro43}
Let $\al_n$ and $\var_n$ be sequences such that $\al_n\to\al>0$ and $\var_n\to0$ as $n\to+\infty$.
Let $v_n$ be a sequence of solutions of \eqref{41} in $B_{\frac{1}{\var_n}}$, corresponding to the exponent $\al_n$. If $\|v_n\|_X\leq A$ for some uniform positive constant $A$ for every $n$, then there exists $C>0$ independent of $n$ such that
	\begin{align}\label{q421}
		|\nabla v_n(x)|\leq\frac{C}{(1+|x|)^{\frac{N-1}{p-1}}}
		\ \ for\ every\ x\in B_{\frac{1}{\var_n}}\ and\ for\ every\ n\in\N.
	\end{align}
\end{prop}
\begin{proof}
For all $ 1<|y|<\frac{1}{\var_n}$, $\forall 1<R<\frac{1}{\var_n |y|}$, define $v_n^R(y):=R^{\frac{N-p}{p-1}}v_n(Ry)$. From Proposition \ref{pp41}, we get $v_n(x)\leq\frac{C_{0}}{(1+|x|)^{\frac{N-p}{p-1}}}$ for some uniform constant $C_{0}>0$, and hence
\begin{align}\label{q422}
	 v_n^R(y)\leq R^{\frac{N-p}{p-1}}\frac{C_{0}}{|Ry|^{\frac{N-p}{p-1}}}\leq C_{0}.
\end{align}
Moreover, $v_n^R(y)$ satisfies
\begin{align}\label{q423}
	-\Delta_p v_n^R(y)&=R^{N+\al_n}|y|^{\alpha_n}\lr v_n(Ry)+\beta_{\al_n}(n)\rr^{p_{\al_n}^*-1} \\
&\leq C_1 R^{N+\al_n}|y|^{\alpha_n}\left(R|y|\right)^{-\frac{N-p}{p-1}(p_{\al_n}^*-1)}\leq C_1, \qquad \forall \,\, 1<|y|<\frac{1}{\var_n},  \nonumber
\end{align}
where \(C_1\) is independent of $n$, \(R\) and \(y\). Consequently, from the gradient regularity estimate in \cite{T}, we have
$$\|\nabla_y v_n^R(y)\|_{L^\infty(B(0,4)\setminus B(0,2))}\leq C_2,$$
where \(C_2\) is independent of $n$ and \(R\).

For all $1<|x|<\frac{1}{\var_n}$, let $x=Ry$, where $2<|y|<4$, thus we get
$$|\nabla_x v_n(x)|=R^{\frac{1-N}{p-1}}|\nabla_y v_n^R(y)|\leq C_2R^{\frac{1-N}{p-1}}\leq\frac{4^{\frac{N-1}{p-1}}C_2}{|x|^{\frac{N-1}{p-1}}}.$$
This gives the desired estimate \eqref{q421}.
\end{proof}

Given $\Omega\subset \R^N$, $q\geq 1$, and a non-negative locally integrable function $\omega_0:\R^n\to \R$, we define the Banach space $L^q(\Omega;\omega_0)$ as the space of measurable functions $\varphi:\Omega\to \R$ such that the norm
$$\|\varphi \|_{L^q(\Omega;\,\omega_0)}:=\left(\int_{\Omega} |\varphi|^q \,\omega_0(x)\, \md x\right)^{\frac 1 q}<+\infty.$$
Given $\omega_1 \in L^1_{\rm loc}(\R^n\setminus \{0\})$ non-negative,
we denote by $C^1_{c,0}(\R^N)$ the space of compactly supported functions of class $C^1$ that are constant in a neighborhood of the origin, and we
define $\dot{W}^{1,q}(\mathbb R^N;\omega_1)$ as the closure of $C^1_{c,0}(\R^N)$ with respect to the norm
$$\|\varphi\|_{\dot{W}^{1,q}(\mathbb R^N;\,\omega_1)}:=\left(\int_{\mathbb R^N}|D\varphi|^q \,\omega_1(x) \,\md x\right)^{\frac 1 q}.$$
Inspired by Figalli and Zhang \cite{FZ}, we can prove the following weighted Sobolev embedding $\dot{W}^{1,2}(\R^N;\,|\nabla U_\al|^{p-2})\hookrightarrow L^2(\R^N;\,|x|^\beta U_\al^{p_\al^*-2})$.
\begin{prop}\label{pp43}
Assume $1<p<N$ and $\al>0$. If $\varphi\in \dot{W}^{1,2}(\R^N;\,|\nabla U_\al|^{p-2})\cap L^2(\R^N;\,|x|^\beta U_\al^{p_\al^*-2})$,
	$\beta\in\left[\frac{p\al-2\al-p}{p-1}, \frac{p+p\al}{p-1}\right]\bigcap\left(\frac{p\al+2p-N}{p-1},+\infty\right)$, then we have
	\begin{align}	
		\int_{\R^N}	|x|^\beta U_\al^{p_\al^*-2}|\varphi|^2\md x\leq
		C(N,p,\al)\int_{\R^N}	|\nabla U_\al|^{p-2} |\nabla\varphi|^{2}\md x.\label{aa436}
		\end{align}
\end{prop}

\begin{proof}
To prove \eqref{aa436}, we can assume by approximation that $\varphi\in C_{c}^1(\R^N)$. From Fubini's theorem and using polar coordinates, from $\beta>\frac{p\al+2p-N}{p-1},$ we obtain
	\begin{align}	
		&\quad\int_{\R^N}	|x|^\beta U_\al^{p_\al^*-2}|\varphi|^2\md x\label{a436}\\ &\leq  C(N,p,\al)
		\int_{\Sm^{N-1}}\int_0^{+\infty}r^{\beta+N-1}\lr1+r^{\frac{p+\al}{p-1}}\rr^{-\frac{N-p}{p+\al}(\frac{p(N+\al)}{N-p}-2)}|\varphi(r\theta)|^2\md r\md\theta\nonumber\\
		&\leq C(N,p,\al)\int_{\Sm^{N-1}}\int_0^{+\infty}r^{\beta+N-1}\lr1+r^{\frac{p+\al}{p-1}}\rr^{-\frac{N-p}{p+\al}(\frac{p(N+\al)}{N-p}-2)}\int_r^{+\infty}|\varphi(t\theta)||\nabla\varphi(t\theta)|\md t\md r\md\theta\nonumber\\
		&\leq C(N,p,\al) \int_{\Sm^{N-1}}\int_0^{+\infty}\int_0^{t}|\varphi(t\theta)||\nabla\varphi(t\theta)|r^{\beta+N-1}\lr1+r^{\frac{p+\al}{p-1}}\rr^{-\frac{N-p}{p+\al}(\frac{p(N+\al)}{N-p}-2)}\md r\md t\md\theta\nonumber\\
		&\leq C(N,p,\al)\int_{\Sm^{N-1}}\int_0^{+\infty}|\varphi(t\theta)||\nabla\varphi(t\theta)|t^{\beta+N}\lr1+t^{\frac{p+\al}{p-1}}\rr^{-\frac{N-p}{p+\al}(\frac{p(N+\al)}{N-p}-2)}\md t\md\theta\nonumber\\
		&\leq C(N,p,\al)\lr\int_{\Sm^{N-1}}\int_0^{+\infty}|\varphi(t\theta)|^2t^{\beta+N-1}\lr1+t^{\frac{p+\al}{p-1}}\rr^{-\frac{N-p}{p+\al}(\frac{p(N+\al)}{N-p}-2)}\md t\md\theta\rr^{\frac12\nonumber}\\
		&\quad\times\lr\int_{\Sm^{N-1}}\int_0^{+\infty}|\nabla\varphi(t\theta)|^2t^{\beta+N+1}\lr1+t^{\frac{p+\al}{p-1}}\rr^{-\frac{N-p}{p+\al}(\frac{p(N+\al)}{N-p}-2)}\md t\md\theta\rr^{\frac12}.\nonumber
	\end{align}	
Note that, for $|x|\in(0,1]$, if $\beta+2\geq\frac{(1+\al)(p-2)}{p-1}$, then $$|x|^{\beta+2}\sim|x|^{\beta+2-\frac{(1+\al)(p-2)}{p-1}}|\nabla U_\al|^{p-2}\leq C|\nabla U_\al|^{p-2};$$
 and for $|x|\in(1,+\infty)$, if $ \beta+2\leq\frac{3p+p\al-2}{p-1}$, then $$|x|^{\beta+2}\lr1+|x|^{\frac{p+\al}{p-1}}\rr^{-\frac{N-p}{p+\al}(\frac{p(N+\al)}{N-p}-2)}\sim|x|^{\beta+2-\frac{3p+p\al-2}{p-1}}|\nabla U_\al|^{p-2}\leq C|\nabla U_\al|^{p-2}.$$
Hence \eqref{aa436} follows from \eqref{a436}. This finishes our proof.
\end{proof}
\begin{rem}
In \cite{HY}, without any symmetry assumption on $V$ or $f$, the authors obtained the following \(L^p\)-Hardy inequality:
\begin{align}\label{hy1}
\int_{\Omega} V |\nabla u|^q \mathrm{d}x \geq - \int_{\Omega} \frac{\mathrm{div}(V|\nabla f|^{q-2}\nabla f)}{f^{q-1}}|u|^q\md x.
\end{align}
The estimate \eqref{aa436} in Proposition \ref{pp43} may also be derived from \eqref{hy1} by suitable choice of the function $f$.
\end{rem}

By applying the weighted Sobolev inequality \eqref{aa436} in Proposition \ref{pp43} with $\alpha=\alpha_n$, we overcome the nonlinear virtue of the $p$-Laplacian $\Delta_p$ and first prove the following uniformly decay estimate in integral form for $w_n$.
\begin{prop}\label{ppp43}
If $p\in(1,N)$,  $\beta\in(\max\{\al_n,\frac{p\al_n-2\al_n-p}{p-1},\frac{p\al_n+2p-N}{p-1}\},\frac{p+p\al_n}{p-1}]$, then there exists some uniform constant $C>0$ independent of $n$ such that
	\begin{align}	
		\int_{B_{\frac{1}{\var_n}}}|x|^\beta U_{\al_n}^{p_{\al_n}^*-2}|w_n|^2\md x&\leq
		C.\label{aa438}
	\end{align}
\end{prop}

\begin{proof}
Taking $\varphi=w_n:=\frac{u_{n}-v_{n}}{\|u_n-v_n\|_{L^{\infty}(\mathbb{R}^{N})}}$ in Proposition \ref{pp43}, by $|w_n|\leq1$, we get
	\begin{align}	
		\int_{B_{\frac{1}{\var_n}}}|x|^\beta U_{\al_n}^{p_{\al_n}^*-2}|w_n|^2\md x&\leq
		C\int_{B_{\frac{1}{\var_n}}}|\nabla w_n|^2|\nabla U_{\al_n}|^{p-2}\md x\nonumber\\
		&\leq C\int_{B_{\frac{1}{\var_n}}}|\nabla w_n|^2\lr|\nabla u_n|^{p-2}+|\nabla v_n|^{p-2}\rr\md x\nonumber\\
		&\leq C\frac{1}{\|u_n-v_n\|^2_{L^{\infty}(\mathbb{R}^{N})}}\int_{B_{\frac{1}{\var_n}}}\nabla (u_n-v_n)\lr|\nabla u_n|^{p-2}\nabla u_n-|\nabla v_n|^{p-2}\nabla v_n\rr\md x\nonumber\\
		&\leq C\int_{B_{\frac{1}{\var_n}}}|x|^{\al_n} \xi_n(x)|w_n|^2\md x\nonumber\\
		&= C\int_{B_1}|x|^{\al_n} \xi_n(x)|w_n|^2\md x+C\int_{B_{\frac{1}{\var_n}}\setminus B_1}|x|^{\al_n} \xi_n(x)|w_n|^2\md x\nonumber\\
		&\leq C+C\int_{B_{\frac{1}{\var_n}}\setminus B_1}|x|^{\al_n} \xi_n(x)|w_n|^{2-\delta}\md x \label{a438}\\
		&\leq C+C\lr\int_{B_{\frac{1}{\var_n}}}|x|^\beta U_{\al_n}^{p_{\al_n}^*-2}|w_n|^2\md x\rr^{\frac{2-\delta}{2}}\nonumber\\
		&\quad\times\lr\int_{B_{\frac{1}{\var_n}}\setminus B_1}|x|^{\frac{2\al_n}{\delta}}|\xi_n(x)|^{\frac2\delta}|x|^{-\frac{\beta(2-\delta)}{\delta}}U_{\al_n}^
		{-(p_{\al_n}^*-2)\frac{2-\delta}{\delta}}\md x\rr^{\frac{\delta}{2}}\nonumber
	\end{align}
for $\delta\in(0,2)$, where
	\begin{align*}
		\xi_n(x):&=(p_{\al_n}^*-1)\int_0^1\lr t u_n+(1-t) v_n+\beta_{\al_n}(n)\rr^{p_{\al_n}^*-2}\md t.
	\end{align*}
If we can find $\delta\in(0,2)$ such that
	\begin{align}\label{ab438}
		\int_{B_{\frac{1}{\var_n}}\setminus B_1}|x|^{\frac{2\alpha_n}{\delta}}|\xi_n(x)|^{\frac2\delta}|x|^{-\frac{\beta(2-\delta)}{\delta}}U_{\al_n}^
		{-(p_{\al_n}^*-2)\frac{2-\delta}{\delta}}\md x<+\infty,
	\end{align}
then we can get \eqref{aa438}. In fact, from Proposition \ref{pp41}, we have
	\begin{align*}
		&\quad\int_{B_{\frac{1}{\var_n}}\setminus B_1}|x|^{\frac{2\alpha_n}{\delta}}|\xi_n(x)|^{\frac2\delta}|x|^{-\frac{\beta(2-\delta)}{\delta}}U_{\al_n}^
		{-(p_{\al_n}^*-2)\frac{2-\delta}{\delta}}\md x\nonumber\\
		&\leq C\int_{B_{\frac{1}{\var_n}}\setminus B_1}|x|^{\frac{2\alpha_n}{\delta}}\lr1+|x|\rr^{-\frac{pN+p{\al_n}-2N+2p}{p-1}\frac2\delta}|x|^{-\frac{\beta(2-\delta)}{\delta}}(1+|x|)^{\frac{pN+p{\al_n}-2N+2p}{p-1}\frac{2-\delta}{\delta}}\md x\nonumber\\
		&\leq C\int_{B_{\frac{1}{\var_n}}\setminus B_1}|x|^{\frac{2\alpha_n}{\delta}-\frac{\beta(2-\delta)}{\delta}}\lr1+|x|\rr^{-\frac{pN+p\al_n-2N+2p}{p-1}}\md x\nonumber\\
		&\leq C\int_{1}^{+\infty}\frac{r^{\frac{2\alpha_n}{\delta}-\frac{\beta(2-\delta)}{\delta}}r^{N-1}}{\lr1+r\rr^{\frac{pN+p\al_n-2N+2p}{p-1}}}\md r.
	\end{align*}
Since $\beta\in(\max\{\al_n,\frac{p\al_n-2\al)n-p}{p-1},\frac{p\al_n+2p-N}{p-1}\},\frac{p+p\al_n}{p-1}]$, taking $\delta\in\left(0,\frac{2(\beta-\al_n)}{\beta+\frac{N-p\al_n-2p}{p-1}}\right)$, we have $$\frac{pN+p\al_n-2N+2p}{p-1}-\lr\frac{2\alpha_n}{\delta}-\frac{\beta(2-\delta)}{\delta}+N-1\rr>1.$$
Therefore, we can take $\delta\in\left(0,\min\left\{\frac{2(\beta-\al_n)}{\beta+\frac{N-p\al_n-2p}{p-1}},2\right\}\right)$, then from \eqref{a438} and Young inequality,  we get
\begin{align}	
	\int_{B_{\frac{1}{\var_n}}}|x|^\beta U_{\al_n}^{p_{\al_n}^*-2}|w_n|^2\md x&\leq C+C\lr\int_{B_{\frac{1}{\var_n}}}|x|^\beta U_{\al_n}^{p_{\al_n}^*-2}|w_n|^2\md x\rr^{\frac{2-\delta}{2}}\nonumber\\
	&\leq C+\frac\delta2C^{\frac{2}{\delta}}+\frac{2-\delta}{2}\int_{B_{\frac{1}{\var_n}}}|x|^\beta U_{\al_n}^{p_{\al_n}^*-2}|w_n|^2\md x,\nonumber
\end{align}
and hence \eqref{aa438} holds. This concludes our proof of Proposition \ref{ppp43}.
\end{proof}

In order to overcome the difficulties caused by the unavailability of Kelvin type transforms and the absence of the Green function representation formula, based on \eqref{aa438}, through a De Giorgi-Moser-Nash iteration argument, we can finally establish the uniformly fast decay estimate \eqref{412} for $w_n$ in the following proposition, without using Kelvin type transforms and the Green function representation formula.
\begin{prop}\label{ppq43}
Assume $1<p<N$. Let $\al_n$ and $\var_n$ be sequences such that $\al_n\to\al>0$ and $\var_n\to0$ as $n\to+\infty$. Let $v_n$ be a sequence of nonradial solutions of \eqref{41} in $B_{\frac{1}{\var_n}}$ related to the exponent $\al_n$, and let $w_n:=\frac{u_n-v_n}{\|u_n-v_n\|_{L^{\infty}(\mathbb{R}^{N})}}$, where $u_n$ is the radial solution of \eqref{41} related to the exponent $\al_n$ given by \eqref{32}. If $v_n\to U_\alpha$ in $X$, then there exists a uniform constant $C_0>0$ independent of $n$ such that
\begin{align}\label{412}
		|w_n(y)|\leq\frac{C_0}{(1+|y|)^{\frac{N-p}{2(p-1)}}}, \qquad \forall \,\, y\in \mathbb{R}^{N}.
\end{align}
\end{prop}

\begin{proof}
We will carry out our proof by discussing the following two different cases.	
	
\medskip

$\bullet$ Case 1: $1\leq |y|<\frac{1}{2\var_n}$. \\
From $w_n=0$ on $x\in\partial B_{\frac{1}{\var_n}}$, by taking $\al=\al_n$ and $\beta=\frac{p+p\al_n}{p-1}$ in Proposition \ref{pp43}, we have
	\begin{align*}	
		\int_{B_{\frac{1}{\var_n}}}|x|^{\frac{p+p\al_n}{p-1}} U_{\al_n}^{p_{\al_n}^*-2}|w_n|^2\md x&\leq C.
	\end{align*}
Therefore, taking $R=\frac{5}{12}|y|$, we have
\begin{align}\label{aaq427}
y\in  B_{3R}\setminus B_{2R}\subset B_{4R}\setminus B_{R}\subset B_{\frac{1}{\var_n}},
\end{align}
and
\begin{align}	\label{aaa427q}
	\int_{B_{4R}\setminus B_{R}}|w_n|^2\md x&\leq CR^{\frac{Np+p-2N}{p-1}}.
\end{align}

Define $f(x)=|x|^{p-2}x$, $f_j=|x|^{p-2}x_j$, thus
	$$(f_i)_{x_j}=(p-2)|x|^{p-4}x_ix_j+|x|^{p-2}\delta_{ij}:=A_{ij},$$	
and
	\begin{align}
		\sum_{i,j=1}^{N}A_{ij}\xi_i\xi_j&=\sum_{i,j=1}^{N}\lr(p-2)|x|^{p-4}x_ix_j+|x|^{p-2}\delta_{i,j}\rr\xi_i\xi_j\label{432qq}\\
		&=\sum_{i,j=1}^{N}(p-2)|x|^{p-4}x_i\xi_ix_j\xi_j+|x|^{p-2}|\xi|^2\nonumber\\
		&=|x|^{p-2}|\xi|^2\lr1+(p-2)\left|\frac{x}{|x|}\cdot\frac{\xi}{|\xi|}\right|^2\rr\nonumber\\
		&\geq
		\begin{cases}
			|x|^{p-2}|\xi|^2,\quad & p>2,\nonumber\\
			(p-1)|x|^{p-2}|\xi|^2,\quad &1<p\leq2.
		\end{cases}
	\end{align}
	Note that
	\begin{align*}
		&\quad|\nabla u_n|^{p-2}\nabla u_n-|\nabla v_n|^{p-2}\nabla v_n\\
		&=\int_0^1(p-2)|t\nabla u_n+(1-t)\nabla v_n|^{p-4}(tu_n+(1-t)v_n)_{x_i}
		(tu_n+(1-t)v_n)_{x_j}\\
		&\quad+|t\nabla u_n+(1-t)\nabla v_n|^{p-2}\delta_{ij}\md t \, (u_n-v_n)_{x_i},
	\end{align*}
	let us define
	\begin{align*}
		a_{ij}^n(x):&=\int_0^1(p-2)|t\nabla u_n+(1-t)\nabla v_n|^{p-4}(t u_n+(1-t) v_n)_{x_i}
		(tu_n+(1-t) v_n)_{x_j}\\
		&\quad+|t\nabla u_n+(1-t)\nabla v_n|^{p-2}\delta_{i,j}\md t.
	\end{align*}
	Then
	\begin{align*}
		-\text{div}(|\nabla u_n|^{p-2}\nabla u_n-|\nabla v_n|^{p-2}\nabla v_n  )=|x|^{\al_n}\lr(u_n+\beta_{\al_n}(n))^{p_{\al_n}^*-1}-
		(v_n+\beta_{\al_n}(n))^{p_{\al_n}^*-1}\rr,
	\end{align*}
	it is equal to
	\begin{align*}
		-\lr a_{ij}^n(x)(u_n-v_n)_{x_j}\rr_{x_i}=|x|^{\al_n}\lr(u_n+\beta_{\al_n}(n))^{p_{\al_n}^*-1}-
		(v_n+\beta_{\al_n}(n))^{p_{\al_n}^*-1}\rr.
	\end{align*}	
	From the definition of $w_n$, we get	
	\begin{align*}
		-\lr a_{i,j}^n(x)(w_n)_{x_j}\rr_{x_i}&=|x|^{\al_n}\int_0^1(p_{\al_n}^*-1)\lr tu_n+(1-t)v_n+\beta_{\al_n}(n)\rr^{p_{\al_n}^*-2}\md t \frac{u_n-v_n}{\|u_n-v_n\|_{L^{\infty}(\mathbb{R}^{N})}}\\
		&=:|x|^{\al_n}\zeta_n(x)w_n,
	\end{align*}
where
$$\zeta_n(x):=\int_0^1(p_{\al_n}^*-1)\lr tu_n+(1-t)v_n+\beta_{\al_n}(n)\rr^{p_{\al_n}^*-2}\md t.$$	
Thus $w_n$ satisfies the equation
	\begin{align*}
		\begin{cases}
			-\lr a_{ij}^n(x)(w_n)_{x_j}\rr_{x_i}=|x|^{\al_n}\zeta_n(x)w_n,  &x\in B_{\frac{1}{\var_n}},\\
			w_n(x)=0,  &x\in\partial B_{\frac{1}{\var_n}}.
		\end{cases}
	\end{align*}	
	On the one hand,
	$$|a_{ij}^n(x)(w_n)_{x_j}|\leq(p-1)\int_0^1|(1-t)\nabla u_n+t\nabla v_n|^{p-2}\md t |\nabla w_n|.$$
	On the other hand, from \eqref{432qq}, we have
	\begin{align*}
		a_{ij}^n(x)(w_n)_{x_j}(w_n)_{x_i}
		\geq
		\begin{cases}
			\displaystyle \int_0^1|(1-t)\nabla u_n+t\nabla v_n|^{p-2}\md t |\nabla w_n|^2,& p\geq2,\\
			\displaystyle (p-1)\int_0^1|(1-t)\nabla u_n+t\nabla v_n|^{p-2}\md t |\nabla w_n|^2, &1<p<2.
		\end{cases}
	\end{align*}
	Let
	$$\eta_n(x):=\int_0^1|(1-t)\nabla u_n+t\nabla v_n|^{p-2}\md t.$$
We have the following estimates on $\eta_n(x)$ and $\zeta_n(x)$:
\begin{align}\label{435}
	\eta_n(x)
	\geq
	\begin{cases}
		C(p) |\nabla u_n|^{p-2},& p\geq2,\\
			\lr|\nabla u_n|+|\nabla v_n|\rr^{p-2}, &1<p<2,
	\end{cases}
\end{align}
\begin{align}\label{436}
	\eta_n(x)
	\leq
	\begin{cases}
	\lr|\nabla u_n|+|\nabla v_n|\rr^{p-2},& p\geq2,\\
		C(p) |\nabla u_n|^{p-2}, &1<p<2,
	\end{cases}
\end{align}
and
\begin{align}\label{437}
	\zeta_n(x)\leq
	\begin{cases}
		C|u_n+\beta_{\al_n}(n)+v_n|^{p_{\al_n}^*-2},& p\geq\frac{2N}{N+2+\al_n},\\
		C|u_n+\beta_{\al_n}(n)|^{p_{\al_n}^*-2}, &1<p<\frac{2N}{N+2+\al_n}.
	\end{cases}
\end{align}
In fact, if $p\in[2,N)$ and $|\nabla u_n|\geq|\nabla(v_n-u_n)|$, we have
	\begin{align*}
		\eta_n(x)&\geq\int_0^1\Big||\nabla u_n|-t|\nabla (v_n-u_n)|\Big|^{p-2}\md t\\&\geq\int_0^{\frac12}\lr|\nabla u_n|-t|\nabla (v_n-u_n)|\rr^{p-2}\md t\\&\geq\int_0^{\frac12}\lr|\nabla u_n|-\frac12|\nabla (v_n-u_n)|\rr^{p-2}\md t\\&\geq\left(\frac{1}{2}\right)^{p-1}|\nabla u_n|^{p-2},
	\end{align*}
	and if $|\nabla u_n|<|\nabla(v_n-u_n)|$, we obtain
	\begin{align*}
	\eta_n(x)&\geq\int_0^1\Big||\nabla u_n|-t|\nabla (v_n-u_n)|\Big|^{p-2}\md t\\
	&=\frac{1}{p-1}\frac{|\nabla u_n|^{p-1}+(|\nabla(v_n-u_n)|-|\nabla u_n|)^{p-1}}{|\nabla(v_n-u_n)|}\\
	&\geq\frac{2^{2-p}}{p-1}|\nabla(v_n-u_n)|^{p-2}\\
	&\geq\frac{2^{2-p}}{p-1}|\nabla u_n|^{p-2}.
	\end{align*}
If $p\in(1,2)$, we get
\begin{align*}
	\eta_n(x)&=\int_0^1|(1-t)\nabla u_n+t\nabla v_n|^{p-2}\md t\\
	&\geq\lr|\nabla u_n|+|\nabla v_n|\rr^{p-2}.
\end{align*}
Thus \eqref{435} holds true. If $p\in[2,N)$, we have
\begin{align*}
	\eta_n(x)&=\int_0^1|(1-t)\nabla u_n+t\nabla v_n|^{p-2}\md t\\
	&\leq\lr|\nabla u_n|+|\nabla v_n|\rr^{p-2}.
\end{align*}
If $p\in(1,2)$ and $|\nabla u_n|\geq|\nabla(v_n-u_n)|$, we obtain
\begin{align*}
\eta_n(x)&=\int_0^1|\nabla u_n+t\nabla (v_n-u_n)|^{p-2}\md t\\
&\leq\int_0^1\Big||\nabla u_n|-t|\nabla (v_n-u_n)|\Big|^{p-2}\md t\\&=\frac{1}{p-1}\frac{|\nabla u_n|^{p-1}-(|\nabla u_n|-|\nabla(v_n-u_n)|)^{p-1}}{|\nabla(v_n-u_n)|}\\
&\leq|\nabla u_n|^{p-2},
\end{align*}
if $|\nabla u_n|<|\nabla(v_n-u_n)|$, then
\begin{align*}
	\eta_n(x)&\leq\int_0^1\Big||\nabla u_n|-t|\nabla (v_n-u_n)|\Big|^{p-2}\md t\\&=\frac{1}{p-1}\frac{|\nabla u_n|^{p-1}+(|\nabla(v_n-u_n)|-|\nabla u_n|)^{p-1}}{|\nabla(v_n-u_n)|}\\
	&\leq\frac{2^{2-p}}{p-1}|\nabla u_n|^{p-2}.
\end{align*}
Thus \eqref{436} holds true. The estimate \eqref{437} follows from \eqref{436} directly. Now, we take
	\begin{align*}
		F(\overline w_n)
		=\begin{cases}
			\overline w_n^q,& \overline w_n\leq l,\\
			q l^{q-1}\overline w_n-(q-1) l^q, &\overline w_n> l
		\end{cases}
	\end{align*}
	for some $q,l\in\R^+$, and
	\begin{align*}
		G(w_n)
		=\text{sgn}(w_n)F(\overline w_n)F'(\overline w_n),
	\end{align*}
	where $\overline w_n:=|w_n|$. Thus
	\begin{align*}
		F'(\overline w_n)
		=\begin{cases}
			q\overline w_n^{q-1},& \overline w_n\leq l,\\
			q l^{q-1}, &\overline w_n> l
		\end{cases}
	\end{align*}
	and
	\begin{align*}
		G'(w_n)
		=\begin{cases}
			q^{-1}(2q-1)\lr F'(\overline w_n)\rr^2,& \overline w_n\leq l,\\
			\lr F'(\overline w_n)\rr^2, &\overline w_n> l.
		\end{cases}
	\end{align*}

Since $B_{4R}\subset B_{\frac{1}{\var_n}}$, for any given $\delta\in\left(0,\frac{R}{2}\right)$, we can choose cut-off function $\varphi(x)$ satisfying $0\leq\varphi(x)\leq1$, $\varphi(x)=1$ in $B_{4R-\delta}\setminus B_{R+\delta}$, $\varphi(x)=0$ in $(\R^N\setminus B_{4R})\cup B_{R}$, $0<\varphi(x)<1$ in $(B_{R+\delta}\setminus B_{R})\cup (B_{4R}\setminus B_{4R-\delta})$ and $|\nabla\varphi|\leq\frac{C}{\delta}$. Let $$\psi_n(x):=\varphi^2(x)G(w_n),$$
	then
	$$\lr\psi_n(x)\rr_{x_i}=2\varphi\varphi_{x_i}G(w_n)+\varphi^2G'(w_n)\lr w_n\rr_{x_i}.$$	
We have
	\begin{align}
		&\quad a_{ij}^n(x)(w_n)_{x_j}(\psi_n)_{x_i}-|x|^{\al_n}\zeta_n(x)w_n\psi_n\label{423yy}\\
		&=a_{ij}^n(x)(w_n)_{x_j}\lr2\varphi\varphi_{x_i}G(w_n)+\varphi^2G'(w_n)\lr w_n\rr_{x_i}\rr-|x|^{\al_n}\zeta_n(x)w_n\varphi^2(x)G(w_n)\nonumber\\
		&\geq \varphi^2G'(w_n)a_{ij}^n(x)(w_n)_{x_j}(w_n)_{x_i}-2\varphi|\nabla\varphi||G(w_n)||(p-1)\eta_n(x)||\nabla w_n|\nonumber\\
		&\quad-|x|^{\al_n}\zeta_n(x)|w_n|\varphi^2(x)|G(w_n)|\nonumber\\
		&\geq C(p)\varphi^2|F'(\overline w_n)|^2|\eta_n(x)||\nabla w_n|^2-2\varphi|\nabla\varphi||F(\overline w_n)||F'(\overline w_n)||(p-1)\eta_n(x)||\nabla w_n|\nonumber\\
		&\quad-|x|^{\al_n}\zeta_n(x)|w_n|\varphi^2(x)|F(\overline w_n)||F'(\overline w_n)|\nonumber\\
		&= C(p)\varphi^2|\eta_n(x)||F'(\overline w_n)\nabla \overline w_n|^2-2(p-1)|\nabla\varphi F(\overline w_n)||\varphi F'(\overline w_n)\eta_n(x)\nabla \overline w_n|\nonumber\\
		&\quad-|x|^{\al_n}\zeta_n(x)|\varphi F(\overline w_n)||\varphi F'(\overline w_n)\overline w_n|\nonumber.
	\end{align}	
	Let $z_n=F(\overline w_n)$, and note that $\overline w_nF'(\overline w_n)\leq qF(\overline w_n)$. Since
	\begin{align*}
		\int_{B_{4R}\setminus B_{R}} a_{ij}^n(x)(w_n)_{x_j}(\psi_n)_{x_i}-|x|^{\al_n}\zeta_n(x)w_n\psi_n(x)\md x=0,
	\end{align*}	
	from \eqref{423yy}, we obtain	
	\begin{align*}
		&\quad \int_{B_{4R}\setminus B_{R}} |\eta_n(x)||\varphi\nabla z_n|^2\md x\\&\leq	C\int_{B_{4R}\setminus B_{R}}|\nabla\varphi \cdot z_n||\varphi\eta_n\cdot\nabla z_n|+C|x|^{\alpha_n}\zeta_n(x)|\varphi z_n|^2\md x\\
		&\leq  C\lr	\int_{B_{4R}\setminus B_{R}}|\eta_n(x)||\nabla\varphi|^2 | z_n|^2\md x\rr^{\frac12}\lr	\int_{B_{4R}\setminus B_{R}}|\eta_n(x)||\varphi|^2|\nabla z_n|^2\md x\rr^{\frac12}\\
		&\quad+C\int_{B_{4R}\setminus B_{R}}|x|^{\alpha_n}\zeta_n(x)|\varphi z_n|^2\md x.
	\end{align*}	
Thus, we have	
\begin{align*}
	&\quad \int_{B_{4R}\setminus B_{R}} |\eta_n(x)||\varphi\nabla z_n|^2\md x\\
	&\leq  C\int_{B_{4R}\setminus B_{R}}|\eta_n(x)||\nabla\varphi|^2 | z_n|^2\md x+C\int_{B_{4R}\setminus B_{R}}|x|^{\alpha_n}\zeta_n(x)|\varphi z_n|^2\md x.
\end{align*}	
	From Propositions \ref{pp41} and \ref{pro43}, \eqref{435}, \eqref{436} and \eqref{437}, we have
	\begin{align*}
\int_{B_{4R}\setminus B_{R}} |\eta_n(x)||\varphi\nabla z_n|^2\md x\geq \frac{C}{R^{\frac{(N-1)(p-2)}{p-1}}}\int_{B_{4R}\setminus B_{R}} |\varphi\nabla z_n|^2\md x
	\end{align*}		
and
\begin{align*}
	&\quad\int_{B_{4R}\setminus B_{R}}|\eta_n||\nabla\varphi|^2 | z_n|^2\md x+C\int_{B_{4R}\setminus B_{R}}|x|^{\alpha_n}\zeta_n(x)|\varphi z_n|^2\md x\\
	&\leq\frac{C}{R^{\frac{(N-1)(p-2)}{p-1}}}\int_{B_{4R}\setminus B_{R}}|\nabla\varphi|^2 | z_n|^2\md x+\frac{C}{R^{\frac{Np-2N+2p+\al_n}{p-1}}}\int_{B_{4R}\setminus B_{R}}|\varphi z_n|^2\md x.
\end{align*}		
	Furthermore, we get
	\begin{align}
		&\quad\int_{B_{4R}\setminus B_{R}} |\varphi\nabla z_n|^2\md x\label{424rr}\\
		&\leq C	\int_{B_{4R}\setminus B_{R}}|\nabla\varphi|^2 | z_n|^2\md x+\frac{C}{R^{\frac{3p+\al_n-2}{p-1}}}\int_{B_{4R}\setminus B_{R}}|\varphi z_n|^2\md x.\nonumber
	\end{align}	
	From the Sobolev embedding inequality, H\"older inequality and $|\varphi|\leq1$, we have
	\begin{align*}
		\lr\int_{B_{4R}\setminus B_{R}} |\varphi z_n|^{\frac{2N}{N-2}}\md x\rr^{\frac{N-2}{N}}&\leq C\int_{B_{4R}\setminus B_{R}} |\varphi\nabla z_n|^2\md x+C	\int_{B_{4R}\setminus B_{R}}|\nabla\varphi|^2 | z_n|^2\md x\\
		&\leq C\lr\frac{1}{\delta^2}+\frac{1}{R^{\frac{3p+\al_n-2}{p-1}}}\rr	\int_{B_{4R}\setminus B_{R}} | z_n|^2\md x.
	\end{align*}	
	Thus
	\begin{align*}
		\lr\int_{B_{4R-\delta}\setminus B_{R+\delta}} | z_n|^{\frac{2N}{N-2}}\md x\rr^{\frac{N-2}{N}}
		\leq C\lr\frac{1}{\delta^2}+\frac{1}{R^{\frac{3p+\al_n-2}{p-1}}}\rr	\int_{B_{4R}\setminus B_{R}} | z_n|^2\md x
	\end{align*}	
and
	\begin{align*}
		\lr\int_{B_{4R-\delta}\setminus B_{R+\delta}} | \overline w_n|^{q2^*}\md x\rr^{\frac{1}{q2^*}}\leq \lr\frac{C}{\delta^2}\rr	^{\frac{1}{2q}}\lr\int_{B_{4R}\setminus B_{R}} | \overline w_n|^{2q}\md x\rr^{\frac{1}{2q}}.
	\end{align*}	
	Let
	$$\gamma=2q, \,\,\,\, \chi=\frac{N}{N-2},\,\,\,\, \gamma_i=\chi^i\gamma,\,\,\,\, h_i=\sum_{j=0}^{i}\frac{R}{2^{j+1}} \,\,\,\, \text{for}\ i=0,1,2,\cdots,$$
	then, we get
	\begin{align*}
		&\quad\lr\int_{B_{4R-h_{i+1}}\setminus B_{R+h_{i+1}}} | \overline w_n|^{\gamma_{i+1}}\md x\rr^{\frac{1}{\gamma_{i+1}}}\\&\leq\lr\frac{1}{(h_{i+1}-h_{i})^2}\rr	^{\frac{1}{\gamma_{i}}}\lr\int_{B_{4R-h_{i}}\setminus B_{R+h_{i}}} | \overline w_n|^{\gamma_{i}}\md x\rr^{\frac{1}{\gamma_{i}}}\\
		&\leq \lr\frac{1}{(\frac{R}{2^{i+1}})^2}\rr	^{\frac{1}{\gamma_{i}}}\lr\int_{B_{4R-h_{i}}\setminus B_{R+h_{i}}} | \overline w_n|^{\gamma_{i}}\md x\rr^{\frac{1}{\gamma_{i}}}\\
		&\leq \frac{4^{\sum^{i}_{k=0}\frac{k+1}{\gamma_k}}}{R^{\sum^{i}_{k=0}\frac{2}{\gamma_k}}}\lr\int_{B_{\frac{7}{2}R}\setminus B_{\frac32R}} | \overline w_n|^{\gamma_{0}}\md x\rr^{\frac{1}{\gamma_{0}}}.
	\end{align*}	
	Taking $q=1$, and letting $i\to+\infty$, from \eqref{aaa427q}, we get	
	\begin{align}\label{425hh}
		|w_n(y)|\leq \|\overline w_n\|_{L^\infty_{(B_{3R}\setminus B_{2R})}}&\leq\frac{C}{R^{\frac{N}{2}}}\lr\int_{B_{\frac{7}{2}R}\setminus B_{\frac32R}} | \overline w_n|^{2}\md x\rr^{\frac{1}{2}}\\
		&\leq CR^{-\frac N 2+\frac12\frac{Np+p-2N}{p-1}}\nonumber\\
		&\leq CR^{-\frac{N-p}{2(p-1)}}\leq C_0|y|^{-\frac{N-p}{2(p-1)}}.\nonumber
	\end{align}	

\medskip	

$\bullet$ Case 2: $\frac{1}{2\var_n}\leq|y|<\frac{1}{\var_n}$.\\
In this case, noting $w_n=0$ on $x\in\partial B_{\frac{1}{\var_n}}$, by taking $\beta=\frac{p+p\al_n}{p-1}$ in Proposition \ref{pp43}, we get
\begin{align*}	
	\int_{B_{\frac{1}{\var_n}}}|x|^{\frac{p+p\al_n}{p-1}} U_{\al_n}^{p_{\al_n}^*-2}|w_n|^2\md x&\leq
	C.
\end{align*}
Therefore, taking $z_n:=\frac{1}{\var_n}\frac{y}{|y|}$ and $R=\frac{1}{8\var_n}\in\Big(\frac{|y|}{8},\frac{|y|}{4}\Big]$, we have
\begin{align}\label{qqq429}
	y\in  \left(B_{4R}(z_n)\cap B_{\frac{1}{\var_n}}\right)\subset B_{\frac{1}{\var_n}}
\end{align}
and
\begin{align}	\label{aaa427}
	\int_{B_{4R}(z_n)\cap B_{\frac{1}{\var_n}}}|w_n|^2\md x&\leq
	C\lr\frac{1}{\var_n}\rr^{\frac{Np+p-2N}{p-1}}.
\end{align}

Since $(B_{4R}(z_n)\cap B_{\frac{1}{\var_n}})\subset B_{\frac{1}{\var_n}}$, for any given $\delta\in\left(0,\frac{R}{2}\right)$, we can choose cut-off function $\varphi(x)$ satisfying $0\leq\varphi(x)\leq1$, $\varphi(x)=1$ in $B_{4R-\delta}(z_n)$, $\varphi(x)=0$ in $\R^N\setminus B_{4R}(z_n)$, $0<\varphi(x)<1$ in $B_{4R}(z_n)\setminus B_{4R-\delta}(z_n)$ and $|\nabla\varphi|\leq\frac{C}{\delta}$.
By \eqref{423yy}, we get
\begin{align*}
	&\quad \int_{B_{4R}(z_n)\cap B_{\frac{1}{\var_n}}} |\eta_n(x)|\varphi\nabla z_n|^2\md x\\&\leq	C\int_{B_{4R}(z_n)\cap B_{\frac{1}{\var_n}}}|\nabla\varphi \cdot z_n||\varphi\eta_n\cdot\nabla z_n|+C|x|^{\alpha_n}\zeta_n(x)|\varphi z_n|^2\md x\\
	&\leq  C\lr	\int_{B_{4R}(z_n)\cap B_{\frac{1}{\var_n}}}|\eta_n||\nabla\varphi|^2 | z_n|^2\md x\rr^{\frac12}\lr	\int_{B_{4R}(z_n)\cap B_{\frac{1}{\var_n}}}|\eta_n||\varphi|^2|\nabla z_n|^2\md x\rr^{\frac12}\\
	&\quad+C\int_{B_{4R}(z_n)\cap B_{\frac{1}{\var_n}}}|x|^{\alpha_n}\zeta_n(x)|\varphi z_n|^2\md x.
\end{align*}		
Then, we have	
\begin{align*}
	&\quad \int_{B_{4R}(z_n)\cap B_{\frac{1}{\var_n}}} |\eta_n(x)||\varphi\nabla z_n|^2\md x\\
	&\leq  C\int_{B_{4R}(z_n)\cap B_{\frac{1}{\var_n}}}|\eta_n(x)||\nabla\varphi|^2 | z_n|^2\md x+C\int_{B_{4R}(z_n)\cap B_{\frac{1}{\var_n}}}|x|^{\alpha_n}\zeta_n(x)|\varphi z_n|^2\md x.
\end{align*}	
From Proposition \ref{pp41}, \eqref{435}, \eqref{436} and \eqref{437}, we have
\begin{align*}
	\int_{B_{4R}(z_n)\cap B_{\frac{1}{\var_n}}} |\eta_n(x)||\varphi\nabla z_n|^2\md x\geq C\var_n^{\frac{(N-1)(p-2)}{p-1}}\int_{B_{4R}
		(z_n)\cap B_{\frac{1}{\var_n}}} |\varphi\nabla z_n|^2\md x
\end{align*}		
and
\begin{align*}
	&\quad\int_{B_{4R}(z_n)\cap B_{\frac{1}{\var_n}}}|\eta_n||\nabla\varphi|^2 | z_n|^2\md x+C\int_{B_{4R}(z_n)\cap B_{\frac{1}{\var_n}}}|x|^{\alpha_n}\zeta_n(x)|\varphi z_n|^2\md x\\
	&\leq C\var_n^{\frac{(N-1)(p-2)}{p-1}}\int_{B_{4R}(z_n)\cap B_{\frac{1}{\var_n}}}|\nabla\varphi|^2 | z_n|^2\md x+C\var_n^{\frac{Np-2N+2p+\al_n}{p-1}}\int_{B_{4R}(z_n)\cap B_{\frac{1}{\var_n}}}|\varphi z_n|^2\md x.
\end{align*}			
Thus we get
\begin{align}\label{424rrq}
	&\quad\int_{B_{4R}(z_n)\cap B_{\frac{1}{\var_n}}} |\varphi\nabla z_n|^2\md x\\
	&\leq C	\int_{B_{4R}(z_n)\cap B_{\frac{1}{\var_n}}}|\nabla\varphi|^2 | z_n|^2\md x+C\var_n^{\frac{3p+\al_n-2}{p-1}}\int_{B_{4R}(z_n)\cap B_{\frac{1}{\var_n}}}|\varphi z_n|^2\md x.\nonumber
\end{align}	
From the Sobolev embedding inequality, H\"older inequality and $|z_n|\leq1$, we get
\begin{align*}
	\lr\int_{B_{4R}(z_n)\cap B_{\frac{1}{\var_n}}} |\varphi z_n|^{\frac{2N}{N-2}}\md x\rr^{\frac{N-2}{N}}&\leq C\int_{B_{4R}(z_n)\cap B_{\frac{1}{\var_n}}} |\varphi\nabla z_n|^2\md x+C	\int_{B_{4R}(z_n)\cap B_{\frac{1}{\var_n}}}|\nabla\varphi|^2 | z_n|^2\md x\\
	&\leq C\lr\frac{1}{\delta^2}+\var_n^{\frac{3p+\al_n-2}{p-1}}\rr	\int_{B_{4R}(z_n)\cap B_{\frac{1}{\var_n}}} | z_n|^2\md x.
\end{align*}	
Thus
\begin{align*}
	\lr\int_{B_{4R-\delta}(z_n)\cap B_{\frac{1}{\var_n}}} | z_n|^{\frac{2N}{N-2}}\md x\rr^{\frac{N-2}{N}}
	\leq C\lr\frac{1}{\delta^2}+\var_n^{\frac{3p+\al_n-2}{p-1}}\rr	\int_{B_{4R}(z_n)\cap B_{\frac{1}{\var_n}}} | z_n|^2\md x.
\end{align*}	
and
\begin{align*}
	\lr\int_{B_{4R-\delta}(z_n)\cap B_{\frac{1}{\var_n}}} | \overline w_n|^{q2^*}\md x\rr^{\frac{1}{q2^*}}\leq \lr\frac{C}{\delta^2}\rr	^{\frac{1}{2q}}\lr\int_{B_{4R}(z_n)\cap B_{\frac{1}{\var_n}}} | \overline w_n|^{2q}\md x\rr^{\frac{1}{2q}}.
\end{align*}	
Let
$$\gamma=2q, \ \chi=\frac{N}{N-2},\ \gamma_i=\chi^i\gamma,\ h_i=\sum_{j=0}^{i}\frac{R}{2^{j+1}}, \qquad \text{for}\ i=0,1,2,\cdots,$$
then, we derive
\begin{align*}
	&\quad\lr\int_{B_{4R-h_{i+1}}(z_n)\cap B_{\frac{1}{\var_n}}} | \overline w_n|^{\gamma_{i+1}}\md x\rr^{\frac{1}{\gamma_{i+1}}}\\&\leq\lr\frac{1}{(h_{i+1}-h_{i})^2}\rr^{\frac{1}{\gamma_{i}}}\lr\int_{B_{4R-h_i}(z_n)\cap B_{\frac{1}{\var_n}}} | \overline w_n|^{\gamma_{i}}\md x\rr^{\frac{1}{\gamma_{i}}}\\
	&\leq \lr\frac{1}{(\frac{R}{2^{i+1}})^2}\rr	^{\frac{1}{\gamma_{i}}}\lr\int_{B_{4R-h_i}(z_n)\cap B_{\frac{1}{\var_n}}} | \overline w_n|^{\gamma_{i}}\md x\rr^{\frac{1}{\gamma_{i}}}\\
	&\leq \frac{4^{\sum^{i}_{k=0}\frac{k+1}{\gamma_k}}}{R^{\sum^{i}_{k=0}\frac{2}{\gamma_k}}}\lr\int_{B_{\frac72R}(z_n)\cap B_{\frac{1}{\var_n}}} | \overline w_n|^{\gamma_{0}}\md x\rr^{\frac{1}{\gamma_{0}}}.
\end{align*}	
Taking $q=1$, and letting $i\to+\infty$, from \eqref{aaa427}, we can conclude that
	\begin{align}\label{425hhh}
	|w_n(y)|\leq \|\overline w_n\|_{L^\infty(B_{3R}(z_n)\cap B_{\frac{1}{\var_n}})}&\leq\frac{C}{R^{\frac{N}{2}}}\lr\int_{B_{\frac{7}{2}R}(z_n)\cap B_{\frac{1}{\var_n}}} | \overline w_n|^{2}\md x\rr^{\frac{1}{2}}\\
	&\leq CR^{-\frac N 2}\var_n^{\frac12\frac{Np+p-2N}{p-1}}\nonumber\\
	&\leq C\var_n^{\frac{N-p}{2(p-1)}}\leq C_0|y|^{-\frac{N-p}{2(p-1)}}.\nonumber
\end{align}	
This concludes our proof of Proposition \ref{ppq43}.
\end{proof}

Finally, by considering the problem satisfied by $w_n$ directly, using the Pohozaev identity on uniform bounded domains $\Omega_n$ such that $\Omega_n\rightrightarrows B_r(0)$ as $n\rightarrow+\infty$ with $r=(p-1)^{\frac{p-1}{p+\al}}$ and the nontrivial fast decay estimate \eqref{412}, we successfully overcame the nonlinear virtue of the $p$-Laplacian $\Delta_p$ and the absence of the Green function representation formula, and proved the uniform lower bound $\|u_n-v_n\|_\infty\geq C_0$ in the following proposition, which indicates that the limit of approximate solutions is non-radial solution and plays a quite crucial role in our proof of global bifurcation result in Section 5.
\begin{prop}\label{pp44}
	Let $\al_n$ and $\var_n$ be sequences such that $\al_n\to\al>0$ and $\var_n\to0$ as $n\to+\infty$.
	Let $v_n$ be a sequence of nonradial solutions of \eqref{41} in $B_{\frac{1}{\var_n}}$ related to the exponent $\al_n$. If $\al\ne\al(k)$ for all $k\in\N$, then there exists a uniform constant $C>0$ independent of $n$ such that
\begin{align}\label{413}
		\|u_n-v_n\|_{L^{\infty}(\mathbb{R}^{N})}\geq C,
\end{align}
where $u_n$ is the radial solution of \eqref{41} related to the exponent $\al_n$ given by \eqref{32}.
\end{prop}
\begin{proof}
Suppose on the contrary that, there exists a sequence of nonradial solutions $v_n$ to \eqref{41} in $B_{\frac{1}{\var_n}}$ related to the exponent $\al_n$ such that
\begin{align}\label{p414}
		\|u_n-v_n\|_{L^{\infty}(\mathbb{R}^{N})}\to0, \quad\text{as}\ n\to+\infty.
	\end{align}
Let $w_n:=\frac{u_n-v_n}{\|u_n-v_n\|_{L^{\infty}(\mathbb{R}^{N})}}$, then $w_n$ satisfies the following equation in the weak sense
\begin{align}
&\int_{B_{\frac{1}{\var_n}}}\lr|\nabla u_n|^{p-2}\nabla w_n+\frac{p-2}{2}
		\int_0^1\lr (1-t)|\nabla u_n|^{2}+t|\nabla v_n|^{2}\rr^{\frac{p-4}{2}}\md t
		\lr\nabla w_n\cdot\nabla(u_n+v_n)\rr\nabla v_n\rr\nabla\varphi\md x\nonumber\\
&=\frac{1}{\|u_n-v_n\|_{L^{\infty}(\mathbb{R}^{N})}}\int_{B_{\frac{1}{\var_n}}}|x|^{\alpha_n}\lr \lr u_n+\beta_{\al_n}(n)\rr^{p_{\al_n}^*-1}-
\lr v_n+\beta_{\al_n}(n)\rr^{p_{\al_n}^*-1}\rr\varphi\md x,\quad\forall\varphi\in C_c^\infty(\R^N).\nonumber
\end{align}
Since $\|w_n\|_{L^{\infty}(\mathbb{R}^{N})}=1$, through the standard regularity theorem, we can deduce from \eqref{36} and \eqref{p414} that $w_n\to w$ in $C_{loc}^{1,\eta}(\R^N)$, and $w$ satisfies the equation
\begin{align}\label{425p}
&\quad\int_{\R^N}\lr|\nabla U_\al|^{p-2}\nabla w+(p-2)
		|\nabla U_\al|^{p-4}
		\lr\nabla w\cdot\nabla U_\al\rr\nabla U_\al\rr\nabla\varphi\md x\\
&=(p_{\al}^*-1)\int_{\R^N}|x|^{\alpha}U_\al^{p_{\al}^*-2}w\varphi\md x,\quad\forall\varphi\in C_c^\infty(\R^N).\nonumber
\end{align}
Since $\al\ne\al(k)$, by Theorem \ref{th11}, we obtain
\begin{align}\label{425}
w(x)=A\frac{(p-1)-|x|^{\frac{p+\al}{p-1}}}
{(1+|x|^{\frac{p+\al}{p-1}})^{\frac{N+\al}{p+\al}}}  \qquad\text{for\ some}\ A\in\R.
\end{align}

If $A\ne0$, from \eqref{425}, there exists a radius $r=(p-1)^{\frac{p-1}{p+\al}}$ and domain $\Omega_n\subset B_{\frac{1}{\var_n}}$ such that $u_n-v_n=0$ on $\partial\Omega_n$ and
\begin{align}\label{414ppp}
\Omega_n\to B_r,\quad n\to+\infty,
\end{align}
i.e., for any $\varepsilon>0$, there exists $N_0\geq1$ such that $B_{r-\varepsilon}\subset\Omega_n\subset B_{r+\varepsilon}$ for any $n\geq N_0$. From the Pohozaev identity, multiplying the equation \eqref{41} by $v_n+\beta_{\al_n}(n)$, and integrating on $\Omega_n$, we have
\begin{align}\label{414}
&\int_{\Omega_n}|\nabla v_n|^p\md x-\int_{\partial \Omega_n}|\nabla v_n|^{p-2}(\nabla v_n\cdot\nu)(v_n+\beta_{\al_n}(n))\md \mms=\int_{\Omega_n}
|x|^{\alpha_n}\lr v_n+\beta_{\al_n}(n)\rr^{p_{\al_n}^*}\md x.
	\end{align}
Multiplying the equation \eqref{41} by $x\cdot\nabla(v_n+\beta_{\al_n}(n))$ and integrating on $\Omega_n$, we obtain
\begin{align}\label{415}
&\quad\frac{N-p}{p}\int_{\Omega_n}|\nabla v_n|^p\md x-\frac{1}{p}\int_{\partial \Omega_n}|\nabla v_n|^{p}(x\cdot\nu)\md \mms+\int_{\partial \Omega_n}|\nabla v_n|^{p-2}(\nabla v_n\cdot x)(\nabla v_n\cdot\nu)\md \mms\\
&=\frac{N+\al_n}{p^*_{\al_n}}\int_{\Omega_n}|x|^{\alpha_n}\lr v_n+\beta_{\al_n}(n)\rr^{p_{\al_n}^*}\md x-\frac{1}{p^*_{\al_n}}\int_{\partial\Omega_n}|x|^{\alpha_n}\lr v_n+\beta_{\al_n}(n)\rr^{p_{\al_n}^*}(x\cdot\nu)\md \mms,\nonumber
	\end{align}
where $\nu$ represents the unit outer normal vector on $\partial\Omega_n$.

By subtracting \eqref{414}$\times\frac{N-p}{p}$ from \eqref{415}, we deduce from \eqref{41} that
\begin{align}\label{416}
&\quad\int_{\partial \Omega_n}|\nabla v_n|^{p-2}(\nabla v_n\cdot x)(\nabla v_n\cdot\nu)\md \mms+\frac{N-p}{p}\int_{\partial \Omega_n}|\nabla v_n|^{p-2}(\nabla v_n\cdot\nu)(v_n+\beta_{\al_n}(n))\md \mms\\
&=\frac{1}{p}\int_{\partial \Omega_n}|\nabla v_n|^{p}(x\cdot\nu)\md S-\frac{1}{p^*_{\al_n}}\int_{\partial\Omega_n}|x|^{\alpha_n}\lr v_n+\beta_{\al_n}(n)\rr^{p_{\al_n}^*}(x\cdot\nu)\md \mms\nonumber.
	\end{align}
Similarly, we can also conclude that $u_n$ satisfies the identity
\begin{align}\label{417}
&\quad\int_{\partial \Omega_n}|\nabla u_n|^{p-2}(\nabla u_n\cdot x)(\nabla u_n\cdot\nu)\md \mms+\frac{N-p}{p}\int_{\partial \Omega_n}|\nabla u_n|^{p-2}(\nabla u_n\cdot \nu)(u_n+\beta_{\al_n}(n))\md \mms\\
&=\frac{1}{p}\int_{\partial \Omega_n}|\nabla u_n|^{p}(x\cdot\nu)\md \mms-\frac{1}{p^*_{\al_n}}\int_{\partial\Omega_n}|x|^{\alpha_n}\lr u_n+\beta_{\al_n}(n)\rr^{p_{\al_n}^*}(x\cdot\nu)\md \mms\nonumber.
	\end{align}
Subtracting \eqref{416} with \eqref{417}, we obtain
\begin{align}\label{418}
&\quad\int_{\partial \Omega_n}|\nabla u_n|^{p-2}(\nabla u_n\cdot x)(\nabla u_n\cdot\nu)-|\nabla v_n|^{p-2}(\nabla v_n\cdot x)(\nabla v_n\cdot\nu)\md \mms\\
&+\frac{N-p}{p}\int_{\partial \Omega_n}\lr|\nabla u_n|^{p-2}(\nabla u_n\cdot \nu)(u_n+\beta_{\al_n}(n))-|\nabla v_n|^{p-2}(\nabla v_n\cdot \nu)(v_n+\beta_{\al_n}(n))\rr\md \mms\nonumber\\
&=\frac{1}{p}\int_{\partial \Omega_n}\lr|\nabla u_n|^{p}-|\nabla v_n|^{p}\rr(x\cdot\nu)\md \mms\nonumber.
	\end{align}
By the mean value theorem, we have
\begin{align}\label{420}
&\quad|\nabla u_n|^{p-2}(\nabla u_n\cdot x)(\nabla u_n\cdot\nu)-|\nabla v_n|^{p-2}(\nabla v_n\cdot x)(\nabla v_n\cdot\nu)\\
&=\frac{p-2}{2}\int_{0}^1\lr t|\nabla u_n|^{2}+(1-t)|\nabla v_n|^{2}\rr^{\frac{p-4}{2}}\md t(\nabla u_n-\nabla v_n)(\nabla u_n+\nabla v_n)(\nabla u_n\cdot x)(\nabla u_n\cdot\nu)\nonumber\\
&\quad+|\nabla v_n|^{p-2}((\nabla u_n-\nabla v_n)\cdot\nu)(\nabla u_n\cdot x)+|\nabla v_n|^{p-2}((\nabla u_n-\nabla v_n)\cdot x)(\nabla v_n\cdot \nu)\nonumber,
	\end{align}
\begin{align}\label{421}
&\quad|\nabla u_n|^{p-2}(\nabla u_n\cdot \nu)(u_n+\beta_{\al_n}(n))-|\nabla v_n|^{p-2}(\nabla v_n\cdot \nu)(v_n+\beta_{\al_n}(n))\\
&=\frac{p-2}{2}\int_{0}^1\lr (1-t)|\nabla u_n|^{2}+t|\nabla v_n|^{2}\rr^{\frac{p-4}{2}}\md t(\nabla u_n-\nabla v_n)(\nabla u_n+\nabla v_n)\nonumber\\
&\quad\times(\nabla u_n\cdot \nu)(u_n+\beta_{\al_n}(n))+|\nabla v_n|^{p-2}((\nabla u_n-\nabla v_n)\cdot\nu)(u_n+\beta_{\al_n}(n))\nonumber,	\end{align}
and
\begin{align}\label{422}
		&\quad\lr|\nabla u_n|^{p}-|\nabla v_n|^{p}\rr(x\cdot\nu)\\
&=\frac{p}{2}\int_{0}^1\lr (1-t)|\nabla u_n|^{2}+t|\nabla v_n|^{2}\rr^{\frac{p-2}{2}}\md t(\nabla u_n-\nabla v_n)(\nabla u_n+\nabla v_n)(x\cdot\nu)\nonumber.
	\end{align}
Combining \eqref{418}, \eqref{420}, \eqref{421} and \eqref{422}, we have
\begin{align}\label{423}
&\quad\frac{p-2}{2}\int_{\partial\Omega_n}\int_{0}^1\lr (1-t)|\nabla u_n|^{2}+t|\nabla v_n|^{2}\rr^{\frac{p-4}{2}}\md t(\nabla u_n-\nabla v_n)(\nabla u_n+\nabla v_n)\\
&\quad\times(\nabla u_n\cdot x)(\nabla u_n\cdot\nu)\md S
+\int_{\partial\Omega_n}|\nabla v_n|^{p-2}((\nabla u_n-\nabla v_n)\cdot\nu)(\nabla u_n\cdot x)\md \mms\nonumber\\
&+\int_{\partial\Omega_n}|\nabla v_n|^{p-2}((\nabla u_n-\nabla v_n)\cdot x)(\nabla v_n\cdot \nu)\md \mms\nonumber\\
&+\frac{N-p}{p}\frac{p-2}{2}\int_{\partial\Omega_n}\int_{0}^1\lr t|\nabla u_n|^{2}+(1-t)|\nabla v_n|^{2}\rr^{\frac{p-4}{2}}\md t(\nabla u_n-\nabla v_n)(\nabla u_n+\nabla v_n)\nonumber\\
&\quad\times(\nabla u_n\cdot \nu)(u_n+\beta_{\al_n}(n))\md S+\frac{N-p}{p}\int_{\partial\Omega_n}|\nabla v_n|^{p-2}((\nabla u_n-\nabla v_n)\cdot\nu)(u_n+\beta_{\al_n}(n))\md \mms\nonumber\\
&-\frac{1}{2}\int_{\partial\Omega_n}\int_{0}^1\lr (1-t)|\nabla u_n|^{2}+t|\nabla v_n|^{2}\rr^{\frac{p-2}{2}}\md t(\nabla u_n-\nabla v_n)(\nabla u_n+\nabla v_n)(x\cdot\nu)\md \mms=0.\nonumber
	\end{align}

Using the decay properties of $u_n$, $v_n$, $w_n$ and Lebesgue's dominated convergence theorem, taking the limit $n\to\infty$ in \eqref{423}, we get
\begin{align}\label{426}
		&\quad(p-2)\int_{\partial B_r}|\nabla U_\al|^{p-4}(\nabla w\cdot \nabla U_\al)(\nabla U_\al\cdot x)(\nabla U_\al\cdot\nu)\md \mms\\
		&\quad+2\int_{\partial B_r}|\nabla U_\al|^{p-2}(\nabla w\cdot\nu)(\nabla U_\al\cdot x)\md \mms\nonumber\\
&+\frac{(N-p)(p-2)}{p}\int_{\partial B_r}|\nabla U_\al|^{p-4}(\nabla w\cdot\nabla U_\al)(\nabla U_\al\cdot \nu)U_\al\md \mms\nonumber\\
&+\frac{N-p}{p}\int_{\partial B_r}|\nabla U_\al|^{p-2}(\nabla w\cdot\nu)U_\al\md \mms-\int_{\partial B_r}|\nabla U_\al|^{p-2}(\nabla w\cdot \nabla U_\al)(x\cdot\nu)\md \mms=0.\nonumber
	\end{align}
Plugging \eqref{425} into \eqref{426}, we get
$$A=0.$$
Therefore, we have
\begin{align}\label{p440}w_n\to w\equiv 0\quad \text{in}\ C^{1,\eta}_{loc}(\R^N).	\end{align}
Let $x_n\in B_{\frac{1}{\var_n}}$ be such that $|w_n(x_n)|=1=\|w_n\|_{L^{\infty}(\mathbb{R}^{N})}$. From the uniform decay estimate in Proposition \ref{ppq43}, we know the sequence $\{x_n\}$ is bounded, thus there exists some point $x_0\in\mathbb{R}^{N}$ such that $x_n\to x_0$ (up to subsequence) and hence $w(x_0)=1$, which contradicts \eqref{p440}.
Thus we have proved \eqref{413} and hence concluded our proof of Proposition \ref{pp44}.
\end{proof}

\bigskip

\section{The bifurcation result: completion of our proof of the main Theorem \ref{th14}}

In this section, we will complete our proof of Theorem \ref{th14}. For this purpose, let \(\{\varepsilon_n\}\) be a sequence such that \(\varepsilon_n\rightarrow0\), there exists a sequence of nondegenerate radial solutions \(u_{n,\alpha}\) of the equation \eqref{31} (with \(\varepsilon = \varepsilon_n\)) which converges to \(U_{\alpha}\) as \(n\rightarrow+\infty\).

In Section 3, for all $k\in\mathbb{N}$, we have proved that \((\alpha_{k}^n,u_{n,\alpha_{k}^n})\) are nonradial bifurcation points, which generate continua \(\mathcal{C}(\alpha_{k}^n)\) in the space \((0,+\infty)\times\mathcal{Z}_n\), where \(\alpha_{k}^n\) is the unique root of equation \eqref{330}. Moreover, when \(k\) is an even integer, these continua also exist in the space \((0,+\infty)\times\mathcal{Z}^l_n\), where \(\mathcal{Z}^l_n\) is defined in the proof of Theorem \ref{thm39}. These continua \(\mathcal{C}(\alpha_{k}^n)\) are global and satisfy the well-known Rabinowitz alternative Theorem (see Theorem \ref{th38}).

In addition, according to Corollary \ref{c12}, the Morse exponent of $U_\al$ changes when $\al$ crosses $\al(k)$, with $\al(k)=\frac{p\sqrt{(N+p-2)^2+4(k-1)(p-1)(k+N-1)}-p(N+p-2)}{2(p-1)}$ and all the eigenfunctions associated with the linearized problem (i.e., the solutions of \eqref{22}) are contained within the space $X$, which is defined in \eqref{19}. Define the space
\begin{align}\label{51}
\mathcal{Z}:=\Big\{&h\in X \ \text{such\  that}\ h(x_1,\cdots,x_N)=h\lr g(x_1,\cdots,x_{N-1}),x_N\rr \\ &\text{for\ any}\ g\in\mathcal{O}(N-1)\Big\}\nonumber.
\end{align}
By extending the function by zero outside of $B_{\frac{1}{\var_n}}$, through the regularity theorems, we can deduce that $\mathcal{C}(\al_k^n)$ belongs to the space
$$\mathcal{H}:=(0,+\infty)\times\mathcal{Z}.$$
Furthermore, according to Proposition \ref{pp32}, we know that as \(n \to +\infty\), \(u_{n,\alpha_{k}^n}\) converges to \(U_{\alpha(k)}\) within the space \(\mathcal{Z}\), which is a subset of \(X\). By \cite{SW1}, we know that, if we restrict the Morse exponent of $U_\al$ in the space $\mathcal{Z}$, then it will increase by $1$ as $\alpha$ cross $\al(k)$, i.e.,
\begin{align*}
m(\al(k)+\delta)-m(\al(k)-\delta)=1,
\end{align*}
where $m$ is the Morse index of $U_\al$ in the space $\mathcal{Z}$. We aim to utilize this alteration in the Morse exponent of \(U_{\alpha}\) within the space \(\mathcal{Z}\) and prove that, as \(n\to+\infty\), these continua \(\mathcal{C}(\alpha_{k}^n)\) converge in an appropriate sense to continua of nonradial solutions of \eqref{11} which bifurcate from \((\alpha(k),U_{\alpha(k)})\) in the product space \(\mathcal{H}=(0,+\infty)\times\mathcal{Z}\). To achieve this, we draw on certain ideas previously employed in \cite{AG}, and also refer to \cite{GP}.

In order to establish the bifurcation result, we need the following topological result (refer to Theorem 9.1 in \cite{W}).
\begin{lem}[\textbf{\cite{W}}]\label{th51}
Let $X_n$ be a sequence of connected subsets of a metric space $X$. Let $\liminf(X_n)$ $(\limsup(X_n)$, resp.$)$ denote the set of all $x\in X$ such that any neighborhood of $x$ intersects all except finitely many of $X_n$ $($infinitely many of $X_n$, resp.$)$. If\\
(i) $\liminf(X_n)\ne\emptyset$, \\
(ii) $\bigcup X_n$ is precompact, \\
then $\limsup(X_n)$ is nonempty, compact and connected.
\end{lem}

Let \(n\) be sufficiently large and let \(\alpha_{k}^{n}\) be given by \eqref{330}, so that \((\alpha_{k}^{n},u_{n,\alpha_{k}^{n}})\) is a bifurcation point for problem \eqref{31}. Denote by \(\mathcal{C}(\alpha_{k}^{n})\) the maximal connected component that bifurcates from \((\alpha_{k}^{n},u_{n,\alpha_{k}^{n}})\) in the space \(\mathcal{H}\). Fix \(\delta>0\) such that, in the interval \([\alpha_{k}-\delta,\alpha_{k}+\delta]\), there is no other exponent \(\alpha_{j}\) with \(j\neq k\). Define
$$\mathcal{H}_n:=\mathcal{C}(\al_k^n)\bigcap B_{\delta,\mathcal{H}}(\al_k^n,u_{n,\al_k^n}).$$
where
$$B_{\delta,\mathcal{H}}(\al_k^n,u_{n,\al_k^n}):=\{(\al,h)\in \mathcal{H}\
\text{such\ that}\ |\al-\al_k^n|+\|h-u_{n,\al_k^n}\|_X<\delta\},$$
where the space $X$ and its norm are defined by \eqref{19} and \eqref{110}. Let $\mathcal{M}^n_k$ be the
maximal connected component of $\mathcal{H}_n$ that contains $(\al_k^n,u_{n,\al_k^n})$. Then, $\mathcal{M}^n_k\neq\emptyset$ and $(\al(k),U_{\al(k)})\in\liminf\limits_{n}(\mathcal{M}^n_k)$.

\begin{lem}\label{th52}
For every $k\in\mathbb{N}$, the set $\bigcup\limits_{n}\mathcal{M}^n_k$ is precompact in $\mathcal{H}$.
\end{lem}
\begin{proof}
Let $(\al^{(k)}_j,h^{(k)}_j)$ be a sequence in $\bigcup\mathcal{M}^n_k$, then $(\al^{(k)}_j,h^{(k)}_j)\in\mathcal{M}_k^{n(j)}$ for some $n(j)\geq1$. First, let us consider the case: $n(j)\to+\infty$, as $j\to+\infty$.

From the definition of $\mathcal{M}_k^{n(j)}$, the functions $h_j^{(k)}$ satisfy equation \eqref{31} with $\var = \var_{n(j)}$ and $\al = \al_j^{(k)}$. Since $|\al_j^{(k)}-\al^{n(j)}_{k}|+|\al^{n(j)}_{k}-\al(k)|<2\delta$ for $n(j)$ large enough, we have $\al_j^{(k)}\in(\al(k) - 2\delta,\al(k) + 2\delta)$ for $j$ large enough. Consequently, by passing to a subsequence if necessary, we have $\al_j^{(k)}\to\hat\al$, where $\hat\al\in[\al(k) - 2\delta,\al(k) + 2\delta]$. In addition, since $h_j^{(k)}\in\mathcal{Z}$ and $\|h_j^{(k)}-u_{n(j),\al_k^{n(j)}}\|_X<\delta$, it follows that $\|h_j^{(k)}\|_X<\delta+\sup \|u_{n(j),\al_k^{n(j)}}\|_X$. From Proposition \ref{pp32}, we get $\|u_{n(j),\al_k^{n(j)}}\|_X\leq C$. As a result, there exists a constant $C>0$ such that $\|h_{j}^{(k)}\|_\gamma\leq C$ and $\|h_{j}^{(k)}\|_{1,p}\leq C$ for all $j$. Then, by passing to a subsequence, we can assert that $h_j^{(k)}\to\hat h$ weakly in $D^{1,p}(\mathbb{R}^N)$ and almost everywhere in $\mathbb{R}^N$. Moreover, using equation \eqref{31}, we deduce that $h^{(k)}_j\to\hat h$ in $C^{1}_{loc}(\mathbb{R}^N)$, where $\hat h$ is a solution of equation \eqref{11} with the exponent $\al=\hat\al$.

Moreover, according to Proposition \ref{pp41}, there exists $C>0$ such that
\begin{align}\label{52}
		h_j^{(k)}(x)\leq\frac{C}{(1+|x|)^{\frac{N-p}{p-1}}}
\ \ \text{for\ every}\ x\in\R^N\ \text{and\ for\ every}\ j\in\N.
	\end{align}
Then we have
\begin{align}\label{53}
		|h_j^{(k)}(x)-\hat h(x)| \leq |h_j^{(k)}(x)|+|\hat h(x)|\leq\frac{C}{(1+|x|)^{\frac{N-p}{p-1}}},
	\end{align}
which implies that for every $\var>0$, there exists $r>0$ such that, for any $j>0$, $(1+|x|)^\gamma|h_j^{(k)}(x)-\hat h(x)|<\var$ if $|x|>r$. By the uniform convergence of $h_j^{(k)}$ to $\hat h$ on compact sets of $\R^N$, we obtain
$(1+|x|)^\gamma|h_j^{(k)}(x)-\hat h(x)|<\var$ in $B_r(0)$ if $j$ is large enough, i.e., $h_j^{(k)}\to\hat h$ in $L^\infty_{\gamma}(\R^N)$.

Furthermore, from \eqref{52}, we get
\begin{align}\label{54}
\int_{\R^N}|\nabla h_j^{(k)}|^p\md x=\int_{\R^N}|x|^{\al^{(k)}_j}\lr h_j^{(k)}+\beta_{\al^{(k)}_j}(n)\rr^{p_{\al^{(k)}_j}^*-1}h_j^{(k)}\md x.
	\end{align}
By \eqref{44}, and taking the limit in \eqref{54}, we have
\begin{align}\label{55}
\int_{\R^N}|\nabla h_j^{(k)}|^p\md x\to\int_{\R^N}|x|^{\hat\alpha}\hat h^{p_{\hat\al}^*-1}\hat h\md x=\int_{\R^N}|\nabla \hat h|^p\md x.
	\end{align}
Thus
$\int_{\R^N}|\nabla (h_j^{(k)}-\hat h)|^p\md x\to0$. Then $h_j^{(k)}\to\hat h$ strongly in $X$.

\smallskip

Next, we consider the other case: $n(j)\not\to+\infty$, as $j\rightarrow+\infty$. Then, up to a subsequence, we may assume that $n(j)$ converges to $n_0\in\N$. Repeating the proof for the case $n(j)\to+\infty$, we can obtain a subsequence $h_{j}^{(k)}$ that converges in $X$ to a solution of \eqref{31} with $\var=\var_{n_0}$ and the exponent $\al=\hat\al$. This finishes our proof of Lemma \ref{th52}.
\end{proof}

\begin{lem}\label{th53}
The set $\limsup\limits_{n}(\bigcup\mathcal{M}^n_k)\backslash\{(\al(k),U_{\al(k)})\}$ is nonempty.
\end{lem}
\begin{proof}
From Theorem \ref{th38} and regularity of the global continuum $\mathcal{C}(\al_k^n)$ bifurcates from the point $(\al_k^n,u_{n,\al_k^n})$, we infer that, either this continuum $\mathcal{C}(\al_k^n)$ is unbounded in $\mathcal{H}$, or there exists another bifurcation point $(\al_i^n,u_{n,\al_i^n})\in\mathcal{C}(\al_k^n)$ with $\al_i^n\not\in[\al(k)-\delta,\al(k)+\delta]$, or $\mathcal{C}(\al_k^n)$ meets $\{0\}\times \mathcal{Z}$. Thus, on the closure of any component
$\mathcal{M}^n_k$, there exists a point $(\hat \al_{k}^n,\hat h_{n,\hat \al_{k}^n})\in\partial B_{\delta,\mathcal{H}}(\al_k^n,u_{n,\al_k^n})$, i.e.,
\begin{align}\label{56}
|\hat \al_{k}^n-\al_k^n|+\|\hat h_{n,\hat \al_{k}^n}-u_{n,\al_k^{n}}\|_X=\delta,
	\end{align}
where $\hat h_{n,\hat \al_{k}^n}$ is a solution of \eqref{31} in $B_{\frac{1}{\var_n}}$ for the exponent $\al=\hat \al_{k}^n$.

From the standard regularity theorems and the boundedness of $\hat\al_{k}^n$ and $\hat h_{n,\hat \al_{k}^n}$, up to a subsequence, we get
$(\hat \al_{k}^n,\hat h_{n,\hat \al_{k}^n})\to(\hat\al_k,\hat h_k)$, where $\hat h_k$ is the solution of \eqref{11} for the component $\al=\hat \al_k,\hat \al_k\in[\al(k)-2\delta,\al(k)+2\delta]$ and satisfies
\begin{align*}
|\hat \al_k-\al(k)|+\|\hat h_k-U_{\al(k)}\|_X=\delta>0.
	\end{align*}
Note that $(\hat\al_k,\hat h_k)\in\limsup\limits_{n}(\mathcal{M}_k^n)$ but $(\hat\al_k,\hat h_k)\ne(\al(k),U_{\al(k)})$. This finishes our proof of Lemma \ref{th53}.
\end{proof}

Now we define the curve
\begin{align}\label{57}
\Gamma=\Big\{&(\al,U_\al)\in(0,+\infty)\times X \ \text{such\  that}\ U_\al\ \text{is\ the} \\ &\text{unique\ radial\ solution\ of}\ \eqref{11}\ such\ that\ U_\al(0)=1\Big\}\nonumber.
\end{align}
\begin{thm}\label{th54}
For any $k\geq2$, the points $(\al(k),U_{\al(k)})$ are nonradial bifurcation points for the curve
$\Gamma$.
\end{thm}
\begin{proof}
For $k\geq2$, we consider the bifurcation points $(\al_k^n,u_{n,\al_k^n})$ in $B_{\frac{1}{\var_n}}$ for problem \eqref{31}, as well as the connected components $\mathcal{M}^n_k$ of the bifurcation continua in $B_{\delta,\mathcal{H}}(\al_k^n,u_{n,\al_k^n})$.

From the Proposition \ref{th52}, the sequence of the sets $\mathcal{M}_k^n$ satisfies the hypotheses of Lemma \ref{th51} in the space $\mathcal{H}$, thus we have
$$\mathcal{C}_k:=\limsup\limits_{n}(\mathcal{M}^n_k)$$
is nonempty, compact and connected. Furthermore, $\mathcal{C}_k$ contains $(\al(k),U_{\al(k)})$, and Lemma \ref{th53} implies $\mathcal{C}_k\setminus\{(\al(k),U_{\al(k)})\}\neq\emptyset$ .

If $(\hat\al_k,\hat h_k)\in\mathcal{C}_k\backslash\{(\al(k),U_{\al(k)})\}$, then there exists a sequence of points $(\hat\al_k^n,\hat h_k^n)\in\mathcal{M}^n_k$ such that $(\hat\al_k^n,\hat h_k^n)\to(\hat\al_k,\hat h_k)$ in $\mathcal{H}$, $\hat h_k$ is a solution of \eqref{11} with the exponent $\al=\hat\al_k$, and $\hat h_{k}>0$ since $(\hat\al_k^n,\hat h_k^n)\in B_{\delta,\mathcal{H}}(\al_k^n,u_{n,\al_k^n})$ for $\delta$ small. We aim to prove $\hat h_k\ne U_{\la,\hat\al_k}$ for any $\la>0$.

First, by Proposition \ref{pp42}, we obtain $\hat h_k\ne U_{\la,\hat\al_k}$ for any $\la\ne1$. In addition, if $\hat\al_k=\al(k)$, then $\hat h_k\ne U_{\hat\al_k}$, since $(\hat\al_k,\hat h_k)\in \mathcal{C}_k\backslash\{(\al(k),U_{\al(k)})\}$.

We only need to consider the case $\la=1$ and $\hat\al_k\ne\al(k)$, and prove $\hat h_k\ne U_{\hat\al_k}$ by showing that
$$\|\hat h_k^n-U_{\hat\al_k}\|_X>C>0$$
for any $n$ large enough and for some uniform positive constant $C$ independent of $n$. From Proposition \ref{pp32}, we know that $u_{n,\hat\al_k^n}\to U_{\hat\al_k}$ in $X$, as $n\rightarrow+\infty$. Thus it is equivalent to prove that
\begin{align}\label{58}\|\hat h_k^n- u_{n,\hat\al_k^n}\|_X>C>0
\end{align}
for any $n$ large enough and for some uniform positive constant $C$ independent of $n$. Indeed, \eqref{58} follows from Proposition \ref{pp44} immediately.

From the uniqueness of radial solutions to \eqref{11} in Theorem \ref{thm11}, we deduce that $\hat h_k$ is a nonradial solution of \eqref{11} with the exponent $\al=\hat\al_k$. This completes our proof of Theorem \ref{th54}.
\end{proof}
	
\begin{rem}\label{rem55}
The bifurcation occurring at the points $(\alpha(k), U_{\alpha(k)})$ in Theorem \ref{th54} is, in fact, global. Specifically, we have derived the existence of a closed and connected set $\mathcal{C}_k$ that branches off from each point $(\alpha(k), U_{\alpha(k)})$.
\end{rem}

\noindent {\bf Proof of Theorem \ref{th14}.} Theorem \ref{th54} establishes the existence of a continuum $\mathcal{C}_k$ of nonradial solutions to equation \eqref{11}. These solutions are invariant under the action of $\mathcal{O}(N - 1)$ and bifurcate from the points $(\alpha(k), U_{\alpha(k)})$, where $\al(k)=\frac{p\sqrt{(N+p-2)^2+4(k-1)(p-1)(k+N-1)}-p(N+p-2)}{2(p-1)}$ for all $k \geq 2$. Therefore, (i) in Theorem \ref{th14} has been proved.

Furthermore, if $k$ is an even integer, we can replicate the proof of Theorem \ref{th54}, by utilizing the space $\mathcal{Z}^l$, which is defined by
$$\mathcal{Z}^l:=\{h\in X\ \text{s.t.\ $h$\ is\ invariant\ by\ the\ action\ of}\ \mathcal{G}_l \}$$
for $l=1,2,\cdots,\left[\frac{N}{2}\right]$, where $\mathcal{G}_l$ is defined by \eqref{347qq}. From Theorem \ref{thm39} and Remark \ref{rem312}, we are able to obtain $\left[\frac{N}{2}\right]$ distinct continua of nonradial solutions that branch off from the point $(\al(k),U_{\al(k)})$. The $l$-th continua is invariant under the action of $\mathcal{G}_l$. Consequently, by \eqref{347qq}, (ii) in Theorem \ref{th14} has been derived.

Lastly, the precise decay property of all these nonradial solutions $h$ derived in (i) and (ii) follows from Theorem 1.3 in \cite{DLL} (see \eqref{eq0806} and \eqref{eq0806+}), i.e., $h\sim |x|^{-\frac{N-p}{p-1}}$ and $|\nabla h|\sim |x|^{-\frac{N-1}{p-1}}$, as $|x|\rightarrow+\infty$.
\qed

\end{document}